\documentclass[oneside,11pt]{amsart}

\usepackage{amsmath}
\usepackage{amsthm}
\usepackage{yhmath}

\usepackage[small]{caption}

\usepackage{amsfonts}
\usepackage{amssymb}
\usepackage{amsxtra}
\newtheorem{thm}{Theorem}[section]
\newtheorem{prop}[thm]{Proposition}
\newtheorem{lem}[thm]{Lemma}
\newtheorem{cor}[thm]{Corollary}

\newtheorem*{thm*}{Theorem}

\theoremstyle{definition}
\newtheorem{defn}[thm]{Definition}
\newtheorem{rem}[thm]{Remark}

\newtheorem{exmp}[thm]{Example}

\newtheorem{prob}[thm]{Problem}
\newtheorem{const}[thm]{Construction}

\newcommand{\abs}[1]{\lvert{#1}\rvert}

\renewcommand{\bar}[1]{\overline{#1}}
\newcommand{\boundary}{\partial}

\newcommand{\set}[2]{\{\,{#1} \mid {#2} \,\}}
\newcommand{\bigset}[2]{ \bigl\{ \, {#1} \bigm| {#2} \, \bigr\} }

\renewcommand{\emptyset}{\varnothing}

\newcommand{\dist}{\textup{\textsf{d}}}

\newcommand{\field}[1]{\mathbb{#1}}
\newcommand{\Z}{\field{Z}}

\newcommand{\R}{\field{R}}

\newcommand{\E}{\field{E}}
\newcommand{\Hyp}{\field{H}}
\renewcommand{\P}{\field{P}}

\DeclareMathOperator{\CAT}{CAT}
\DeclareMathOperator{\Lk}{Lk}

\DeclareMathOperator{\Stab}{Stab}

\newcommand{\ball}[2]{B ( {#1}, {#2} )}%Ball{center}{radius}

\newcommand{\nbd}[2]{\mathcal{N}_{#2}({#1})}  % Neighborhood{center}{radius}
\newcommand{\bignbd}[2]{\mathcal{N}_{#2} \bigl( {#1} \bigr)}
\newcommand{\Set}[1]{\mathcal{#1}}

\newcommand{\W}{{\mathcal W}}

\DeclareMathOperator{\Cayley}{Cayley}
\DeclareMathOperator{\diam}{diam}
\DeclareMathOperator{\Sat}{Sat} %Saturation

\newcommand{\neb}{\mathcal N}
\newcommand{\stabilizer}{\text{Stab}}
\def\RomanianComma#1{\setbox0=\hbox{#1}{\ooalign{\hidewidth
    \lower1.2ex\hbox{$\mspace{1mu}^{,}$}\hidewidth\crcr\unhbox0}}}
\newcommand{\Drutu}{Dru{\RomanianComma{t}u}}

\usepackage{ifthen}

\newcommand{\showcomments}{yes}
\renewcommand{\showcomments}{no}

\newsavebox{\commentbox}
{\ifthenelse{\equal{\showcomments}{yes}}%
{\footnotemark
    \begin{lrbox}{\commentbox}
    \begin{minipage}[t]{1.25in}\raggedright\sffamily\upshape\tiny
    \footnotemark[\arabic{footnote}]}
{\begin{lrbox}{\commentbox}}}%
{\ifthenelse{\equal{\showcomments}{yes}}%
{\end{minipage}\end{lrbox}\marginpar{\usebox{\commentbox}}}
{\end{lrbox}}}

\usepackage{color}
\usepackage{pinlabel}

\begin{document}

\title{Finiteness properties of Cubulated Groups}

\author{G. Christopher Hruska$^{\dag}$}
\address{Department of Mathematical Sciences\\
University of Wisconsin--Milwaukee\\
PO Box 413\\
Milwaukee, WI 53201\\
USA}
\email{chruska@uwm.edu}
\thanks{$^{\dag}$ Research supported by NSF grant DMS-0808809}

\author{Daniel T. Wise$^{\ddag}$}
\address{Dept. of Math. \& Stats.\\
McGill Univ.\\
Montreal, QC, Canada H3A 0B9}
\email{wise@math.mcgill.ca}



\thanks{$^{\ddag}$ Research supported by NSERC}

\date{\today}

\begin{abstract}
We give a  generalized and self-contained account of Haglund--Paulin's wallspaces and
Sageev's construction of the CAT(0) cube complex dual to a wallspace.
We examine criteria on a wallspace leading to finiteness properties of its dual cube complex.
Our discussion is aimed at readers wishing to apply these methods to produce actions of groups on cube complexes and understand their nature. We develop the wallspace ideas in a level of generality that facilitates their application.

Our main result describes the structure of dual cube complexes arising from relatively hyperbolic groups.
Let $H_1,\dots, H_s$ be relatively quasiconvex
codimen\-sion-1  subgroups of a group $G$
that is  hyperbolic relative to $P_1, \dots, P_r$.
We prove that $G$ acts relatively cocompactly on the associated dual CAT(0) cube complex $C$.
This generalizes Sageev's result that $C$ is cocompact
when $G$ is hyperbolic.
When $P_1,\dots, P_r$ are abelian, we show that the dual CAT(0) cube complex  $C$ has a $G$-cocompact CAT(0) truncation.
\end{abstract}

\maketitle

\vspace{-.5cm}

{\tiny \tableofcontents}

%%%%%%%%%%%%%%%%%%%%%%%%%%%%%%%%%%%%%%%%%%%%%%%%%%%%%%%%%%%%%%%%%%%%%%%%%%%
\section{Introduction}
\label{sec:Introduction}
%%%%%%%%%%%%%%%%%%%%%%%%%%%%%%%%%%%%%%%%%%%%%%%%%%%%%%%%%%%%%%%%%%%%%%%%%%%

One of the most important themes permeating combinatorial group
theory during the past century has been splittings of groups
as free products, amalgamated products, and HNN extensions.
This topic finally matured when Serre advanced the viewpoint
of a group acting on a tree in the 1970s:
a group $G$ acts essentially on a tree if and only if $G$ splits.

Gromov introduced $\CAT(0)$ cube complexes in his
seminal essay on hyperbolic groups in the 1980s.
As trees are one-dimensional $\CAT(0)$ cube complexes,
there was already a plethora of interesting group actions.
In the 1990s,
Sageev showed how to obtain an action on a $\CAT(0)$ cube complex
from codimension--$1$ subgroups
in a manner that naturally, but surprisingly,
generalized the way Serre's trees arise as the graphs dual
to an embedded collection of surfaces in a $3$--manifold.

While only free groups can act freely on a one-dimensional
$\CAT(0)$ cube complex, the variety of groups acting freely in
higher dimensions is staggering.
When $G$ acts freely on a $\CAT(0)$ cube complex $C$,
the quotient $G\backslash C$
provides a natural $K(G,1)$
for $G$,
and many properties of $G$ are revealed
by examining combinatorial and geometric properties of $G \backslash C$
or the action on $C$.
For instance,
applications towards automaticity and the lack of Property~(T)
are given in
\cite{NibloReeves98,NibloReeves97,NibloRoller98}.

While a minimal action of a finitely generated
group $G$ on a tree $T$ immediately
provides a compact ``graph of groups'',
the situation is substantially more complex in higher
dimensions---even when the action is free.
For instance, in general the action of $G$ on the cube complex $C$
provided by Sageev's construction
may fail to be cocompact; $C$ may even fail to be finite dimensional.
However, Sageev proved that $G$ acts cocompactly on
$C$ when the codimension--$1$ subgroups are
quasiconvex and $G$ is hyperbolic---and generalizing
this theorem is a primary motivation of this paper.
There has been much recent research
finding appropriate codimension--$1$ subgroups
of a group $G$,
and then using the dual cube complex to illuminate $G$.
A notable success is the theorem on the
``structure of groups with a quasiconvex hierarchy''
\cite{WiseStructureAnnouncement09,WiseIsraelHierarchy}.
Our main result plays a critical role there
in understanding the dual cube complex arising for
a cusped hyperbolic $3$--manifold.

\subsection{The goal of this paper\except{toc}{:}}
The main goal of this paper is to examine the finiteness properties of the dual cube complex obtained by applying Sageev's construction
to a group together with a collection of
codimension--$1$ subgroups.
We have especially concentrated on the case where the group is relatively hyperbolic, as there are powerful results to be gleaned in this case that require some care.

The following is a simplified version of our main result, which is
Theorem~\ref{thm:MainResult}.
We emphasize that this result plays an important role in the
relatively hyperbolic case of the work
in \cite{WiseIsraelHierarchy}.

\begin{thm}[Relative cocompactness]
\label{quasithm:RelCocompact}
Let $G$ be hyperbolic relative to $P_1,\dots,P_j$,
and let $H_1,\dots,H_k$ be
relatively quasiconvex subgroups.
Choose an $H_i$--wall for each $i$.
Let $C$ be the cube complex dual to the wallspace
consisting of these $H_i$--walls and their $G$--translates.
The induced action of $G$ on $C$ is cocompact relative
to $P_i$--invariant subcomplexes $C_i$
\textup{(}see Definition~\ref{defn:RelCocompact}\,\textup{)}.
\end{thm}

An $H_i$--wall is an $H_i$--invariant
way to ``cut $G$ in half'' along $H_i$.
This is of particular interest when $H_i$ is a codimension--$1$
subgroup of $G$ (see Section~\ref{sub:Codimension1}).
Sometimes there is coarsely a unique
``natural'' partition---for
instance when $G$ and the $H_i$ are fundamental groups of
closed aspherical manifolds \cite{KapovichKleiner2005}.
Since a subgroup $H_i$
might ``cut $G$ in half'' in more than one way,
we are unable to associate a specific dual cube complex
with a collection of subgroups.
Instead we must specify the additional
data of the chosen partitions of $G$,
which are the ``$H_i$--walls'' discussed in
Definition~\ref{def:HWall}.

\begin{defn}[Relative cocompactness]
\label{defn:RelCocompact}
Let $(G,\mathbb{P})$ be a group together with a collection of
subgroups $\mathbb{P} = \{P_1,\dots,P_r\}$.
We say $G$ acts \emph{relatively cocompactly} on the cube complex $C$
if there is a compact subcomplex $K$ and $P_i$--invariant
subcomplexes $C_i$
such that:
\begin{enumerate}
\item $C= GK \cup \bigcup GC_i$
\item $gC_i \cap g'C_j \subset GK$ unless $gC_i=g'C_j$ and $i=j$.
\end{enumerate}
\end{defn}

In many cases, the dual cube complex $C$ of
Theorem~\ref{quasithm:RelCocompact}
can be ``truncated'' by cutting off the infinite ``ends''
of the subcomplexes $C_i$.
This is analogous to the classical case, where one truncates the
cusps of a finite volume hyperbolic manifold.
We illustrate this in the following simplified version
of Corollary~\ref{cor:AbelianTruncation},
which can be applied to complete the proof that
B(6) groups are $\CAT(0)$
and similarly to complete an alternate proof of the Alibegovi\v{c}--Bestvina
result that limit groups are $\CAT(0)$ \cite{AlibegovicBestvina2006}.

\begin{thm}[$\CAT(0)$ truncation]
Let $G$ be hyperbolic relative to virtually abelian subgroups
$P_1,\dots,P_j$.
Let $H_1,\dots,H_k$ be relatively quasiconvex
subgroups with associated $H_i$--walls.
Suppose $G$ acts properly on the dual cube complex $C$.
Then $C$ contains a convex cocompact $G$--invariant subspace.
In particular, $G$ acts properly and cocompactly on a $\CAT(0)$ space.
\end{thm}

We emphasize that the vast majority of our effort and the exposition
is focused around the proof of Theorem~\ref{quasithm:RelCocompact}
as well as its relation to the Bounded Packing Property
as developed in \cite{HruskaWise09}.
However we have also written the paper with the intention of introducing
the methods to new researchers as well as providing a self-contained
reference for future work in this area.
We have included together fundamental results, straightforward
results, and some folk results in a unified language.

While the fundamental tool is Sageev's dual cube complex construction,
the central unifying notion is a wallspace.
These were first introduced by Haglund--Paulin,
but we have developed their notion so that it is
both more flexible and has greater reach
in a manner that maintains its natural geometric motivation.

\medskip
\noindent\textbf{Wallspaces:}
Haglund--Paulin's original definition of a wallspace has an underlying
set $X$ whose walls are partitions $X = U \sqcup V$.
We soften this definition by relaxing their requirement that
$U \cap V = \emptyset$.
We do not believe there is a theoretical difference in the level
of applicability.
However, practically this more general definition facilitates
applications by avoiding awkward choices.
For instance, in a geometric setting when $X$ is not discrete,
a wall often naturally arises from a separating subspace $W$, and
it is convenient to define $U,V$ to be the two closed
halfspaces intersecting in $W$.
One already saw the arbitrary inclusion of $W$ on one side
but not the other in Sageev's thesis, where $W$ arises as a neighborhood
of a codimension--$1$ subgroup.

When the wallspace $(X,\W)$ is independently a metric space $(X,\dist)$,
we say $(X,\W)$ has the \emph{Linear Separation Property}
if there exist constants $\kappa>0$ and $\epsilon$ such that
for all points $x,y \in X$
we have $\#(x,y) \ge \kappa\,\dist(x,y) - \epsilon$.
Here $\#(x,y)$ denotes the number of walls separating $x$ and $y$.
The following theorem, which is restated and proven as
Theorem~\ref{thm:LinearSeparationProper}, has been observed repeatedly
in various forms by researchers applying Sageev's construction.
It is easy, but very useful.

\begin{thm}[Linear Separation $\Longrightarrow$ Proper]
\label{thm:LinearSeparationProperIntro}
Suppose $G$ acts on a wallspace $(X,\W)$,
and the action on the underlying metric space $(X,\dist)$
is metrically proper.
Then the Linear Separation Property for $(X,\W)$ implies that $G$
acts metrically properly on the dual cube complex $C=C(X,\W)$.

More generally $G$ acts metrically properly on $C$ provided that
for some $x\in X$ we have:
$\#(x,gx)\rightarrow \infty$ as $g\rightarrow \infty$.
\end{thm}

Although Theorem~\ref{thm:LinearSeparationProperIntro} holds under the more
general hypothesis stated second, we have deliberately chosen to emphasize
the linear separation hypothesis, as this stronger property on $(X,\W)$
is often more natural to verify.

In practice it is often more difficult to verify the
properness of an action on the dual cube complex
than the (relative) cocompactness.
This is even true in a $\delta$--hyperbolic setting, where
quasiconvexity is a robust phenomenon,
but the property of having sufficiently many separating walls is not locally
discernible in general.

Metric properness translates directly into a very natural ``filling''
condition on the wallspace, but verifying this condition is treacherous
because a collection of walls that locally appears to be highly filling
can be very ineffective globally.
A suggestive (messy) example is given by a highly self-intersecting
essential curve $\sigma$ in a closed surface $S$
so that the complementary components of $S-\sigma$ are all simply connected.
When $\sigma$ is homotopic to a simple curve, then
the cube complex dual to the wallspace consisting of translates of
$\tilde \sigma$ in the universal cover $\tilde S$
is quasi-isometric to a tree.
In this case, while the action is cocompact it cannot be proper.

We illustrate some of the main results
in a simple and suggestive framework in Section~\ref{sec:ConvexWalls}.
For many of these examples, linear separation is obvious because
the simplicity of the walls makes the necessary computation transparent.
In settings that arise ``in the wild'', these computations
can be opaque.
For instance, linear separation is proven for Gromov's random groups
at density $< 1/6$ by Ollivier--Wise in \cite{OllivierWiseDensity},
but at density $< 1/5$
although a wallspace was constructed and the walls are quasiconvex
it remains unknown whether these random groups can
act properly---i.e., with finite stabilizers---on a cube complex.

We offer a second closely related criterion for verifying
properness that is particularly suited to a hyperbolic
setting, where each infinite order element has an axis.

The properness conclusion in the following
is a special case of Theorem~\ref{thm:AxisSeparation}.
The cocompactness follows from Sageev's hyperbolic special case of
Theorem~\ref{quasithm:RelCocompact}.

\begin{thm}[Axis Separation]
Let $G$ be a word-hyperbolic group.
Let $H_1,\dots,H_k$ be quasiconvex subgroups
with associated $H_i$--walls.
Suppose that for each infinite order element $g \in G$
its axis $A_g$ is cut by some translate $g'H_j$.
Then $G$ acts properly and cocompactly on the dual cube complex.
\end{thm}

In Bergeron--Wise \cite{BergeronWiseBoundary}, the properness and cocompactness are treated simultaneously in
a rather soft fashion that works effectively when there is a rich collection of codimension--$1$ subgroups to draw from.
The idea in \cite{BergeronWiseBoundary} is to use $\boundary G$
to choose finitely many quasiconvex codimension--$1$ subgroups
so that each geodesic (and hence each axis) is separated by a wall.

When $G$ acts properly and cocompactly on the dual cube
complex $C=C(X,\W)$ then we obtain local
finiteness of $C$ by a simple argument.
We formulate a characterization of the local finiteness of $C$
that does not require a group action.
The following is a simplified statement
of Theorem~\ref{thm:LocalFiniteness},
which is related to the Parallel Walls Separation Condition
of Brink--Howlett as applied by Niblo--Reeves to prove
local finiteness of their cubulation of a f.g.\ Coxeter group
\cite{BrinkHowlett93,NibloReeves03}.

\begin{thm}[Local finiteness]
Let $(X,\dist)$ be a metric space with a locally finite
collection $\W$ of walls.
Then the dual cube complex $C=C(X,\W)$ is locally finite if and only if
for each compact set $K \subseteq X$ there exists a constant $f(K)$
such that
whenever $\dist(K,W) \geq f(K)$ there is a wall $W'$
separating $K$ from $W$.
\end{thm}

In Section~\ref{sec:FinitenessRH}
we examine the finiteness properties of $C(X,\W)$
in much greater detail in our
principal setting where the relatively hyperbolic
group $G$ is acting on $X$.

\subsection{Why cubulate?}
\label{sec:Why}

We describe below a sample of the information one obtains about a group
admitting cubulations with various levels of finiteness properties.

The earliest motivation for defining nonpositively curved cube complexes
is simply that:
a proper/cocompact action on a $\CAT(0)$ cube complex yields
a proper/cocompact action on a $\CAT(0)$ metric space.
Interestingly enough, this remains a fascinating albeit simple
application, especially in view of the methodology
promoted in this paper, which often surprisingly
provides a $\CAT(0)$ metric from codimension--$1$ subgroups.
As mentioned earlier, the most desirable goal of a cocompact $\CAT(0)$
space can sometimes be obtained by truncating.

If $G$ acts essentially on a $\CAT(0)$ cube complex then $G$ cannot have
Property~(T) by Niblo--Roller \cite{NibloRoller98}.
Moreover, it is implicit in their argument that
if $G$ acts metrically properly then $G$ is a-T-menable
(see the discussion in \cite[Sec~7.4.1]{HaagerupBook01}).
If $G$ acts properly and cocompactly on a $\CAT(0)$ cube complex then
$G$ is biautomatic by Niblo--Reeves \cite{NibloReeves98}.
If $G$ has bounded torsion and acts properly on a finite dimensional
$\CAT(0)$ cube complex then $G$ satisfies a strong form of the
Tits Alternative by Sageev--Wise \cite{SageevWiseTits}.

Right-angled Artin groups arise as the fundamental groups of
well-known nonpositively curved cube complexes.
Exploiting their combinatorial geometry,
a connection has been developed by Haglund--Wise
in \cite{HaglundWiseSpecial},
which has the potential to virtually embed a nicely cubulated group $G$
into a right-angled Artin group,
and thus expose many algebraic features of $G$.

If $G$ acts properly on a $\CAT(0)$
cube complex $C$, then the asymptotic dimension of $G$
is bounded above by the dimension of $C$
by \cite{WrightDimension}.
One obtains a rich family of quasimorphisms of $G$
from irreducible actions on a $\CAT(0)$ cube complex
by recent work of Caprace--Sageev in \cite{CapraceSageev2011}.
When $G$ acts properly on a $\CAT(0)$ cube complex, then
Hagen has shown that $G$ is weakly hyperbolic relative to the hyperplane stabilizers \cite{HagenArboreal}.

\subsection{Survey of known results towards cubulation of groups}

There has been an increasing amount of work towards obtaining
actions of groups on $\CAT(0)$ cube complexes.

Although there has certainly been earlier work, the Bass--Serre theory
of groups acting on trees is a milestone here \cite{Serre77}.
While splittings of groups as amalgamated products or HNN extensions
had been studied for a long time previously,
great perspective was added to this subject through the
theory of groups acting on trees.
The edge groups of these splittings had already been identified as
codimension--$1$ subgroups by Scott \cite{Scott77},
but these ideas were only put into their natural context
with Sageev's thesis \cite{Sageev95}.
Sageev gave his dual cube complex construction and showed that
a group $G$ acts essentially on a $\CAT(0)$ cube complex if and only if
$G$ contains a codimension--$1$ subgroup.
The meaning of ``essential'' was subsequently refined by
Gerasimov and Niblo--Roller \cite{Gerasimov97,NibloRoller98}
to mean that $G$ acts with no global fixed point.
A beautiful exposition of these ideas in a crisp abstract setting
was given by Roller in \cite{RollerPocSets}.

Two early applications of Sageev's construction were given by Niblo--Reeves
\cite{NibloReeves03} towards cubulation of Coxeter groups
and by Wise \cite{WiseSmallCanCube04} towards cubulation of
certain small-cancellation groups.
The role played by wallspaces in Sageev's construction
was subsequently made explicit by Nica and Chatterji--Niblo in
\cite{NicaCubulating04,ChatterjiNiblo04}.

Further applications are
towards cubulating one-relator groups with torsion \cite{LauerWise07},
cubulating Gromov's random groups \cite{OllivierWiseDensity},
cubulating graphs of free groups with cyclic edge groups \cite{HsuWiseCubulating},
cubulating arithmetic hyperbolic groups of simple type
\cite{BergeronHaglundWiseSimple},
cubulating rhombus groups \cite{JanzenWise13},
cubulating certain malnormal amalgams of cubulated groups \cite{HsuWiseCubulatingMalnormal}, and
most recently cubulating the group of Formanek--Procesi
\cite{Gautero12}.

The variety of cubulations that one can obtain by applying Sageev's
construction to a fixed group $G$
is reflected in \cite{WiseFreeCubulation}, where it is shown that for each finitely generated infinite index subgroup $H$ of a finitely generated free group $G$, there is a free $G$--action on a $\CAT(0)$ cube complex with one orbit of hyperplane, such that the stabilizer of each hyperplane equals $H$.

There has been some other work obtaining a cube complex more directly
without an application of Sageev's construction. In particular,
we note the work of
Weinbaum using Dehn's presentation to cubulate prime alternating link
complements (see e.g., \cite{WiseFigure8}),
the work of Aitchison--Rubinstein
\cite{AitchisonRubinstein90} cubulating certain
$3$--manifolds,
the work of Brady--McCammond \cite{McCammond2009} who recognized how to thicken a complex
to obtain a cube complex and applied this to certain $3$--manifolds
and to $C'(1/4)-T(4)$ complexes, and finally with a very different
flavor, the work of Farley \cite{Farley2003} cubulating the diagram groups
of Guba--Sapir \cite{GubaSapir2005}.
At this point, it appears that the strategy of using Sageev's construction
affords the greatest variety of applications.

\section{Wallspaces}
\label{sec:Wallspaces}

\subsection{Wallspaces}
\label{sub:Wallspaces}

Let $X$ be a nonempty set.
A \emph{wall} of $X$ is a pair of subsets
$\{U,V\}$ called (closed) \emph{halfspaces} such that $X = U\cup V$.
We shall not assume that
$U \cap V = \emptyset$---see Remark~\ref{rem:Partition}.
The \emph{open halfspaces} associated to the wall $\{U,V\}$
are the sets $U-(U\cap V)$ and $V-(U\cap V)$.

Points $x,y$ of $X$ are \emph{separated} by a wall $W$ if they lie in
distinct open halfspaces of $W$.
The notation $\#(x,y)$ denotes the number of walls
separating $x$ and $y$.
We say that a point $x$ and a wall $\{U,V\}$ \emph{betwixt each other}
if $x \in U \cap V$.

A \emph{wallspace} is a set $X$ together with a collection of walls $\W$
on $X$ with the following two finiteness properties:
\begin{enumerate}
   \item $\#(x,y)<\infty$ for all $x,y\in X$, and
   \item each point $x$ betwixts only finitely many walls.
\end{enumerate}
We do allow our ``collection'' to contain
duplicates---see Remark~\ref{rem:Bourbaki}---but
we do not allow duplicate walls that are genuine
partitions---see Remark~\ref{rem:Duplicity}.

Distinct walls $W=\{U,V\}$ and $W' = \{U',V'\}$ are \emph{transverse}
if all four of the following intersections are nonempty:
\[
   U \cap U', \quad U \cap V', \quad V \cap U', \quad V\cap V'.
\]

The central example of a wallspace
arises from a $\CAT(0)$ cube complex $C$
by letting $X=C^0$ and letting $\W$ consist of the partitions of $C^0$ induced by hyperplanes.
Note that the walls are genuine partitions here.
A similar example is a wallspace on $C$ whose walls are
pairs of closed halfspaces associated to hyperplanes of $C$.
We discuss this type of example further in
Section~\ref{sub:GeometricWallspaces}.

A group $G$ \emph{acts} on the wallspace $(X,\W)$ if
$G$ acts on $X$ and stabilizes $\W$. Consequently $\#(x,y)=\#(gx,gy)$ for each $g\in G$.

\begin{rem}
\label{rem:Partition}
We note that in Haglund--Paulin's original definition
\cite{HaglundPaulin98} of
``\emph{espaces \`{a} murs}'' the walls are partitions of $X$;
in other words $U \cap V= \emptyset$.
We have chosen to work in a slightly more flexible setting.
\end{rem}

\begin{rem}[Duplicated walls]
\label{rem:Duplicity}
In Remark~\ref{rem:Bourbaki},
we describe a more general definition, essentially using indexed
sets, that allows ``multiplicity'' in our collection of walls,
in the sense that it is possible for $W_1 = \{U,V\}$ and $W_2 = \{U,V\}$.
However,
\begin{itemize}
\item[(\dag)] we do not allow duplicate walls that are
\emph{genuine partitions}
\end{itemize}
in the sense that $U \cap V = \emptyset$ and $U$ and $V$ are
both nonempty.
This will be important in the proof of connectedness of the canonical
component of the dual cube complex given in Proposition~\ref{prop:Connected}.
Duplicate
\emph{vacuous walls} of the form $\{X,\emptyset\}$ do not
create difficulties.

The definition of transversality given below
implies that duplicated walls $W_1,W_2$
are transverse if and only if they are not partitions.
It is conceivable that our treatment could be modified
in the exceptional partition case
so that all duplicated walls are transverse.

One natural way that duplicated walls arise is when $(X,\W)$ is a wallspace
and $Y \subset X$ and we give $Y$ the induced \emph{subwallspace} structure
whose walls have the form $(Y\cap U,Y\cap V)$,
see Definition~\ref{def:Subwallspace}.
We refer the reader to the right-hand side of
Figure~\ref{fig:Duplicity}, where the
discrete subset $Y$ has a duplicated wall that is not a partition.
In contrast, on the left-hand side of the figure, $Y$ has a duplicated
wall that is a genuine partition, which we do not allow.
The reader can imagine thick walls arising from two infinite cyclic subgroups
of a surface group and $Y$ is a third subgroup.
\end{rem}

\begin{figure}
\begin{center}
\includegraphics[width=.6\textwidth]{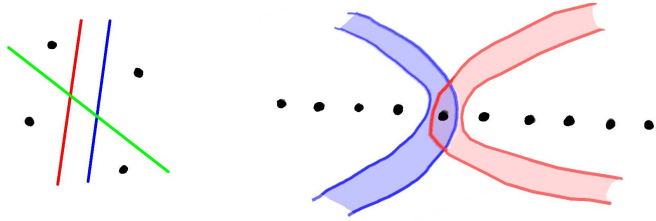}
\end{center}
\caption{On the left, the discrete subset $Y$ has a duplicated wall
that is a genuine partition, which we do not allow.
On the right, the discrete subset $Y$ has a duplicated wall
that is not a partition.}
\label{fig:Duplicity}
\end{figure}

The following more formal approach correctly permits duplicated walls.

\begin{rem}[Wallspaces using indexed collections of halfspaces]
\label{rem:Bourbaki}
A \emph{wallspace} is a pair of sets $(X,\W)$
where each element $W \in \W$
is a pair of indexed subsets $\{H_{-W}, H_{+W}\}$ of $X$
called the \emph{closed halfspaces} of $W$.
Note that the subsets are indexed by the elements of
$\{-,+\} \times \W$.
Moreover we require that the following properties are satisfied:
\begin{enumerate}
\item[(0)] $H_{-W} \cup H_{+W} = X$ for all $W \in \W$,
\item For all $x,y \in X$, there are finitely many $W \in \W$
such that $x \in (H_{-W}) - (H_{+W})$ and $y \in (H_{+W}) - (H_{-W})$.
\item For each $x \in X$, there are finitely many $W \in \W$
such that $x \in H_{-W} \cap H_{+W}$.
\item If $\{H_{-W},H_{+W}\} = \{H_{-W'},H_{+W'}\}$ and $H_{-W} \cap H_{+W} = \emptyset
= H_{-W'} \cap H_{+W'}$ then $W=W'$.
\end{enumerate}
Note that (3) excludes the duplicate walls that are genuine partitions,
as in (\dag).
However, for walls that are not partitions,
it is now sensible for $W,W'$ to be distinct walls
with the same closed halfspaces.
For though $H_{-W},H_{-W'}$ and $H_{+W},H_{+W'}$ have the same underlying
sets, their indices differ.
\end{rem}

\subsection{Geometric wallspaces}
\label{sub:GeometricWallspaces}

A \emph{geometric wallspace} is a metric space $(X,\dist)$
with a collection $\W$ of closed
subspaces called \emph{geometric walls}
such that each geometric wall $W$ is connected, and $X-W$ has exactly two components.

Letting $U$ and $V$ be the components of $X-W$, one obtains closed halfspaces
$W \cup U$ and $W \cup V$, and hence a wall
$\bigl\{(W \cup U), (W \cup V)\bigr\}$ associated to $W$.

Our earlier definition of separation
is motivated by the observation that
a geometric wall $W$ separates $x$ and $y$ if they
lie in distinct components of $X-W$.

Note that $\#(x,y)<\infty$ provided $x,y$
are connected by a path intersecting only finitely many walls of $\W$.
Thus when $X$ is path connected, the local finiteness of the
collection of geometric walls is the critical ingredient
ensuring that $\#(x,y)<\infty$ always holds.

To ensure that geometric wallspaces are associated to underlying
(abstract) wallspaces, we will therefore insist that
they are path connected and that the collection
of geometric walls is locally finite in the sense that
compact subsets of $X$ intersect only finitely many geometric walls.

Many natural examples of geometric wallspaces are geodesic metric
spaces whose walls are convex subspaces.
For instance, this was the case for hyperplanes in a $\CAT(0)$
cube complex $C$.

\begin{lem}
\label{lem:TransverseWallsMeet}
Suppose $(X,\W)$ is a geometric wallspace.
Two distinct abstract walls are transverse
if and only if their associated geometric walls intersect.
\end{lem}

\begin{proof}
If two geometric walls $W$ and $W'$ intersect, then it is clear that
their corresponding abstract walls are transverse, since
each geometric wall lies in both of its associated halfspaces.

Now suppose $W$ and $W'$ are disjoint geometric walls.
Let $U$ and $V$ be the components of $X-W$, and let $U'$ and $V'$ be
the components of $X-W'$.

Since $W$ is connected and disjoint from $W'$,
either $W \subset U'$ or $W \subset V'$.
Without loss of generality, assume $W \subset V'$.
Thus $W \cap U' = \emptyset$.
Similarly, we can assume that $W' \subset V$.
Thus $W' \cap U = \emptyset$.
Since $W \cup U$ is a connected set in the complement
of $W'$ that interects $V'$, we must have
$W \cup U \subseteq V'$.  Therefore
$(W \cup U) \cap (W' \cup U')$ is empty.

Thus the abstract walls associated to $W$ and $W'$ are not transverse.
\end{proof}

\begin{defn}[Subwallspace]
\label{def:Subwallspace}
Let $(X,\W_X)$ be a wallspace and $Y \subset X$.
The induced \emph{subwallspace} structure on $Y$
is the wallspace $(Y,\W_Y)$
whose \emph{induced walls}
have the form $(Y\cap U,Y\cap V)$, for $(U,V) \in \W_X$.
We will ignore the vacuous partitions $(Y,\emptyset)$ in
a subwallspace whenever convenient, as they play no role in the
dual cube complex.

Note that we allow duplicate walls, but
because of the difficulties discussed in Remark~\ref{rem:Duplicity},
we will only consider subwallspaces when there do not exist distinct
walls $(U,V)$ and $(U',V')$ of $X$ such that
$(Y\cap U,Y\cap V)$ and $(Y\cap U', Y\cap V')$
are the same nonvacuous partition.

When $X$ is a geometric wallspace and $Y$ is a connected
subspace, there do not exist (nonvacuous) induced walls
that are genuine partitions.
Indeed, if $Y\cap W = \emptyset$ then the connected set
$Y$ lies in one of the
open halfspaces of $W$, so the induced wall is vacuous.
\end{defn}

\begin{exmp}
Let $X = \R^2$ and let $\W_X$ be the $x$--axis and $y$--axis.
These two walls are transverse.
Indeed, they intersect as guaranteed by
Lemma~\ref{lem:TransverseWallsMeet}.
We now examine the transversality of walls in induced subwallspaces
on subsets of $X$.

Let $Y = S^1 \subset \R^2$, and let $(Y,\W_Y)$
be the induced subwallspace.
Note that this is not a geometric wallspace.
Observe that the two induced walls of $Y$ are transverse but do not
intersect.  Indeed the same holds when $Y = \{\pm 1\}^2$.

Now however, let us consider
$Z = \bigl\{ (+1,-1), (+1,+1), (-1,+1) \bigr\}$, and let $(Z,\W_Z)$ be the induced subwallspace.
In this case, the induced walls of $Z$ are no longer transverse.
\end{exmp}

\subsection{Tracks and transverse walls}

Consider a group $G$ that acts on a wallspace $(Y,\W)$
and also acts on a $2$--complex $X$.
A $G$--equivariant map $X^0\to Y$ induces a wallspace structure on $X^0$.
Of course we could extend this to a (noncontinuous) $G$--equivariant map
$X \to Y$ to make $X$ into a wallspace.
But this might not be geometrically satisfying.
Instead, the following proposition
explains how to enlarge $X$ to $X'$
so that $X'$ and $Y$ have consistent wall structures,
and $X'$ is a geometric wallspace whose walls are tracks
in the sense of Dunwoody.
It is proven by combining variants of Dunwoody resolutions and Stallings binding ties.

\begin{prop}[Walls as tracks]
\label{prop:WallsAsTracks}
Let $G$ act freely on a simply connected
$2$--complex $X$ and suppose $X^0$ is a wallspace with no betwixting
and the wall system is locally finite, with finitely many orbits of walls
each of which has finitely generated stabilizer.
Then there exists a finite $G$--equivariant expansion of $X$
to a $2$--complex $X'$ with the same $0$--skeleton and
the same set of walls and each wall in $X'$ is represented by
a track.
\end{prop}

\begin{proof}
Let $C$ be the dual cube complex of $X^0$.
Consider the $G$--equivariant map $X^0 \to C$ that sends a $0$--cell
to its canonical $0$--cube.
For each $1$--cell of $X$, choose a combinatorial path in $C$,
and in this way obtain a $G$--equivariant map $X^1 \to C$.
For each $G$--representative of
$2$--cell $R$ of $X$, choose a minimal area
disc diagram $D \to C$ for the path $\boundary R \to C$.
We thus obtain a $G$--equivariant map $X \to C$.
Since $D$ has minimal area, all tracks within $D$ end on
$\boundary D$, and hence all tracks intersect $X^1$
so no ``new'' tracks have been added.

As disc diagrams might be singular, this map might not be
combinatorial.
However, $D$ is a square diagram since $C$ is cubical,
and the $2$--cell $R$ can be subdivided so that the map $R \to D$ is
well-behaved
in the sense that the preimage of a ``dual curve'' (a hyperplane) of $D$
is an embedded arc in $R$.
To see this, for each cut-vertex of $D$, we choose a tree in $R$
projecting to it.  For each cut-edge of $D$ we choose a rectangle in $R$
projecting to it---see the left of Figure~\ref{fig:ExpandedDiagram}.
Finally squares of $D$ which have a $0$--cube on a cut-point can be replaced by pentagons.
The result is that the preimage of dual curves of $D$ are tracks in $R$.

\begin{figure}
\labellist
\small
\pinlabel $e_\sigma$ [r] at 875 319
\pinlabel $p$ [b] at 964 542
\pinlabel $q$ [t] at 962 78
\pinlabel $D'$ at 1148 327
\pinlabel $\sigma$ [bl] at 1234 490
\pinlabel $\alpha$ [r] at 1620 390
\pinlabel $\beta$ [l] at 1718 390
\pinlabel $p$ [b] at 1660 551
\pinlabel $q$ [t] at 1660 82
\pinlabel $D'$ at 1868 327
\pinlabel $\sigma$ [bl] at 1943 493
\endlabellist
\begin{center}
\includegraphics[width=.85\textwidth]{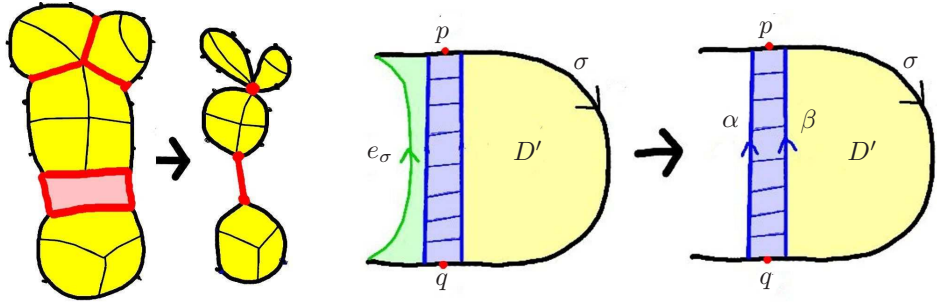}
\end{center}
\caption{On the left we illustrate the map $R\rightarrow D$.
On the right we illustrate the expanded 2-cell and associated disc diagram.}
\label{fig:ExpandedDiagram}
\end{figure}

In conclusion, the preimage of each hyperplane of $C$ is a
\emph{pattern} in $X$, meaning that it is a disjoint union of tracks.

We now use that the stabilizers of walls are finitely generated to
equivariantly connect each such pattern by equivariantly expanding along
$2$--cells.
Indeed, let $p$ and $q$ be points in $X^1$ that lie in distinct
tracks mapping to the same hyperplane $H$ of $C$.
Let $\sigma$ be a combinatorial path whose first $1$--cell contains $p$
and whose last $1$--cell contains $q$.
Replacing $\sigma$ with a subpath if necessary,
we can assume without loss of generality that $p,q$
lie in distinct tracks but that no point of $\sigma$
between $p,q$ maps to $H$.
We add a new $1$--cell $e_\sigma$ whose endpoints are $p,q$
and add a new $2$--cell $R_\sigma$ between $\sigma$ and $e_\sigma$.
We will now choose a disc diagram $D_\sigma$ that $R_\sigma$
will map to, similarly to the method used as above.

We now refer the reader to the right of Figure~\ref{fig:ExpandedDiagram}.
Let $L$ be a geodesic ladder in $C$ joining the $1$--cubes containing the
images $\bar p$ and $\bar q$.
Let $\alpha$ and $\beta$ be the two rails of $L$.
The image $\bar\sigma$ of $\sigma$ is a concatenation
$\bar\sigma_p \bar\sigma' \bar\sigma_q$,
where the last edge of $\bar\sigma_p$ contains $\bar{p}$ and
the first edge of $\bar\sigma_q$ contains $\bar{q}$.
Let $D'$ be a minimal area diagram between $\beta$ and $\bar\sigma'$.
Let $D_\sigma$ be the disc diagram built from
$\bar\sigma_p \cup \bar\sigma_q \cup L \cup_\beta D'$.
We map $e_\sigma$ to the concatenation
$\bar\sigma_p \alpha \bar\sigma_q$.
Finally observe that $\boundary R_\sigma \to C$ equals the boundary
path of $D_\sigma$.
This allows us to map $R_\sigma$ to $D_\sigma$

The ladder $L$ connects the track containing $p$ to the track containing $q$.
No new tracks are introduced.
Indeed the minimal area of $D'$ ensures that each dual curve of $D_\sigma$
intersects the boundary.
The geodesicity of $L$ ensures that dual curves crossing $L$ emerge on
$\bar\sigma'$.
And thus since $e_\sigma$ maps to $\bar\sigma_p \alpha \bar\sigma_q$,
no new tracks emerge on $e_\sigma$.

Using that the stabilizers are finitely generated, we must repeat this procedure only finitely many times to
connect all of the components.
\end{proof}

We now give a technical result asserting that transverse walls
coarsely intersect.  The result is stated using the language of
$H_i$--walls,
and we refer the reader to Definition~\ref{def:HWall}
for the requisite terminology.
Since we will only use this result one time, in the proof of
Theorem~\ref{thm:MainResult},
we suggest that this technical point be skipped on first reading.
We have chosen to  position this lemma together with
Lemma~\ref{lem:TransverseWallsMeet}
and Proposition~\ref{prop:WallsAsTracks}
in order to emphasize that transverse walls do coarsely intersect
in an appropriate sense.

\begin{lem}[Transverse $\Longrightarrow$ Close]
\label{lem:TransverseToCloseSubgroups}
Let $G$ be a group with finite generating set and
Cayley graph $\Gamma$.
Let $H_i$ and $H_j$ be finitely generated subgroups of $G$.
Let $\{U_i,V_i\}$ be an $H_i$--wall and let $\{U_j,V_j\}$ be an
$H_j$--wall.
Then there exists $D$ such that
for all $g,g' \in G$, if the walls
$g\{U_i,V_i\}$ and $g'\{U_j,V_j\}$ are transverse
then $\dist_{\Gamma}(gH_i,g'H_j) < D$.
\end{lem}

Technically ``transverse'' was defined only in the context
of a wallspace, but this relation is obviously definable in the same way here.  Indeed $\Gamma$ can be regarded as a wallspace with two walls.

\begin{proof}
We now describe a construction for an arbitrary finitely generated
subgroup $H \le G$ and an
$H$--wall $\{U,V\}$ in $G$.
Let $U^1$ and $V^1$ be the closed $1$--neighborhoods of
$U$ and $V$ in $\Gamma$.
Then $H$ acts cocompactly on $U^1 \cap V^1$.
Here we use that $U$ and $V$ have $\ddot H$--finite frontiers
and $U \cap V$ is $H$--finite.

We now use that $H$ is finitely generated.
By adding finitely many $H$--orbits of vertices and edges to $U^1$ and $V^1$
we obtain $U^2$ and $V^2$ that are connected and where
$W^2 = U^2 \cap V^2$ is nonempty.
By adding finitely many $H$--orbits of vertices and edges to $W^2$,
we obtain an $H$--cocompact $H$--invariant connected graph $W^3$
containing $H$.
Finally let  $U^3 = U^2 \cup W^3$
and let $V^3 = V^2 \cup W^3$.

We note that $U^3$ and $V^3$ are connected and $H$--invariant, since
each is an intersecting union of connected $H$--invariant subspaces.
Moreover $U^3 \cap V^3$ is connected, since it equals $W^3$
by construction.

We apply the above construction to $\{U_i,V_i\}$ and $\{U_j,V_j\}$
to obtain geometric walls $\bigl\{\bar U_i,\bar V_i\bigr\}$
and $\bigl\{\bar U_j, \bar V_j\bigr\}$ in $\Gamma$.
We emphasize that $\bar W_i = \bar U_i \cap \bar V_i$
is a (connected) $H_i$--cocompact geometric wall,
and likewise for $\bar W_j = \bar U_j \cap \bar V_j$
and $H_j$.

Observe that transversality persists after thickening.
Now apply Lemma~\ref{lem:TransverseWallsMeet}
to the transverse geometric walls $g\bigl\{\bar U_i,\bar V_i\bigr\}$
and $g'\bigl\{\bar U_j,\bar V_j\bigr\}$.
Therefore $g \bar W_i \cap g' \bar W_j$ is nonempty.
The result follows by letting
$D = \diam\bigl( H_i \backslash \bar W_i\bigr)
+ \diam\bigl( H_j \backslash \bar W_j\bigr)$.
\end{proof}

\subsection{Codimension--$1$ subgroups}
\label{sub:Codimension1}

Let $G$ be a finitely generated group with Cayley graph $\Gamma=\Gamma(G,S)$.
A subgroup $H\subset G$ is \emph{codimension--$1$} if
the coset graph $\bar\Gamma = H\backslash \Gamma$
is a \emph{multi-ended} graph in the following sense:
$\bar\Gamma - \Lambda$ has $2$ or more infinite components
for some compact subgraph $\Lambda \subset \bar\Gamma$.
Note that this definition is independent of the finite generating set $S$.

For example, it is well known that any copy of $\Z^n$ in $\Z^{n+1}$
is codimension--$1$.
A frequently encountered example arises when
$G$ is isomorphic to a nontrivial amalgamated
product $A *_C B$ or an HNN extension $A *_C$,
in which case $C$ is a codimension--$1$ subgroup
by a result of Scott \cite{Scott77}.
Another suggestive example is any infinite cyclic subgroup of a
closed surface group.

Every finitely generated infinite index subgroup of a free group
is codimension--$1$.
However the same cannot be said for a closed surface group.
The reader is urged to consider a covering space with a compact core
having exactly one boundary circle;
the corresponding subgroup is not codimension--$1$,
but it does have a property we now briefly turn to.

Observe that if $H$ is codimension--$1$, then $H$ is a \emph{divisive}
subgroup of $G$
in the sense that
for some $d>0$ the complement
$\Gamma-\neb_d(H)$ has more than one component $K$ that is \emph{deep},
meaning that $K$ does not lie in $\neb_c(H)$ for any $c>0$.

Returning to the case when $G$ splits, we note that
$A$, $B$, and $C$ are all divisive subgroups
in a nontrivial amalgam $A *_C B$
and likewise $A$ and $C$ are divisive in $A *_C$.

A divisive subgroup is codimension--$1$ precisely when
there exists $d>0$ such that $\Gamma-\neb_d(H)$ has more than
one $H$--orbit of deep components.

The notions of codimension--$1$ and divisive subgroups
are not equivalent
but are occasionally confused in the literature.
Let $n$ denote the number of deep components in
$\Gamma - \neb_d(H)$.
In the special case when $1 < n < \infty$,
the divisive subgroup $H$ has a finite index subgroup $H'$ that is a
codimension--$1$ subgroup of $G$.
Indeed $H$ acts on the collection of deep components
and we can let $H'$ be the kernel of the resulting permutation
homomorphism.

It is often the case that $n$ increases with an increase in $d$.
As mentioned in the introduction, $n$ is stably equal to $2$ for the situation arising
from a codimension--$1$ submanifold.
Recently Caprace--Przytycki explored ``bipolar'' Coxeter groups,
which have the property that $n$ is stably equal to $2$ for the
divisive subgroups stabilizing the walls of the generating reflections
\cite{CapracePrzytycki2011}.

There is a more sensitive notion called the
\emph{number of ends of the group pair $(G,H)$},
which takes a value between $0$ and $\infty$
(see \cite{Scott77}).
A group is codimension--$1$ precisely when this end invariant
takes the value $e(G,H) \ge 2$.

The property that $H$ is divisive in $G$
is equivalent to saying that the pair
$(G,H)$ has more than one ``filtered end'',
indicated in the literature by $\tilde{e}(G,H) \ge 2$
(see \cite{KrophollerRoller89,GeogheganBook2008}).

\subsubsection{Codimension--$1$ subgroups on the boundary
of divisive subgroups}

We now show that if $G$ has a divisive subgroup $H$,
then $G$ has one or more codimension--$1$ subgroups,
which we refer to as \emph{boundary subgroups} of $H$.
A case of great interest here
is when $H$ is divisive but not codimension--$1$,
in which case these boundary subgroups are all conjugate
to each other in $H$.

For example, when $G = A *_C B$ the codimension--$1$
subgroup $C$ is a boundary
subgroup of the divisive subgroup $A$.
When $G= A *_{C^t=D}$ the divisive subgroup $A$ has
boundary subgroups $C$ and $D$.
When $G = \pi_1(S)$ for a surface $S$,
and $\hat{S} \to S$ is an infinite degree cover with a core $T$
having a unique boundary circle $\boundary T$,
the divisive subgroup $\pi_1(\hat{S})$
has boundary subgroup $\pi_1(\boundary T)$.

\begin{thm}
Let $G$ be a finitely generated group.
Every divisive subgroup of $G$ contains a codimension--$1$
subgroup of $G$.
Consequently, $G$ contains a divisive subgroup
if and only if $G$ contains a codimension--$1$ subgroup.
\end{thm}

The \emph{frontier} $\boundary S$
of a subset $S \subseteq G$ is the set of vertices
of $S$ that are adjacent to a vertex of $G - S$ in $\Gamma$.

\begin{proof}
Choose a neighborhood $\neb_r(H)$ such that $\Gamma - \neb_r(H)$
has at least two deep components.
For each deep component $\Lambda$, let $K=\Stab_H(\Lambda)$.
Let $E$ denote the union of all the edges
that join $\Lambda$ to $\neb_r(H)$.
Note that $E$ is contained
in the $1$--neighborhood of the frontier of $\Lambda$,
and $\Gamma - E$ has at least two $K$--orbits of deep
components: namely $\Lambda$ and $\Gamma - (\Lambda \cup E)$.
The latter is $K$--deep since it is $H$--deep and $K\subset H$.

Observe that $H$ acts cocompactly on the frontier
$\boundary \bigl(\Gamma-\neb_r(H)\bigr)$,
and $H$ preserves the partition of the frontier into
$\sqcup \boundary \Lambda_i$ where $\{\Lambda_i\}$
is the collection of all components of $\Gamma - \neb_r(H)$.
Consequently each $K_i = \Stab_H(\Lambda_i)$ acts cocompactly
on $\boundary \Lambda_i$.
\end{proof}

\subsection{Wallspaces from codimension--$1$ subgroups}
\label{sub:WallspaceFromCondimension1}

Sageev gave a construction in \cite{Sageev95}
which produces a $G$--action on a $\CAT(0)$ cube complex
from a collection of codimension--$1$
subgroups $H_i$ of the finitely generated $G$.
His construction has two steps:
one first produces a system of walls associated to the subgroups,
and one then produces a cube complex dual to the system of walls.

A codimension--$1$ subgroup $H$ can be used to produce a wall in several
ways:
for instance, Sageev chose the wall $\{HK,\Gamma-HK\}$
where $K$ is a fixed deep complementary component of $\neb_d(H)$.
Alternately one could decompose $\Gamma - \neb_d(H)$
as a disjoint union $K \sqcup K'$ of
disjoint $H$--invariant subspaces that are unions of components.
We then obtain the wall:
\begin{equation}
\label{eqn:Alternate}
\tag{$\ddag$}
   \bigl\{ \neb_d(H) \cup K,
   \ \neb_d(H) \cup K' \bigr\}.
\end{equation}
While it is of greatest interest to insist that both $K$ and $K'$ are deep,
the alternate decomposition in (\ref{eqn:Alternate})
leads to a functional wall---even without assuming that $K$ or $K'$
are deep (in particular, $H$ need not be codimension--$1$).
We will examine such decompositions in
Section~\ref{sub:WallspaceOnGroup}.

A collection $H_1,\dots,H_k$ of subgroups of $G$
together with choices
$K_i,K'_i$ for each $H_i$
leads in this way to a collection of walls in $\Gamma$,
which we extend to a $G$--equivariant collection.
This transforms $\Gamma$, and hence $G$, into a wallspace.

\subsection{A wallspace structure for a group}
\label{sub:WallspaceOnGroup}

In this subsection, we provide a framework for endowing a group $G$
with a wallspace structure associated to finitely may ways of
cutting $G$ along subgroups $H_i$.
This provides important examples of wallspaces
generalizing the construction in
Section~\ref{sub:WallspaceFromCondimension1}.
We will use this material in Section~\ref{sec:RelCocompact}.

\begin{defn}[$H$--wall]
\label{def:HWall}
Let $H$ be a subgroup of $G$.
An \emph{$H$--wall} $\{U,V\}$ is a pair of subsets satisfying the following
conditions:
  \begin{enumerate}
  \item $G= U \cup V$
  \item $\{U,V\}$ is $H$--invariant,
  in the sense that $\{h U,h V\} = \{U,V\}$ for each
  $h \in H$.
  Note that an element might interchange the subsets.
  \item $U \cap V$ is $H$--finite
  \item $U$ and $V$ have $\ddot H$--finite frontiers,
where $\ddot H$ is the subgroup of $H$ (of index at most two)
stabilizing both $U$ and $V$.
  \end{enumerate}

The \emph{stabilizer} of a wall $\{U,V\}$ is the set
$\set{g \in G}{\{gU,gV\}=\{U,V\}}$.
Let $\bar{H}$ denote the stabilizer of the
$H$--wall $\{U,V\}$.
Note that
$H \le \bar{H}$ is finite index when $\{U,V\}$ is nonvacuous.
(The exceptional vacuous $H$--wall $\{G,\emptyset\}$ has $\bar{H}=G$.)
Indeed, $\bar{H}$ stabilizes the union of the
frontiers of $U$ and $V$, and hence stabilizes a collection
of finitely many left cosets of $\ddot{H}$.
When $\{U,V\}$ is nonvacuous, this is a \emph{nonempty}
collection of left cosets.
It thus follows that
$\bar{H}$ is commensurable with $H$.
\end{defn}

\begin{defn}[Wallspace from $H$--walls]
\label{defn:WallsFromHalfspacePairs}
Let $H_1,\dots,H_r$ be subgroups of a finitely generated group $G$.
For each $i$, choose an $H_i$--wall $\{U_i,V_i\}$.
(We do allow repeated subgroups
in the sense that $H_i$ might equal $H_j$ for $i\ne j$.
However, we do not allow repeated walls that are nonvacuous partitions.
So if $U_i \cap V_i = \emptyset$ and $\{U_i,V_i\}=\{U_j,V_j\}$
then either $i=j$ or $U_i,V_i \ne \emptyset$.)

These choices determine a $G$--wallspace whose underlying set is $G$ and
whose walls are $\{gU_i,gV_i\}_i$.
Walls $\{gU_i,gV_i\}_i$ and $\{g'U_j,g'V_j\}_j$
are considered to be equal
if their underlying pairs of halfspaces are equal and $i=j$.
$G$ acts on the wallspace by $g\{U,V\}_i = \{gU,gV\}_i$.
Note that the way in which walls are indexed by their
associated subgroups remains invariant by the $G$--action.
We shall sometimes suppress these indices.

Let us now verify that this construction produces a wallspace.
We first check that each $x \in G$ is betwixted by finitely many walls.
If $gx$ and $g'x$ are in the same $H_i$--orbit, then
$gx=hg'x$ for some $h \in H_i$. Thus $H_i g = H_i g'$.
Since $U_i \cap V_i$ contains finitely many $H_i$--orbits,
there are finitely many right cosets $H_i g$ such that $gx \in U_i \cap V_i$.
If $x$ is betwixted by a wall $\{ g^{-1}U_i,g^{-1}V_i\}_i$ then
$gx \in U_i \cap V_i$.
Therefore the elements $g$ lie in only finitely many right cosets,
i.e., $H_i$--orbits.
Since $\{U_i,V_i\}$ is $H_i$--invariant,
for each $i$ there are only finitely many walls betwixting $x$.

We now verify that $\#(x,y)<\infty$ for all $x,y \in G$.
Here we use that $G$ is finitely generated.
Let $\gamma$ be a path in the Cayley graph from $x$ to $y$.
Since $U_i$ has $\ddot{H}_i$--finite frontier,
we see that finitely many $\ddot{H}_i g$--translates of $\gamma$
intersect the frontier of $U_i$.
If $x$ and $y$ are separated by a wall $\{g^{-1}U_i,g^{-1}V_i\}_i$,
then $g \gamma$ intersects the frontier of $U_i$.
As above there are finitely many such possibilities for each $i$.
Consequently $\#(x,y) < \infty$ for all $x,y$.

There are two ways for distinct walls arising in this construction to
correspond to an identical pair of halfspaces of $G$.
Suppose $\{gU_i,gV_i\} = \{g'U_j,g'V_j\}$ and $i\ne j$.
If $\{U_i,V_i\} \ne \{G,\emptyset\}$ then $\bar{H}_i$ and $\bar{H}_j$
are conjugate, and hence $H_i$ and $H_j$ are commensurable within $G$.
If $\{U_i,V_i\}=\{G,\emptyset\}$ then $H_i$ and $H_j$
can be arbitrary subgroups of $G$.
\end{defn}

\begin{rem}
$H$--walls are related to ``$\ddot{H}$--almost-invariant subsets''
(see e.g., \cite{NibloRoller98}).
A nonvacuous
$H$--wall $\{U,V\}$ is the same as a pair of $\ddot{H}$--almost-invariant sets $U$ and $V$ that cover $G$ and whose intersection is $\ddot{H}$--finite.
If $U$ is an $H$--almost-invariant subset of $G$, then $\{U,G-U\}$ is an $H$--wall.
Conversely, if $\{U,V\}$ is a nonvacuous
$H$--wall then $U$ and $V$ are $\ddot{H}$--almost-invariant sets whose intersection is
$\ddot{H}$--finite.
\end{rem}

\section{The dual $\CAT(0)$ cube complex}
\label{sec:Cubulating}

\subsection{$\CAT(0)$ cube complexes}

\begin{defn}
An \emph{$n$--cube} is a copy of $[-1,1]^n$.
Restricting $i$ of its coordinates to $\pm 1$ yields a \emph{subcube}
which can be identified with an $(n-i)$--cube in various ways.
A \emph{cube complex} is obtained by gluing cubes along subcubes.
(The actual gluing maps are modeled on isometries but are determined completely by the combinatorial data.)

A \emph{flag complex} is a simplicial complex with the property that
$n+1$ vertices span an $n$--simplex if and only if they are pairwise adjacent.

The \emph{link} of a $0$--cube $c \in C^0$
is the complex associated to the ``$\epsilon$--sphere'' about $c$ in $C$,
as illustrated in Figure~\ref{fig:CubeComplexAndLinks}.
More precisely, $\text{link}(c)$ has an $n$--simplex for each corner
of $(n+1)$--cube at $c$, where such corners are glued together precisely
according to the way their associated cubes are glued together.

\begin{figure}
\begin{center}
\includegraphics[width=.7\textwidth]{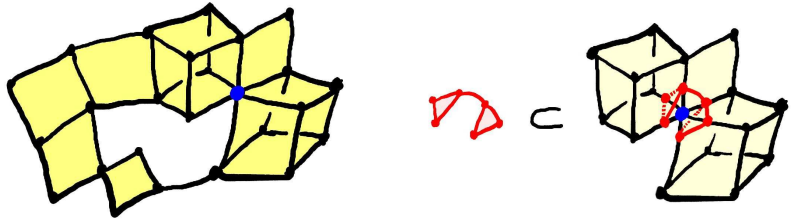}
\end{center}
\caption{On the left is a cube complex.
On the right the link of a vertex is shown.}
\label{fig:CubeComplexAndLinks}
\end{figure}

A cube complex $C$ is \emph{nonpositively curved} if $\text{link}(c)$
is a flag complex for each $c \in C^0$.

Finally, $C$ is $\CAT(0)$ if $C$ is simply connected and nonpositively curved.
Using a slightly different definition,
Gromov introduced nonpositively curved cube complexes as a source of examples
of $\CAT(0)$ metric spaces.
It is a fact that when $C$ is simply connected and satisfies the local nonpositive curvature condition, then $C$ has a $\CAT(0)$ metric
where each $n$--cube is isometric to the standard Euclidean $n$--cube.

This fact was verified by Moussong and Gromov
in the finite dimensional case \cite{Moussong88,Gromov87} and by Leary in
the general case \cite{Leary_KanThurston}.  We will
not call upon this fact until Section~\ref{sec:Truncating}.
\end{defn}

\subsection{The cube complex dual to a wallspace}

Following Sageev, cube complexes dual to wallspaces were constructed
by Niblo--Reeves in \cite{NibloReeves03} where the walls are the
``reflection walls'' naturally considered in the study of Coxeter groups,
and it was exploited by the second author
in \cite{WiseSmallCanCube04} where the walls
are ``hypergraphs'', which are tracks cutting across the $2$--cells in a
$2$--complex.
However the language of wallspaces had not yet been adopted.
Subsequently, the procedure was described
and carefully examined by Nica and Chatterji--Niblo
in the context of cubulating
the spaces with walls of Haglund--Paulin
in \cite{NicaCubulating04,ChatterjiNiblo04}.
We shall now review the construction in the slightly more general
context of wallspaces, as defined in Section~\ref{sub:Wallspaces}.

\begin{const}
\label{const:DualCubeComplex}
We now define the $\CAT(0)$ cube complex $C$ \emph{dual} to a wallspace
$(X,\W)$.
An \emph{orientation $\sigma(W)$ of a wall} $W =\{U,V\}$
is a choice of one of the two ordered pairs:
$(U,V)$ or $(V,U)$.%
\footnote{The reader might wish to review
the formalism advanced in Remark~\ref{rem:Bourbaki}.
In that setting, an orientation of $W$ is an ordered pair
of indexed halfspaces.
This corresponds to a choice of $(-,+)$ or $(+,-)$.
Following this formalism, a wall $W$
whose two halfspaces $H_{-W}$ and $H_{+W}$
are both equal to $X$ actually has two orientations.}
We use the notation $\sigma(W) =
\bigl(\overleftarrow{\sigma}(W),\overrightarrow{\sigma}(W)\bigr)$.

An \emph{orientation $\sigma$ of the wallspace $\W$} is a choice of orientation
$\sigma(W)$ for each wall $W \in \W$.
We emphasize that duplicated walls need not be oriented
in the same way by $\sigma$.

A $0$--cube $c^o$ of $C$ is an orientation of the wallspace
that satisfies the following conditions:
\begin{enumerate}
\item \label{item:Intersect}
$\overleftarrow{c^o}(W) \cap \overleftarrow{c^o}(W') \ne \emptyset$ for all $W,W'\in\W$.
\item \label{item:Finite}
For each $x \in X$, we have $x \in \overleftarrow{c^o}(W)$
for all but finitely many $W\in\W$.
\end{enumerate}

Two $0$--cubes are connected  by a $1$--cube $c^1$ if
there is a unique wall $W$ to which they assign opposite orientations.
In this case, $c^1$ is \emph{dual} to the wall $W$.

For $n\ge 2$, we add an $n$--cube whenever its $(n-1)$--skeleton is present.
\end{const}

\begin{figure}
\labellist
\small\hair 2pt
\pinlabel \color{blue}{$a$} at 10 87
\pinlabel $b$ at 33 114
\pinlabel $c$ at 60 80
\pinlabel $d$ at 118 32
\pinlabel $e$ at 83 54
\pinlabel $f$ at 17 40
\pinlabel ${ \bigl\{ \{a,b,f\}, \{b,c,d,e\} \bigr\} }$ at 71 135
\pinlabel ${ \bigl\{ \{a,b\}, \{a,c,d,e,f\} \bigr\} }$ [bl] at 72 100
\pinlabel ${ \bigl\{ \{a,b,c,e,f\}, \{d,e\} \bigr\} }$ [l] at 116 85
\pinlabel ${ \bigl\{ \{a,b,c,e\}, \{d,f\} \bigr\} }$ [l] at 124 52
\pinlabel ${ \bigl\{ \{a,b,c,e,f\}, \{d,e\} \bigr\} }$ [bl] at 109 -1
\endlabellist
\begin{center}
\includegraphics[width=.4\textwidth]{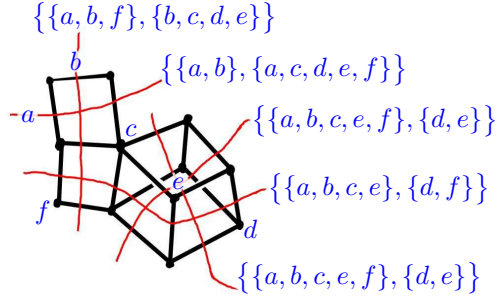}
\end{center}
\caption{A wallspace $(X,\W)$ and its dual cube complex
$C(X,\W)$ are indicated.
The set $X= \{a,b,c,d,e,f\}$
and $\W$ consists of the five walls indicated above.
Note that $a$ betwixts the second wall,
and $e$ betwixts the third and fifth walls.
The fourth wall is a genuine partition, while the others are not.
The third and fifth walls are duplicates.
We have drawn $C(X,\W)$ superimposed on $(X,\W)$
so that the reader can see the dual relationship.
In fact, $X$ can be regarded as a set of points in $C(X,\W)$.
Namely $a,b$ are centers of $1$--cubes,
$c,d,f$ are $0$--cubes, and $e$ is the center of a $3$--cube.
The pairs of closed halfspaces of $C(X,\W)$
induce the given wallspace structure on $X$.}
\label{fig:DummyExample}
\end{figure}

We refer the reader to Figure~\ref{fig:DummyExample}
for an example of a wallspace and its dual cube complex.

\begin{rem}
It is sometimes the case that there are no
``duplicate halfspaces'' in the sense that
no halfspace is associated to more than one wall.
Indeed, this holds in the
setting of Haglund--Paulin,
where the walls are distinct partitions of $X$.
In this case, a $0$--cube is equivalent to a certain
subcollection of the set of all halfspaces,
namely a choice of one halfspace from each wall.
(This is a separate issue from duplicate walls,
which were discussed in Remark~\ref{rem:Duplicity}.)
\end{rem}

Consideration of adjacent $0$--cubes leads to the following definition,
which plays a role in a number of fundamental proofs.

\begin{defn}
\label{def:PartialOrder}
We define a partial order $\preceq$ on $2^X \times 2^X$
as follows: $\bigl( U,V \bigr) \preceq
\bigl( U',V' \bigr)$ if either:
$U \subsetneq U'$ or: $U = U'$ and
$V \supseteq V'$.
As usual the notation $\prec$ indicates $\preceq$ but not equal.
\end{defn}

Let $c$ be a $0$--cube, and
consider the set
\[
   S \ = \ \bigset{c(W)}{W \in \W}
   \ = \
 \bigset{\bigl(\overleftarrow{c}(W),\overrightarrow{c}(W)\bigr)}{W \in \W}.
\]
If the ordered pair $c(W)$ is $\preceq$--minimal
in $S$, then $W$ is dual to a $1$--cube adjacent to the $0$--cube $c$.
However the converse does not hold in general---e.g.,
when $X = \{1,2,3\}$ and $\W =
\bigl\{ \{ \{1,2\},\{1,2,3\} \}, \{ \{2,3\},\{1\}\} \bigr\}$.
We frequently use the contrapositive of this fact, namely:
if a wall $W$ is not dual to an adjacent $1$--cube, then
$c(W)$ is not $\preceq$--minimal.

\begin{rem}[Boundary components]
\label{rem:BoundaryComponents}
If we were to omit Condition~(2) from the definition of a $0$--cube,
Construction~\ref{const:DualCubeComplex} would produce
a disjoint union of $\CAT(0)$
cube complexes, of which the dual cube complex
$C$ is one connected component.
The other components are \emph{boundary components}
and can be associated with a boundary at infinity
in a manner that has not yet been fully investigated
(see for instance Roller \cite{RollerPocSets} and
Guralnik \cite{GuralnikBoundaries,GuralnikLocalFiniteness}).
Some interesting applications towards the Poisson boundary
are explored by Nevo--Sageev in \cite{NevoSageevBoundary}.
We emphasize that Propositions
\ref{prop:SimplyConnected}~and~\ref{prop:NPC}
apply to each boundary component.

We recommend that the reader consider the example of a wallspace
$(X,\W)$
arising from two infinite systems of parallel lines in $\R^2$.
When we omit Condition~(2), Construction~\ref{const:DualCubeComplex}
provides a cube complex consisting of nine components,
each of which is a $\CAT(0)$ cube complex.
One of these components is the dual cube complex $C(X,\W)$.
\end{rem}

\begin{defn}[Canonical cubes]
Let $x_0 \in X$.
Let $\set{W_j}{j \in J}$ be the set of walls of $\W$ that betwixt $x_0$,
and let $\set{W_i}{i\in I}$ be the set of walls of $\W$ that do not
betwixt $x_0$.
For each $i \in I$, we orient $W_i$ so that its left halfspace
contains $x_0$.
Observe that $J$ is finite, by the definition of wallspace.
The $2^{\abs{J}}$ distinct orientations of $\set{W_j}{j\in J}$
determine a $\abs{J}$--cube.
This cube is the \emph{canonical cube} corresponding to $x_0$.

In the case where all walls are genuine partitions of $X$,
all canonical cubes are $0$--cubes,
and these have been termed ``canonical vertices'' in the literature.
\end{defn}

\begin{thm}
$C$ is a $\CAT(0)$ cube complex.
\end{thm}

\begin{proof}
We follow Sageev \cite{Sageev95} here.  We must show that $C$ is connected,
that $\pi_1(C)=1$, and that $C$ is nonpositively curved.
These are verified in the following three propositions.
\end{proof}

\begin{prop}
\label{prop:Connected}
$C$ is connected.
\end{prop}

Proposition~\ref{prop:Connected}
would be false if there are duplicate nonvacuous walls
that are genuine partitions $\{U,{X-U}\}$
(see Remark~\ref{rem:Duplicity}).
Indeed
if $u \in U$ and $v \in X-U$ and there are two walls
of the form $\{U,X-U\}$, then there would be no path
between the canonical cubes corresponding to $u$ and $v$.
Therefore $C$ would be disconnected.
Vacuous walls must always be oriented towards $X$.
Thus their presence has no effect on $C$.

\begin{proof}[Proof of Proposition~\ref{prop:Connected}]
It is critical that there are no duplicate walls that are genuine partitions (see Remark~\ref{rem:Duplicity}).
Choose a basepoint $x_0 \in X$.
Let $d$ be an arbitrary $0$--cube of $C$.
Let $W_1,\dots,W_n$ be the walls
such that $x_0 \notin \overleftarrow{d}(W_i)$.

Consider the ordered pairs $\bigl\{ d(W_1),\dots,d(W_n) \bigr\}$.
One of these, say $d(W_1)$, must be $\preceq$--minimal
in $\bigl\{ d(W_1),\dots,d(W_n) \bigr\}$
in the sense of Definition~\ref{def:PartialOrder}.
As claimed below,
the orientation of $W_1$ can be reversed
to obtain another $0$--cube adjacent to $d$.
Repeating this procedure $n$ times, we eventually reach a $0$--cube $c$
such that $x_0 \in \overleftarrow{c}(W)$ for all $W \in \W$.
Finally we note that $c$ lies on the canonical cube corresponding to
$x_0$.

\textit{Claim:}
$\overrightarrow{d}(W_1) \cap \overleftarrow{d}(W) \ne \emptyset$
for all $W \in \W$.

We first note that
if $W\notin\{W_1,\dots,W_n\}$ then clearly
$\overrightarrow{d}(W_1) \cap \overleftarrow{d}(W) \ne \emptyset$ since both halfspaces contain $x_0$.

We now show that
$\overleftarrow{d}(W_i) \cap \overrightarrow{d}(W_1) \ne \emptyset$
for $i\in\{2,\dots,n\}$.
There are two cases:

\textit{Case~1:} If
$\overleftarrow{d}(W_i) \not\subseteq \overleftarrow{d}(W_1)$, then
$\overleftarrow{d}(W_i)$ intersects
$\bigl( X - \overleftarrow{d}(W_1) \bigr)$
which lies in $\overrightarrow{d}(W_1)$.
Hence
$\overleftarrow{d}(W_i) \cap \overrightarrow{d}(W_1) \ne \emptyset$.

\textit{Case~2:} If
$\overleftarrow{d}(W_i) \subseteq \overleftarrow{d}(W_1)$,
then we must have
$\overleftarrow{d}(W_i) = \overleftarrow{d}(W_1)$
since $d(W_i) \nprec d(W_1)$.
Arguing by contradiction, suppose that
$\overrightarrow{d}(W_1) \cap \overleftarrow{d}(W_i) = \emptyset$.
Then $\overrightarrow{d}(W_1) \cap \overleftarrow{d}(W_1) = \emptyset$;
in other words, $W_1$ is a genuine partition.
Observe that $W_1$ is not a vacuous wall $\{X,\emptyset\}$
since both of its halfspaces are nonempty.
Indeed, $x_0 \in \overrightarrow{d}(W_1)$ by hypothesis,
and $\overleftarrow{d}(W_1) \cap \overleftarrow{d}(W_i) \ne \emptyset$
by the definition of $0$--cube.
It is impossible for $W_i$ to be a duplicate of $W_1$, since
by hypothesis that there are no duplicate genuine partitions
(with the possible exception of vacuous walls).
Therefore $W_i$ is not (also) a genuine partition;
i.e., $\overleftarrow{d}(W_i) \cap \overrightarrow{d}(W_i)
\ne\emptyset$.
Therefore
\[
   \overrightarrow{d}(W_1) = X - \overleftarrow{d}(W_1)
   = X- \overleftarrow{d}(W_i)
   \subsetneq \overrightarrow{d}(W_i).
\]
Consequently $d(W_i) \prec d(W_1)$,
which contradicts the minimality of $d(W_1)$.
Thus we must have
$\overleftarrow{d}(W_i) \cap \overrightarrow{d}(W_1) \ne \emptyset$.
\end{proof}

\begin{prop}
\label{prop:SimplyConnected}
$C$ is simply connected.
\end{prop}

\begin{proof}
Consider a closed edge path $P=e_1\cdots e_t$.
Choose an ``innermost'' pair of edges $e_p,e_q$
dual to the same wall $W_1$.
By this we mean that there does not exist a pair $(r,s) \ne (p,q)$
with $p\le r < s \le q$ such that $e_r$ and $e_s$ are dual to the same
wall $W$.

If they are consecutive then $e_p = e_q^{-1}$, and
we homotope by removing a backtrack.

Otherwise let $W_2$ be the wall dual to $e_{p+1}$.
We show below that $W_1$ and $W_2$ are transverse.
Consequently the orientations of
$W_1,W_2$ can be independently reversed
at the terminal $0$--cube $c_{++}$ of $e_p$ to provide a
$4$--cycle $e_p e_{p+1} (e'_p)^{-1} (e'_{p+1})^{-1}$
in the $1$--skeleton,
which bounds a $2$--cube by construction.
Specifically, its four $0$--cubes are $c_{\pm \pm}$, where
$c_{+-}$ is the initial $0$--cube of $e_p$ and $c_{-+}$ is the terminal
$0$--cube of $e_{p+1}$ and $c_{--}$ is obtained from
$c_{++}$ be reversing the orientations of both $W_1$ and $W_2$.

We are thus able to homotope $e_p e_{p+1}$ to $e'_{p+1} e'_p$.
The resulting path
\[
   e_1 \cdots e_{p-1} e'_{p+1} e'_p e_{p+2} \cdots e_t
\]
contains a shorter innermost pair of edges dual to $W_1$.
Continuing this process and removing backtracks, we eventually arrive at
a trivial path.

To see that $W_1$ and $W_2$ are transverse,
observe that
since the edges $e_{p+1},\dots,e_{q-1}$ are dual to distinct walls,
the edge $e_{p+1}$ is the only one dual to $W_2$.
Let $x$ and $y$ be the initial and terminal $0$--cubes of $e_p$.
Let $z$ and $w$ be the initial and terminal $0$--cubes of $e_q$.
Then $w,x,y,z$ assign orientations to the walls $W_1$ and $W_2$
in all four possible combinations.
By the definition of $0$--cube, all four intersections of
halfspaces are therefore nonempty.
Therefore $W_1$ and $W_2$ are transverse.
By the definition of the dual cube complex $C$,
the $1$--cubes $e_p,e_{p+1}$ dual to $W_1,W_2$ form the corner of a
$2$--cube at $y$.  The path $e'_{p+1} e'_p$ is at the opposite corner.
\end{proof}

\begin{prop}
\label{prop:NPC}
$C$ is nonpositively curved.
\end{prop}

\begin{proof}[Sketch]
We first observe that $\text{link}(c)$ is a \emph{simplicial} complex.
The main point is that only a single $i$--cube is added for each $i$--fold
collection of pairwise crossing walls.

Consider $n$ pairwise adjacent vertices in $\text{link}(c)$.
Their associated $1$--cubes are dual to $n$ pairwise crossing walls.
These $n$ vertices span an $(n-1)$--simplex that is a corner of the cube
associated to these $n$ walls (and this $0$--cube).
\end{proof}

\begin{rem}[The induced action on $C(X,\W)$]
\label{rem:InducedAction}
In conclusion, we note that if $G$ acts on a wallspace $(X,\W)$,
there is an induced action on $C(X,\W)$.
The action on the $0$--skeleton is given by
$\overleftarrow{gc} (gW) = g \overleftarrow{c}(W)$
or equivalently $\overleftarrow{gc} (W) = g \overleftarrow{c}(g^{-1}W)$.
We note that the action preserves the correspondence between
walls and hyperplanes.
\end{rem}

\subsection{Cubes}
\label{subsec:Cubes}

In this subsection we examine the cubes of the dual cube complex
in more detail.
It is easiest to describe the maximal cubes of $C(X,\W)$,
which correspond to finite cardinality
maximal collections of pairwise transverse walls.
The description of an arbitrary cube is a bit more subtle
since an arbitrary cube is not always contained in a maximal cube.
Instead we find that an arbitrary cube corresponds to a certain
collection of halfspaces.
Much of what we discuss holds for the cubes
in the boundary components mentioned in
Remark~\ref{rem:BoundaryComponents}.
Indeed one is led to these boundary components by considering infinite
cardinality maximal collections of pairwise transverse walls.

Consider a $k$--dimensional cube $c$ in $C(X,\W)$,
and let $d$ be a $0$--cube of $c$.
The $k$ distinct $1$--cubes at the corner of $c$ containing $d$
are dual to pairwise transverse walls $\{W_1,\dots,W_k\}$.
Traveling from $0$--cube to $0$--cube in the $1$--skeleton of $c$
corresponds to reversing the orientations of these walls,
whilst preserving the orientations of all other walls in $\W$.
We refer to $\{W_1,\dots,W_k\}$ as the \emph{independent}
walls of $c$ and the remaining walls as \emph{dependent} walls.
See Figure~\ref{fig:IndependentWalls} for an example.
In Section~\ref{subsec:Hemi} we introduce the notion
of a hemiwallspace, which generalizes this viewpoint on cubes
to interpret arbitrary convex subcomplexes.

\begin{figure}[ht]
\begin{center}
\includegraphics[width=.4\textwidth]{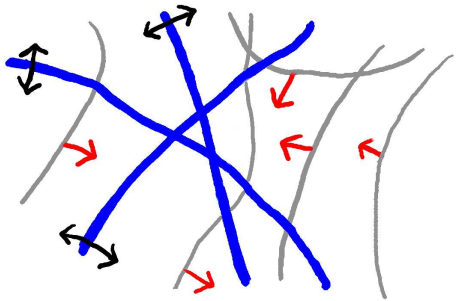}
\end{center}
\caption{A collection of walls that corresponds to a $3$--cube
in $C(X,\W)$.  The figure shows $3$ independent walls and $5$ dependent walls.
Note that this $3$--cube lies in a $4$--cube.}
\label{fig:IndependentWalls}
\end{figure}

We note that a $0$--cube has no independent walls---indeed
each wall in $\W$ is assigned a fixed orientation.
A $1$--cube has a single independent wall, namely the wall dual to it.

We emphasize that the data determining the $k$--cube $c$
requires a fixed choice of orientation for each dependent wall,
but allows the orientation to vary for the $k$ independent walls
to reach all $2^k$ of its $0$--cubes.
The fixed orientations of the dependent walls must be chosen to
satisfy Conditions (1)~and~(2) of the definition of $0$--cube.
There is always at least one way to make these fixed choices
(see Remark~\ref{rem:WallsAndPoint-Cube}).
But often there is more than one way; in particular, there is more than one
way if and only if
$\{W_1,\dots,W_k\}$ is not a maximal collection of pairwise transverse
walls.

An orientation $d$ of a wall $W$ is \emph{towards} a point $x \in X$
if $x \in \overleftarrow{d}(W)$.
More generally $d$ orients $W$ \emph{towards} a subset $S \subseteq X$
if $d$ orients $W$ towards some point of $S$.
We say that $d$ orients $W$ \emph{towards} a wall $W'$
if $d$ orients $W$ towards each halfspace of $W'$.

If $c$ is a cube of $C(X,\W)$, then
each dependent wall $W$ of $c$ has a well-defined orientation by $c$.
This dependent wall is oriented by $c$
towards each independent wall of $c$ and towards the chosen halfspace of
each dependent wall.
The reader may with to consider this language in conjunction with
Figure~\ref{fig:IndependentWalls}.

\begin{rem}
\label{rem:WallsAndPoint-Cube}
Given a finite collection of pairwise transverse walls
$W_1,\dots,W_k$ and a point $p \in X$,
there is an associated cube $c$ in $C(X,\W)$.
The independent walls of $c$ are precisely the walls
$W_1,\dots,W_k$ together with those walls $W$ that betwixt $p$ and
are transverse to every $W_i$.
If $W$ is a dependent wall that is not transverse with some $W_i$,
then $c$ orients $W$ towards $W_i$.
If $W$ is transverse with every $W_i$ and does not betwixt $p$,
then $c$ orients $W$ towards $p$.
We leave it to the reader to verify that the choices above
do indeed determine a cube in $C(X,\W)$.
A similar conclusion is explicitly verified in
Lemma~\ref{lem:RepresentingFiniteMaximal}.
\end{rem}

Remark~\ref{rem:WallsAndPoint-Cube} leads to the following:

\begin{cor}
\label{cor:FiniteDim}
The dual cube complex $C(X,\W)$ is finite dimensional
if and only if there is a finite upper bound on the size of collections
of pairwise transverse walls. \qed
\end{cor}

Having interpreted arbitrary cubes in $C(X,\W)$ in terms of independent
and dependent walls, we now examine the (finite-dimensional)
maximal cubes.

\begin{prop}
\label{prop:MaximalCorrespondence}
Maximal cubes in the dual cube complex are in one-to-one
correspondence with finite maximal collections of pairwise transverse
\textup{(}nonvacuous\textup{)} walls in a wallspace.
\end{prop}

The proof of Proposition~\ref{prop:MaximalCorrespondence}
is broken into
Lemmas \ref{lem:CubeToTransverseWalls}
and~\ref{lem:RepresentingFiniteMaximal} below.

\begin{lem}[Descending chain condition]
\label{lem:NoInfiniteChains}
Let $c$ be a $0$--cube of the dual cube complex $C(X,\W)$.
Then the set $\bigset{c(W)}{W \in \W}$ does not contain
an infinite properly descending chain with respect to $\preceq$.
\end{lem}

\begin{proof}
Suppose there is an infinite properly descending chain:
\[
   \bigl(\overleftarrow{c}(W_1),\overrightarrow{c}(W_1)\bigr)
   \,\succ\, \bigl(\overleftarrow{c}(W_2),\overrightarrow{c}(W_2)\bigr)
   \,\succ\, \cdots
\]
The first coordinates cannot be an infinite properly descending
chain
\[
   \overleftarrow{c}(W_1) \supsetneq \overleftarrow{c}(W_2) \supsetneq \cdots
\]
because such a chain would contradict Condition~(2)
from the definition of a $0$--cube indicated in
Construction~\ref{const:DualCubeComplex}.
Indeed if $x \in \overleftarrow{c}(W_1) - \overleftarrow{c}(W_2)$, then $x \notin \overleftarrow{c}(W_i)$ for all
$i \ge 2$.  This contradicts the fact that $c$
satisfies Condition~(2).

It follows that if there is an infinite
properly descending chain, then there must be one of the form:
\[
   \bigl(H,\overrightarrow{c}(W_1)\bigr) \succ \bigl(H,\overrightarrow{c}(W_2)\bigr)
   \succ \cdots
\]
where $H = \overleftarrow{c}(W_i)$ for all $i$.
In particular $\overrightarrow{c}(W_1) \subsetneq \overrightarrow{c}(W_2) \subsetneq \cdots$.
Letting $x \in H \cap \overrightarrow{c}(W_2)$, we see that $x$ betwixts $W_i$ for all
$i \ge 2$.
This contradicts the definition of a wallspace.
\end{proof}

\begin{lem}
\label{lem:CubeToTransverseWalls}
For each maximal cube $c$ of the dual cube complex $C(X,\W)$,
the collection of independent walls is a maximal family of
pairwise transverse walls.
\end{lem}

\begin{proof}
Let $c$ be a maximal cube.
Observe that if $W$ is a dependent wall of $c$, then
the orientation $c(W)$ is well-defined.
Consider the set
\[
   \mathcal{S} = \bigset{c(W)}{ \text{$W$ is dependent in $c$ and transverse to every independent wall of $c$} }.
\]
Suppose by way of contradiction that $\mathcal{S}$ is nonempty.
Applying Lemma~\ref{lem:NoInfiniteChains} to any $0$--cube $d$
of $c$,
we see that $\mathcal{S}$ must have a minimal element $c(W_0)$
with respect to $\preceq$.

Since $c$ is maximal, the orientation of $W_0$ cannot be reversed,
which means there is another wall $W_1$ so that
$\overrightarrow{c} (W_0) \cap \overleftarrow{c}(W_1) = \emptyset$.
Observe that $W_0$ is transverse to each independent wall, but is not transverse to $W_1$.
Therefore $W_1$ is dependent.
Since $X - \overleftarrow{c}(W_0) \subseteq \overrightarrow{c}(W_0)$,
it follows that $\overleftarrow{c}(W_1) \subseteq \overleftarrow{c}(W_0)$.
If $\overleftarrow{c}(W_1)=\overleftarrow{c}(W_0)$, then $W_0 = \bigl\{\overleftarrow{c}(W_0), \overrightarrow{c}(W_0)\bigr\}$ is a partition
of $X$.
Now our hypothesis that there are no duplicate genuine partitions
implies that $\overrightarrow{c}(W_0) \subsetneq \overrightarrow{c}(W_1)$.
In particular, we have shown that
$\bigl(\overleftarrow{c}(W_1),\overrightarrow{c}(W_1)\bigr) \prec \bigl(\overleftarrow{c}(W_0),\overrightarrow{c}(W_0)\bigr)$.
In other words, $c(W_1) \prec c(W_0)$.

We will now show that $W_1$ is transverse to every independent wall of $c$.
Indeed let $W$ be independent.  Then both halfspaces of $W$ intersect
$\overleftarrow{c}(W_1)$.
Since $\overrightarrow{c}(W_0) \cap \overleftarrow{c}(W_1) = \emptyset$,
we have
$\overrightarrow{c}(W_0) \subseteq X- \overleftarrow{c}(W_1)
\subseteq \overrightarrow{c}(W_1)$.
But $W$ is transverse with $W_0$, so both halfspaces of $W$
intersect $\overrightarrow{c}(W_0)$, and hence $\overrightarrow{c}(W_1)$.
Therefore $W$ is transverse to $W_1$.
Since $W$ is arbitrary, it follows that $W_1$ is transverse to
every independent wall of $c$,
contradicting the fact that $W_0$ was the minimal such wall.
\end{proof}

\begin{lem}
\label{lem:RepresentingFiniteMaximal}
Each finite maximal collection $\{W_1,\dots,W_k\}$
of pairwise transverse nonvacuous
walls is associated to a maximal cube in the dual cube complex
$C(X,\W)$.
Specifically this cube is the unique cube of $C(X,\W)$ whose independent
walls are $\{W_1,\dots,W_k\}$.
\end{lem}

It is interesting that the associated maximal cube actually arises
in $C(X,\W)$ and \emph{not} in one of the boundary components.
This contrasts with the infinite case:
An infinite maximal collection of pairwise transverse walls
could be associated with an infinite increasing sequence of cubes in
many ways.
One of these might lie in $C(X,\W)$
but uncountably many lie in distinct boundary components.
While $C(X,\W)$ may contain some infinite increasing sequences of cubes,
there could be maximal collections of pairwise transverse walls that are
only associated with boundary components and not associated at all with
$C(X,\W)$
(see Figure~\ref{fig:MaximalNotCanonical}).

\begin{figure}
\begin{center}
\includegraphics[width=.4\textwidth]{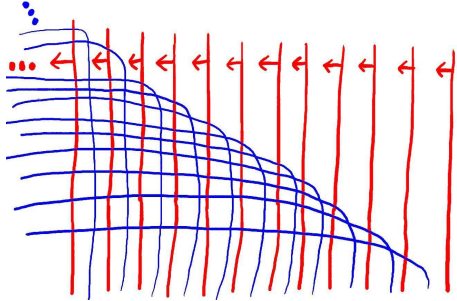}
\end{center}
\caption{A maximal collection of pairwise transverse walls
in the plane $X=\R^2$ that is not associated to an infinite increasing
sequence of cubes in the canonical component.}
\label{fig:MaximalNotCanonical}
\end{figure}

\begin{proof}[Proof of Lemma~\ref{lem:RepresentingFiniteMaximal}]
Suppose $W_1,\dots,W_k$ is a finite collection of pairwise transverse
nondegenerate walls.
We consider these to be independent and all other walls to be dependent.
Let $x \in X$ be a fixed point.
We define an orientation $c$ of $\W$ as follows.
Orient the independent walls towards $x$.
(Due to betwixting, there may be more than one way of doing this.)
For each $i$ choose $x_i \in \overrightarrow{c}(W_i)$.
(This can be done since all halfspaces are nonempty.)
Let $W$ be a dependent wall.  Then there exists at least one independent wall
$W_i$ that is not transverse to $W$.
In particular, there are orientations $\sigma(W_i)$ and $\sigma(W)$
so that $\overleftarrow{\sigma}(W_i) \cap \overrightarrow{\sigma}(W) =\emptyset$.
In particular, we have
$\overleftarrow{\sigma}(W_i) \subseteq \overleftarrow{\sigma}(W)$.
Furthermore $\sigma(W)$ is the unique orientation of $W$
towards the independent wall $W_i$.
We define $c(W) = \sigma(W)$ for each dependent wall $W$.

We now verify that $c$ is a $0$--cube.
Let us first check that each dependent wall is oriented by $c$ towards
the chosen halfspace of each other dependent wall.
Let $W$ and $W'$ be dependent walls.
Then there are independent walls $W_i = \{U_i,V_i\}$ and
$W_j = \{U'_i,V'_i\}$ such that $U_i \subseteq \overleftarrow{c}(W)$
and $U'_i \subseteq \overleftarrow{c}(W')$.
Since $W_i$ and $W_j$ are transverse, it follows that $U_i \cap U'_i
\ne \emptyset$.
Therefore $\overleftarrow{c}(W) \cap \overleftarrow{c}(W') \ne \emptyset$.

If $\overleftarrow{c}(W)$ does not contain $x$, then $\overleftarrow{c}(W_i)$
also does not contain $x$. Consequently, $\overleftarrow{c}(W_i)$
contains $x_i$, and therefore so does $\overleftarrow{c}(W)$.
Therefore $W$ is one of the finitely many walls separating $x$ from $x_i$.
Since there are only finitely many independent walls, all but finitely many dependent walls are oriented towards $x$.
\end{proof}

\subsection{Hemiwallspaces and convex subcomplexes}
\label{subsec:Hemi}

\newcommand{\V}{\mathcal V}
\newcommand{\U}{\mathcal U}

In this section we introduce and study hemiwallspaces, which
are associated to convex subcomplexes of the dual cube complex.
Hemiwallspaces generalize our association of cubes in $C(X,\W)$
with collections of independent and dependent walls in $\W$.

To accommodate this study,
we will shift gears towards a viewpoint of a wallspace
as a collection of halfspaces of $X$.
When equipped with an involution $\iota$
(generalized complementation),
such a collection of halfspaces leads to a very elegant theory,
and Roller has adopted this as the main viewpoint \cite{RollerPocSets}.
However our own viewpoint will be slightly different and aimed
at natural applications in group theory.

When there are no duplicated halfspaces\footnote{To handle duplicated halfspaces,
one must work with indexed collections of halfspaces;
cf.~Remark~\ref{rem:Bourbaki}.},
a wallspace is equivalent to a collection of subsets $\V$
together with a \emph{pairing} $\iota\colon \V \to \V$
that is a fixed point free involution
such that for each $V \in \V$, we have $V \cup \iota(V) = X$,
each $x$ lies in finitely many $V \cap \iota(V)$, and
for each $x,y \in X$ there are finitely
many $V$ such that $x \in V - \iota(V)$ and $y \in \iota(V) - V$.
To emphasize that we are thinking of a collection of halfspaces $\V$
rather than a collection of walls $\W$,
we will often use the notation $(X,\V)$.
Technically we have in mind $(X,\V,\iota)$, but
we shall suppress mentioning $\iota$,
which is sometimes even uniquely determined by $\V$.

We will use this halfspace terminology
throughout this section,
and our arguments and definitions will implicitly adhere to this convention.
We will later return to this viewpoint in Section~\ref{sec:RelCocompact}.
The general case requires few modifications and a slightly
different language---$\iota$ simply exchanges the two indexed halfspaces
of each wall.

Observe that for Haglund--Paulin wallspaces complementation is the
only possible way
of defining the pairing $\iota$ on the collection of halfspaces $\V$,
and so the wallspace is uniquely determined by $\V$.
More generally, we say a collection of halfspaces is \emph{unambiguous}
if there is only one possible way of defining $\iota$ so that $V \cup \iota(V) = X$ for all $V \in \V$.
In practice however, there are two commonly encountered ways
that $\V$ can fail to be unambiguous.
Firstly, if $\V$ contains distinct
halfspaces $U_1,U_2,V_1,V_2$ such that each
$U_i \cup V_j = X$ then $\iota$ might not be uniquely determined by $\V$.
Secondly, there are scenarios where it is natural to consider
a collection $\V$ where certain halfspaces are repeated
(we would then regard $\V$ as an indexed collection of sets).
We emphasize that this second type of failure is
more severe than the first, as it is not simply a matter of suppressing
$\iota$, but rather a matter of the difficulty of defining $\iota$.

\begin{defn}[Hemiwallspace]
\label{def:Hemi}
For a wallspace $(X,\V)$,
an associated \emph{hemiwallspace} is a subcollection $\U$ of halfspaces
of $\V$
such that for each wall $\bigl\{V,\iota(V)\bigr\}$ at least one of $V$ or $\iota(V)$
is contained in $\U$.%
\footnote{In Construction~\ref{const:CubulatingHemi},
we shall also adopt the convention
that some subcollection of $\U$ is a (canonical) $0$--cube of $X$.
Hence hemiwallspaces correspond precisely with convex subcomplexes
of $C(X)$.}
We note that there is a partial pairing of $\U$ induced by $\iota$.

\begin{rem}[Dependent and independent walls]
An alternate approach to defining hemiwallspaces
involves restricting orientations.
We earlier defined an orientation $\sigma$ on a wallspace as a choice
$\sigma(W) \in W=\{U,V\}$
and used this to define a $0$--cube as a particular type of orientation
whose nature was restricted only by two simple axioms.
An associated \emph{hemiwallspace} is equivalent to
a specific way of fixing the orientations of certain walls.
These are the \emph{dependent} walls, and the other walls are the
\emph{independent} walls.
More precisely, a hemiwallspace is an orientation of the set of dependent
walls.
\end{rem}

As we saw in Section~\ref{subsec:Cubes}, a finite set of pairwise transverse walls provide us with a set of independent walls.
Such a set,
together with an appropriate choice of orientations for the remaining
dependent walls,
leads to a cube in the dual cube complex.
Similarly
a hemiwallspace leads to a certain convex subcomplex of the dual cube complex
as we shall see in Lemma~\ref{lem:ConvexSubcomplex}.
\end{defn}

\begin{exmp}
\label{exmp:SubspaceHemi}
For a subset $P\subset X$ we let
\[
   \U_P = \set{U \in \V}{P \cap U \ne \emptyset},
\]
and we regard $(X,\U_P)$ as a hemiwallspace.
We discuss related but more elaborate hemiwallspaces in
Definition~\ref{defn:InducedHemiwallspaces}.

Another structure associated to $P \subseteq X$
is the induced subwallspace described in Definition~\ref{def:Subwallspace}.
We urge the reader to consider the differences between
these.
\end{exmp}

\begin{const}[Cubulating a hemiwallspace]
\label{const:CubulatingHemi}
Let $\U$ be a hemiwallspace in a wallspace $(X,\V)$.
The cube complex $C(\U)$ dual to $(X,\U)$ is defined as follows:
$0$--cubes are subcollections $c \subset \U$ that have nonempty
pairwise intersection such that for each halfspace $V \in \V$
exactly one of $V$ or $\iota(V)$ is in $c$.
Two $0$--cubes are connected by a $1$--cube if they differ on
two complementary halfspaces.
For $n\ge 2$ we add an $n$--cube if its $(n-1)$--skeleton appears.

In practice, the hemiwallspaces we consider are induced from
subspaces of $X$, for instance as in Example~\ref{exmp:SubspaceHemi}.
In that case,
there is always at least one $0$--cube in the dual cube complex $C(\U)$.
[We shall always assume that this is the case for purposes of
the lemmas below, but their statements and proofs hold for
boundary components as well.]
It is then a consequence of Lemma~\ref{lem:ConvexSubcomplex}
that there is a unique connected component $C(\U)$
with the property that for one and hence any $x \in X$
almost all halfspaces of $c$ contain $x$.

In the more general setting where a hemiwallspace corresponds to restricted orientations, our $0$--cubes are simply chosen so that they conform with the restriction determining the hemiwallspace,
and higher cubes are added as before.
\end{const}

\begin{rem}
\label{rem:Subwallspace}
A hemiwallspace $(X,\U)$
naturally yields a wallspace $(X,\ddot{\U})$ by removing all dependent walls;
that is, ignoring all unpaired halfspaces.
The cube complex $C(\ddot{\U})$ dual to $(X,\ddot{\U})$
is isomorphic to the cube complex $C(\U)$ associated above to $(X,\U)$.
In Section~\ref{sec:RelCocompact}, $C(\ddot{\U})$ will be the cubulation
of the peripheral subspace and $C(\U)$ naturally corresponds to
a convex subcomplex of $C(\V)$.
\end{rem}

\begin{lem}
$C(\ddot{\U})$ embeds naturally in $C(\V)$.
\end{lem}

\begin{proof}
By Remark~\ref{rem:Subwallspace} there is a natural
isomorphism $C({\U}) \to C(\ddot{\U})$ (forgetting the unpaired walls,
i.e., the dependent walls).
The natural embedding $C(\U) \to C(\V)$ is a feature of
Construction~\ref{const:CubulatingHemi} (unrestricting the orientations,
i.e., considering all walls to be independent).
\end{proof}

\begin{lem}
\label{lem:ConvexSubcomplex}
$C(\U)$ is a convex subcomplex of $C(\V)$.
Specifically, combinatorial geodesics in $C(\V)$ that start and end
in $C(\U)$ are actually contained in $C(\U)$.
\end{lem}

\begin{proof}
Consider a geodesic edge path $\gamma$ in $C(\V)$ whose endpoints $a$ and $c$
lie in $C(\U)$.
Observe that each $0$--cube of $\gamma$ lies in $C(\U)$.
Indeed for each $0$--cube $b$ of $\gamma$ each halfspace of $b$ is either
a halfspace of $a$ or a halfspace of $c$ (or both).
\end{proof}

We note that $C(\U)$ is also convex in $C(\V)$ with respect to the
$\CAT(0)$ metric since it is an intersection of ``shifted
halfspaces''.
A \emph{shifted halfspace} is the closure of a component of the complement of $\neb_{\epsilon}(H)$ for some hyperplane $H$ and some $\epsilon<1$.
Shifted halfspaces are slightly smaller than ordinary halfspaces
and are convex in the $\CAT(0)$ metric.

\subsection{The dimension of the dual cube complex and uniform local finiteness}

\begin{defn}
A wallspace
has the \emph{$k$--Plane Intersection Property}
if each collection of $k+1$ walls contains a nontransverse pair.
In other words, every collection of walls that is pairwise transverse
has cardinality at most $k$.
\end{defn}

Equivalently, a wallspace has the $k$--Plane Intersection Property
if and only if
the corresponding $\CAT(0)$ cube complex $C$ has dimension at most $k$.

The $k$--Plane Intersection Property generalizes a property
introduced by Scott \cite{Scott83} in the setting of surfaces in
$3$--manifolds and studied further by Hass--Scott \cite{HassScott92} and
Rubinstein--Sageev \cite{RubinsteinSageev99}.
Let $M$ be a closed, orientable, irreducible $3$--manifold.
An \emph{essential surface} $S \to M$ is a map from a
closed, connected, orientable, aspherical surface $S$
to a $3$--manifold $M$ such that $\tilde{S} \to \tilde{M}$ is an embedding.
Note that the collection of all translates of $\tilde{S}$
provides a geometric wallspace on $\tilde{M}$.
We thus say that an essential surface $S \to M$
has the \emph{$k$--Plane Intersection Property}
if this wallspace satisfies the $k$--Plane Intersection Property
as defined above.

Embedded essential surfaces in $3$--manifolds obviously have
the $1$--Plane Intersection Property.
More generally, if a cover of an essential surface $S$ in a $3$--manifold $M$
lifts to an embedding in a finite cover $\hat{M}$
of $M$, then $S$ has the $k$--Plane Intersection Property
where $k$ is the degree of any regular covering $\bar{M} \to M$
factoring through $\hat{M}$.

Rubinstein--Wang \cite{RubinsteinWang98} constructed an example of
an essential surface in a graph manifold that does not have the $k$--Plane
Intersection Property for any $k$.
In particular, the corresponding $\CAT(0)$ cube complex is infinite
dimensional.

On the other hand, Rubinstein--Sageev \cite{RubinsteinSageev99} showed that
if $M$ is a finite volume hyperbolic $3$--manifold, then every
acylindrical essential surface in $M$ has the $k$--Plane Intersection
Property for some $k$.  (If $M$ is compact and hyperbolic,
then every essential surface in $M$ is acylindrical.)
They also proved a similar result for geometrically finite essential
surfaces in $3$--manifolds $M$ whose JSJ decomposition
contains only hyperbolic pieces glued along tori.

\begin{defn}[Bounded packing]
\label{defn:BoundedPacking}
If $G$ is a group with finite generating set $\mathcal{A}$,
a pair of subsets $A,B \subseteq G$ are \emph{$D$--close}
if $\dist_{\mathcal{A}}(A,B) :=
\min \set{\dist_{\mathcal{A}}(a,b)}{a\in A, \ b\in B} < D$.
A subgroup $H \le G$ has \emph{bounded packing} in $G$
if for each $D>0$ there exists $k$ such that
every collection of $k+1$ left cosets $gH$ of $H$ contains a pair that is
not $D$--close.  In other words, every collection of left cosets of $H$
that is pairwise $D$--close has cardinality at most $k$.
Note that while the specific value of $k$ depends on the metric,
having the bounded packing property is independent of the
generating set.
\end{defn}

Sageev established bounded packing for quasiconvex subgroups of
hyperbolic groups in \cite{Sageev97}.
We prove the following generalization in \cite{HruskaWise09}.

\begin{thm}\label{thm:RelHypBoundedPacking}
Let $H$ be a relatively quasiconvex subgroup of a relatively
hyperbolic group $G$.
Suppose $H\cap gPg^{-1}$ has bounded packing in $gPg^{-1}$ for each
conjugate of each peripheral subgroup $P$.
Then $H$ has bounded packing in $G$.
\end{thm}

\begin{defn}
\label{defn:Osculating}
Let $(X,\W)$ be a wallspace.
Generalizing the notion of a wall separating two points,
we say $W$ \emph{separates} subsets $P,Q$ of $X$
if $P$ and $Q$ lie in distinct open halfspaces
of $W$.
A wall $W''$ \emph{separates} $W,W'$ if
$W''\ne W$, \ $W'' \ne W'$, and
there are closed halfspaces $U$ of $W$ and $V'$ of $W'$
such that $W''$ separates $U$ and $V'$.

Distinct walls $W,W'$ \emph{osculate} if
they are not transverse and there is no wall separating
them.
\end{defn}

\begin{rem}
Walls $W,W'$
osculate in $X$ if and only if
the corresponding hyperplanes $H,H'$ of $C(X)$
are dual to $1$--cubes $e,e'$
that intersect at a $0$--cube
but do not lie in a $2$--cube.
Indeed, let $W=\{U,V\}$ and $W'=\{U',V'\}$
be osculating walls with $V \cap V' = \emptyset$.
Choose a minimal length path $\gamma$ in $C(X)$
between
$C(\mathcal{U}_V)$ and $C(\mathcal{U}_{V'})$.
Reasoning as in the proof of
Proposition~\ref{prop:Connected}
one reaches the conclusion that $\gamma$
is the concatenation of two edges dual to $W$ and
$W'$ respectively.
Conversely, walls $W,W'$ corresponding to hyperplanes
dual to adjacent $1$--cubes $e,e'$ not in a $2$--cube
cannot be separated by any other wall
since it would then be impossible for both $W,W'$
to be flippable.
\end{rem}

The following theorem explains how the bounded packing
property applies to ensure the $k$--Plane
Intersection Property, and hence the finite dimensionality of
the dual cube complex
and similarly its uniform local finiteness.
We emphasize that the additional necessary hypothesis
for finite dimensionality is usually immediate
in a geometric setting.
However the additional hypothesis for uniform local finiteness is far from transparent because it entails
a more global condition
related to the Wall--Wall Separation Property
(see Sections \ref{subsec:WallWallRH}~and~\ref{sec:RHLocalFiniteness}).

\begin{thm}
\label{thm:BoundedPacking}
Let $G$ be a finitely generated group that
acts on a wallspace $(X,\W)$ with finitely many $G$--orbits of walls.

Suppose there are finitely many $G$--orbits
of pairs of transverse walls
and that the stabilizer of each wall has bounded packing in $G$.
Then the wallspace has the $k$--Plane Intersection Property.
Consequently the dual cube complex $C(X,\W)$ is finite dimensional.

Suppose in addition that there are finitely many $G$--orbits
of pairs of osculating walls.
Then $C(X,\W)$ is uniformly locally finite,
in the sense that there is a uniform upper bound
on the number of $1$--cubes incident to each $0$--cube.
\end{thm}

\begin{proof}
Choose representatives $W_1,\dots,W_r$
of the finitely many $G$--orbits of walls,
and let $H_1,\dots,H_r$ denote their stabilizers.
Observe that each wall $W\in \W$ is a translate $gW_i$ of one of our
representatives,
and is associated with the coset $gH_i$.

Choose finitely many pairs of walls
representing the finitely many $G$--orbits of pairs of transverse walls.
Let $D$ be an upper bound
on the distance between their associated cosets.
As there are finitely many distinct $G$--orbits of walls in $X$
and by hypothesis the bounded packing property holds for their stabilizers,
we can choose $k$ to be an upper bound on the size of a collection of
pairwise $D$--close cosets.

By Corollary~\ref{cor:FiniteDim}, the $k$--Plane Intersection Property
for $(X,\W)$ implies
the finite dimensionality of $C(X,\W)$.

Similarly, a collection of $1$--cubes at a $0$--cube
correspond to a collection of walls that
pairwise either osculate or are transverse.
Hence the result follows from the bounded packing property
in the same manner as above.
\end{proof}

If we assume the subgroups in question
are finitely generated, then by
Lemma~\ref{lem:TransverseToCloseSubgroups}
there are finitely many orbits of pairs of transverse walls.
We thus deduce the following corollary.

\begin{cor}
Suppose the finitely generated group $G$ has bounded packing with respect to finitely generated subgroups $H_1,\dots, H_r$.
Let $C(X)$ be the dual CAT(0) cube complex with respect to a system of $H_i$--walls in $X=\Gamma(G)$.
Then $C(X)$ is finite dimensional.
\end{cor}

The following converse to the second part of
Theorem~\ref{thm:BoundedPacking}
was proven in \cite[Thm~3.2(1)]{HruskaWise09}.
(The statement there incorrectly
omits the word ``uniformly'',  but uniform local finiteness
is explicitly used in the proof.)

\begin{thm}
Suppose a finitely generated group $G$ acts on a $\CAT(0)$ cube complex $C$
and $H\le G$ is the stabilizer of a hyperplane.
If $C$ is uniformly locally finite then $H$ has bounded packing in $G$.
\end{thm}

\section{$\CAT(0)$ spaces with convex walls}
\label{sec:ConvexWalls}

In Section~\ref{sec:HighlyGeometric}
we describe our main theorems in the setting
of a $\CAT(0)$ space with convex walls.
In Section~\ref{sec:2ComplexExamples}
we give simple examples of $2$--dimensional $\CAT(0)$ spaces
with convex walls.
The discussion of these examples incorporates
the subsequent finiteness results on dual cube complexes
from Sections \ref{sec:Properness}~and~\ref{sec:RelCocompact}
that were stated in the introduction.
We hope that the reader will be able to navigate
Section~\ref{sec:ConvexWalls} to have concrete examples in mind
and then revisit the section later.

\subsection{Main results illustrated in a highly geometric setting}
\label{sec:HighlyGeometric}

The following is a motivating example of a geometric wallspace
similar to the Coxeter and cubical settings.

\begin{exmp}
Let $X$ be a $\CAT(0)$ space.
Let $\W$ be a locally finite
collection of convex $2$--sided subspaces.
\end{exmp}

We will describe how some such examples arise in certain naturally
occurring $\CAT(0)$ $2$--complexes in Section~\ref{sec:2ComplexExamples}.
Higher dimensional versions of this scenario can be described analogously
but appear to be scarcer.
In a hyperbolic high dimensional instance,
$X$ is $\Hyp^n$ and the convex walls are copies of $\Hyp^{n-1}$
(see for instance Bergeron--Haglund--Wise \cite{BergeronHaglundWiseSimple}).

\textbf{A cocompact action:}
Assume that a group $G$ acts properly and cocompactly on $(X,\W)$.
Observe that stabilizers of walls are quasi-isometrically embedded in $G$.
According to Theorem~\ref{thm:HyperbolicCocompact},
if $G$ is word-hyperbolic, then $G$ acts cocompactly on the
dual cube complex.

\textbf{A relatively cocompact action:}
Let us now consider a motivating case of
the relatively hyperbolic situation;
namely suppose that $X$ is also a $\CAT(0)$ space with isolated flats.
According to Theorem~\ref{thm:RelCocompactWallspace},
$G$ acts relatively cocompactly on the dual cube
complex (see Definition~\ref{defn:RelCocompact}),
and in fact the dual cube complex looks like a
quasi-copy of $G$ together with various cubulations of
free abelian groups hanging off.
The simplest such example arises from a piecewise Euclidean $2$--complex
all of whose $2$--cells are regular even-sided
polygons with at least $6$ sides (see Section~\ref{sec:2ComplexExamples}).
Note that the same result holds without assuming that the polygons have an even number of sides \cite{WiseSmallCanCube04},
but there is no longer an obvious natural choice
of isometrically embedded walls.

\textbf{A proper action:}
Assume now that a group $G$ acts properly on $(X,\W)$.
According to Theorem~\ref{thm:LinearSeparationProper},
$G$ acts metrically properly on the dual cube complex provided that there exists
$L>0$ such that any two points at distance at least $L$ are
separated by some wall.

\begin{prob}
Is there an example of a $\CAT(0)$ space $X$ with a proper cocompact
$G$--action and a locally finite $G$--invariant system $\W$ of convex
walls such that the dual
cube complex $C(X,\W)$ is not finite dimensional?
\end{prob}

\subsection{Examples of $\CAT(0)$ $2$--complexes with convex walls}
\label{sec:2ComplexExamples}

In this section we will describe simple conditions on a $\CAT(0)$
$2$--complex $X$ that endow it with a locally finite collection $\W$
of convex $2$--sided trees.
Moreover this collection \emph{fills} $X$ in the sense that the Linear
Separation condition holds (see Definition~\ref{def:LinearSeparation}).

We first describe a condition yielding trees that are disjoint from
the $0$--cells.

\begin{exmp}
Suppose $X$ is a $\CAT(0)$ $2$--complex such that each $2$--cell is
a regular Euclidean polygon with an even number of sides.
For each $1$--cell $e$, the tree $T$ dual to $e$ passes through the midpoint
$m$ of $e$ as a perpendicular bisector.  If $e$ lies on $r$ $2$--cells
then $m$ is a valence $r$ vertex of $T$.
The tree bisects any $2$--cell $c$ that it passes through---$T \cap c$
is an edge of $T$, which ends at the midpoints of opposite $1$--cells
of $c$, as shown on the left in Figure~\ref{fig:LinkCutters}.
Similar walls have been examined in a slightly more general context
in \cite{HaglundPaulin98}.
\end{exmp}

\begin{figure}
\begin{center}
\includegraphics[width=.8\textwidth]{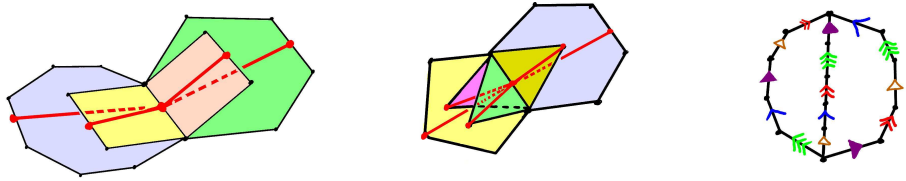}
\end{center}
\caption{The figure on the left shows a portion of the dual tree in
the case when all $2$--cells are regular polygons with an even number
of sides.
The figure in the middle illustrates the case when some $2$--cells
have an odd number of sides, and the dual tree is forced to
pass through $0$--cells.
On the right is a $\ominus$--graph and several sets of midpoints
that are $\pi$--separating,
provided that the edges have length $\ge \pi/3$.}
\label{fig:LinkCutters}
\end{figure}

We now consider the more general setting where polygons are not regular and may not have an even number of sides.
Conditions guaranteeing the existence of such trees are
trickier to state, especially considering that there are
such examples where no isometric $2$--sided trees exist.

The way in which a tree $T$ locally cuts $X$ in half is modeled by
two situations, illustrated in Figure~\ref{fig:LinkCutters}:
\begin{enumerate}
\item perpendicularly cutting through a $1$--cell at an internal point, and
\item perpendicularly cutting through a $0$--cell $v$.
\end{enumerate}

In the second case, $v$ is actually a vertex of $T$, and we now describe what we mean by ``perpendicularly cutting through $v$.''
There is an embedding $\Lk_T(v) \hookrightarrow \Lk_X(v)$.
Note that the edges of $\Lk_X(v)$ are assigned lengths according to the
angles of the corresponding corners of $2$--cells in $X$.
The \emph{angular metric} on $\Lk_X(v)$ is the
induced path metric on $\Lk_X(v)$.

\emph{Perpendicularly cutting through $v$}
requires that:
\begin{enumerate}
\item points of $\Lk_T(v)$ are \emph{$\pi$--separated}
in $\Lk_X(v)$, in the sense that for each distinct
$a,b \in \Lk_T(v)$, the angular
distance in $\Lk_X(v)$ between $a$ and $b$ is at least $\pi$, and
\item $\Lk_T(v)$ separates $\Lk_X(v)$ into two complementary components.
\end{enumerate}

When $Lk_T(v)$ is a $\ominus$--graph whose three arcs have length $\pi$,
then there are two $\pi$--separated separating sets.
But there are more possibilities when the lengths of the arcs are increased,
as shown on the right of Figure~\ref{fig:LinkCutters}.

Patching together these local models to yield a locally finite collection of $2$--sided trees can be tricky.
A simple condition that guarantees that the local models can be patched together requires that:
\begin{enumerate}
\item all polygons are regular,
\item the embedding $\Lk_T(v) \subset \Lk_X(v)$ lies at the centers of edges,
and
\item the center of each edge lies in exactly one $\Lk_T(v)$ subset.
\end{enumerate}
An example of this is illustrated on the right of
Figure~\ref{fig:LinkCutters}.

The $\Lk_T(v)$ subsets make $\Lk_X(v)$ into a wallspace.
This type of local wallspace structure is exploited in
\cite{WiseIsraelHierarchy} to create many global wallspaces in a
similar fashion.

The following example arose recently in the work of Barr{\'e}--Pichot
\cite{BarrePichot2010}.
Barr{\'e}--Pichot use these walls to see that $\pi_1X$ is a-T-menable, by applying a criterion of Haglund--Paulin \cite{HaglundPaulin98}.
\begin{exmp}
Let $X$ be the standard $2$--complex of the following presentation, and let $x$ be the $0$--cell of $X$:
\[
   \langle \,a,b,c\mid abc, acb, ad^{-1}bd^{-1}cd^{-1}\,\rangle
\]
The three $2$--cells and $\text{link}(x)$ are illustrated
in Figure~\ref{fig:BarrePichot}.
Observe that $X$ has a nonpositively curved piecewise-Euclidean $2$--complex whose $2$--cells are equilateral triangles and a regular hexagon. The walls in $\widetilde X$ are ``generated'' by geodesic segments that cut through the $2$--cells as illustrated in Figure~\ref{fig:BarrePichot}. The three such segments in each triangle are the perpendicular bisectors of its sides.
There are nine such segments in the hexagon: three are the perpendicular bisectors of the hexagon, and the other six
cut across the hexagon to form a triangle with consecutive sides of the hexagon.

The graph $\text{link}(x)$ is illustrated on the right of Figure~\ref{fig:BarrePichot}.
There are three walls that perpendicularly cut through  $\text{link}(x)$. We have illustrated one of these walls,
namely the one  corresponding to the four bold vertices of the wall segments on the left.

The linear separation property for the geometric wallspace on $\widetilde X$ generated by these walls is immediate.
Indeed, let $L$ denote the union of all the walls, and consider $\widetilde X- L$.
Each component of $\widetilde X-L$ is a subspace of the union of cells containing a single $0$-cell,
and hence has uniformly bounded diameter.
  \end{exmp}

\begin{figure}
\begin{center}
\includegraphics[width=.9\textwidth]{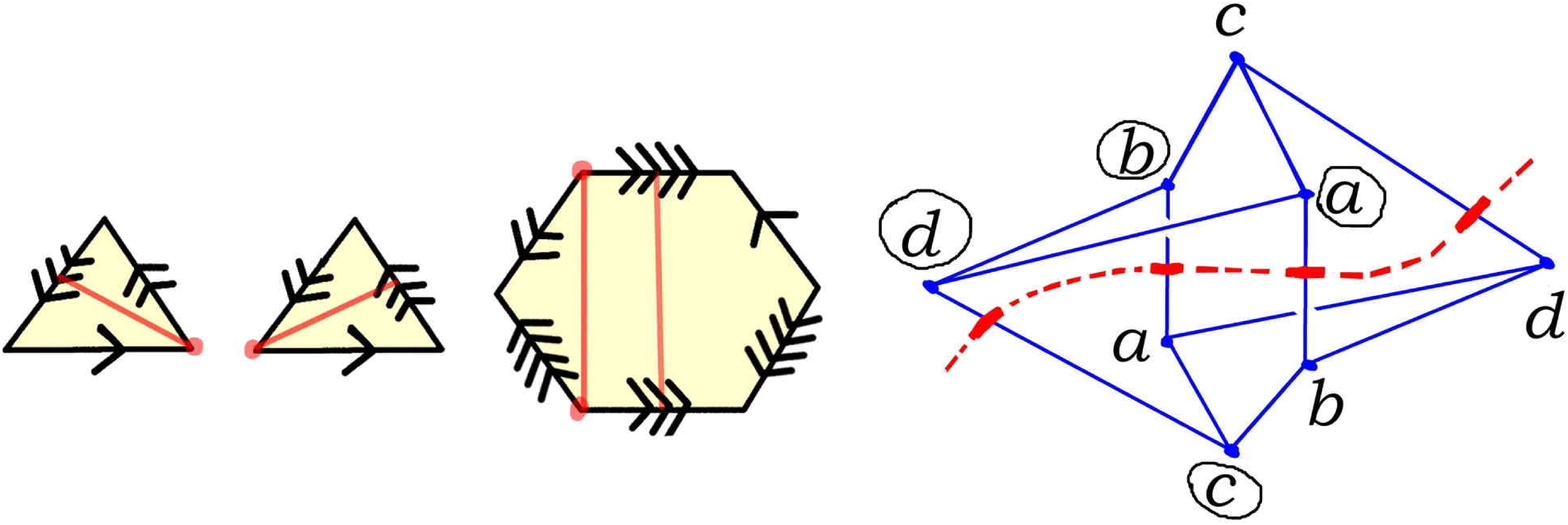}
\end{center}
\caption{The walls in $\widetilde X$ are generated by segments cutting through the $2$--cells, some of these are indicated on the left.
Each wall perpendicularly cuts $\text{link}(x)$. One of the three types of such cuts is illustrated within $\text{link}(x)$ on the right. }
\label{fig:BarrePichot}
\end{figure}

\section{Properness and local finiteness}
\label{sec:Properness}

\subsection{Properness of the action}
\label{sub:Properness}

\begin{defn}[Linear Separation]
\label{def:LinearSeparation}
Let $G$ act on a geometric wallspace $(X,\W)$.
Recall that the underlying space $(X,\dist)$ is a metric space.
The geometric wallspace has the \emph{Linear Separation Property}
if there exists constants $\kappa,\epsilon>0$ such that
for all points $x,y \in X$ we have
\[
   \#(x,y) \ge \kappa\,\dist(x,y) - \epsilon.
\]
\end{defn}

The reader may wish to verify that Linear Separation holds for
the examples in Section~\ref{sec:2ComplexExamples}.

An inequality of the form $\#(x,y) \le \kappa\,\dist(x,y) +\epsilon$
is of less interest here.  For instance this inequality
holds whenever the following ubiquitous condition is satisfied:
$X$ is a geodesic metric space $G$ acts on $(X,\W)$, there exists a bounded
set $B$ such that $X = GB$,
and there are finitely many distinct walls that separate points of $B$.

An action of a group $G$ on a topological space $X$
is \emph{proper} if
$\set{g \in G}{gK \cap K \ne \emptyset}$
is finite for each compact subspace $K \le X$.
We note that an action $G$ on a cube complex $C$
is proper precisely when all $0$--cubes
of $C$ have finite stabilizers.

An action of a group $G$ on a metric space $X$
is \emph{metrically proper}
if for each $r>0$ and each $x \in X$
the set $\set{g}{g B_r(x) \cap B_r(x) \ne \emptyset}$ is finite.
Note that if $G$ acts metrically properly then $G$ acts properly.
The two notions coincide when $X$ is \emph{proper},
which means that every closed ball in $X$ is compact.

In particular, observe that a proper action
on a locally finite cube complex is always
metrically proper.

\begin{thm}
\label{thm:LinearSeparationProper}
Suppose $G$ acts on a geometric wallspace $(X,\W)$,
and the action on the underlying metric space $(X,\dist)$
is metrically proper.
Then the Linear Separation Property for $(X,\W)$ implies that $G$
acts metrically properly on $C(X,\W)$.

More generally $G$ acts metrically properly on $C(X,\W)$ provided that
for some $x\in X$ we have:
$\#(x,gx)\rightarrow \infty$ as $g\rightarrow \infty$.
\end{thm}

We have in mind the case when $G$ has a proper metric
(e.g., a word metric),
in which case
a sequence of distinct elements in $G$ always tends to infinity.

Although Theorem~\ref{thm:LinearSeparationProper} holds under the more
general hypothesis stated second, we have deliberately chosen to emphasize
the Linear Separation hypothesis as this stronger property on $(X,\W)$
is often the natural one to verify.

\begin{proof}
Suppose $(g_i)$ is a sequence of distinct elements in $G$,
and let $x \in X$ be any basepoint.
Since $G$ acts metrically properly on $X$, we have $\dist(x,g_ix) \to \infty$
as $i\to\infty$.
Thus $\#(x,g_i x) \to \infty$.
We shall now verify that $\#(x,y) = \dist_C (c_x, c_y)$,
where $c_x$ denotes the canonical cube containing $x$.
Consequently $\dist_C( c_x, g_i c_x) = \dist_C( c_x, c_{g_i x}) \to \infty$.

Recall that the independent walls of $c_x$ are precisely
those that betwixt $x$.
Observe that if $W$ separates $x$ from $y$
then it is betwixted by neither of these points.
Hence $W$ is a dependent wall oriented towards $x$ in the cube $c_x$
and towards $y$ in $c_y$.
Consequently each such separating $W$ contributes a value of $1$ to
$\dist_C(c_x,c_y)$.
Moreover each hyperplane separating $c_x$ from $c_y$
corresponds to such a separating wall.
\end{proof}

A wall $W$ \emph{crosses}
a set $S$ if $S$ intersects each closed halfspace of $W$.

\begin{lem}[Ball-Ball separation]
\label{lem:compact-compact}
Let $X$ be a metric space and also a wallspace
such that each finite radius ball in $X$ is crossed by finitely many walls.
Suppose $G$ acts cocompactly on $(X,\W)$ and acts properly on the dual cube complex $C(X)$.
For each $r$ there exists $m$ such that
if\/ $\dist_X(x_1,x_2)>m$ then there is a wall separating $\neb_r(x_1)$ from $\neb_r(x_2)$.
\end{lem}

``Ball-Ball Separation'' is considered among other
separation properties in
Section~\ref{subsec:WallWallRH}.

\begin{proof}
By increasing $r$, we can assume without loss of generality that each $x_i=g_ix$ Êfor some $g_i\in G$ and a chosen point $x$.
The hemiwallspace associated with $\neb_r(x)$ corresponds to a compact convex
subcomplex $C_r\bigl(\{x\}\bigr)$ of $C(X)$.
By properness of the $G$ action on $C(X)$, the set of elements
$S = \bigset{g}{g C_r\bigl(\{x\}\bigr)\cap C_r\bigl(\{x\}\bigr)\neq \emptyset}$ is finite.
Disjoint convex subcomplexes of $C(X)$ are separated by a hyperplane
(see for instance \cite[Thm~2.7]{RollerPocSets}).
Consequently $\neb_r(gx)$ can be separated from $\neb_r(x)$ provided $g\not\in S$.
Let $m= \max_{g\in S} \dist_X(x,gx)$.
Then Ê$\neb_r(g_1x)$ is separated from $\neb_r(g_2x)$ provided that $\dist(g_1x,g_2x)>m$.
\end{proof}

\begin{rem}
A wallspace is \emph{Hausdorff} if $\#(x,y) \ne 0$ whenever $x\ne y$.
Note that $(X,\#)$ is a metric space
when $(X,\W)$ is a Hausdorff wallspace.
The condition that $\#(x,gx)\rightarrow \infty$ as $g\rightarrow \infty$
can then be restated as: $G$ acts metrically properly on $(X,\#)$.
Theorem~\ref{thm:LinearSeparationProper} is well-known in this setting.

The following non-Hausdorff
example shows that $\#$ can fail to satisfy the triangle
inequality.
Let $X=\{x,y,z\}$. Let $\W$ consist of a single wall given by
$\bigl(\{x,z\}, \{y,z\}\bigr)$.
Then $\#(x,y) = 1$ but $\#(x,z)=\#(y,z)=0$.
Nevertheless $\#$ is related to a distance function since,
as explained in the proof above $\#(x,y) = \dist_C(c_x,c_y)$.
\end{rem}

We motivate Theorem~\ref{thm:AxisSeparation} by the following scenario:
Let $g$ act on $X$ by an isometry preserving an axis $A$
isometric to the real line.
Let $W$ be a wall in $X$ that separates the ends $\pm\infty$ of $A$
in the sense that $(-\infty,-n)$ and $(n,+\infty)$ lie in distinct
open halfspaces of $W$.
We think of $W$ as ``cutting'' $g$.

A more extreme scenario is where $W=\{U,V\}$
and we have $g(U) \subsetneq U$ and $V \subsetneq g(V)$.

\begin{thm}[Axis Separation]
\label{thm:AxisSeparation}
Let $g \in G$ be an infinite order element.
Suppose an infinite order element $g \in G$ is cut by a wall $W=\{U,V\}$
in the following sense:
\begin{enumerate}
\item $W \ne g^n W$ when $n \ne 0$
\item $U \cap g^n U \ne \emptyset$
for all $n$
\item $V \cap g^n V \ne \emptyset$
for all $n$
\end{enumerate}

If $g$ fixes a $0$--cube $c$ of $C(X,\W)$,
then $(X,\W)$ has an infinite collection of pairwise transverse walls,
and hence $C(X,\W)$ is infinite dimensional.
In particular, if $C(X,\W)$ is finite dimensional,
then $G$ acts with torsion stabilizers.
\end{thm}

\begin{proof}
Suppose $gc=c$ for some $c\in C^0$,
and hence $g^{-m} c = c$ for each $m$.
This means that $g^{-m}c$ and $c$ orient each wall $W$
in the same way, so we have
$\overleftarrow{g^{-m} c}(W) = \overleftarrow{c} (W)$.
As mentioned in Remark~\ref{rem:InducedAction},
$\overleftarrow{g^{-m} c}(W) = g^{-m} \overleftarrow{c} (g^m W)$.
Thus $\overleftarrow{c} (g^m W) = g^m \overleftarrow{c} (W)$
for all $m$.

We will see that the walls $\set{g^{m}W}{m \in \Z}$
are pairwise transverse.
Suppose by way of contradiction that
$W$ and $g^{n}W$ are not transverse for some $n$.
By hypothesis the quarterspaces
$\overleftarrow{c}(W) \cap g^n\overleftarrow{c}(W)$ and
$\overrightarrow{c}(W) \cap g^n \overrightarrow{c}(W)$
are nonempty.
It follows that at least one of the quarterspaces
$\overleftarrow{c}(W) \cap \overrightarrow{c}(g^nW)$ and $\overrightarrow{c}(W) \cap \overleftarrow{c}(g^nW)$ is empty.
Therefore one of the halfspaces $\overleftarrow{c}(W)$ and $\overleftarrow{c}(g^{n}W)$
is contained in the other.
Replacing $n$ with $-n$ if necessary, we can arrange that
$c(g^n W) \prec c(W)$.
This yields an infinite properly descending sequence
\[
   c(W) \succ c(g^n W) \succ c(g^{2n} W) \succ \cdots
\]
which is impossible by Lemma~\ref{lem:NoInfiniteChains}.

We thus conclude that $W$ and $g^{m}W$ are transverse for each $m$,
and thus $g^{k}W$ and $g^{m}W$ are transverse for each $k\neq m$.
By Corollary~\ref{cor:FiniteDim}, the dual cube complex is infinite dimensional.
\end{proof}

\subsection{Local finiteness}
The investigation of local finiteness of $C$ is often
simplified in the presence of a group action on $(X,\W)$ because
if $G$ acts properly and cocompactly on $C$ then $C$ is locally finite.

More generally suppose $G$ acts properly on $C$.
Assume there is a collection of subcomplexes $C_i \subset C$
and a compact subcomplex $K \subset C$ such that
$C= GK \cup \bigcup_i GC_i$.
Suppose each $0$--cube has a neighborhood
that either lies in some $gC_i$ or lies in $GK$.
Then $C$ is locally finite provided that each $C_i$ is
locally finite.  We emphasize that the subcomplexes $gC_i$
might intersect,
and this intersection could be unrelated to $GK$.

We will show in Section~\ref{sec:RelHypApplication}
that the above decomposition of $C$
arises naturally when cubulating certain relatively hyperbolic groups.
In that setting we obtain the stronger conclusion that
points outside of $GK$ lie in a \emph{unique} $gC_i$.
In other words, the subcomplexes $gC_i$ can intersect
only inside $GK$.
We then examine the local finiteness of $C$ in
that setting in Section~\ref{sec:RHLocalFiniteness}.

In this subsection we describe a simple criterion on a wallspace to recognize local finiteness of its cubulation.
We give conditions similar to the Parallel Wall Separation
Condition of Brink--Howlett that was used by Niblo--Reeves to
prove local finiteness of their cubulations of Coxeter groups
\cite{BrinkHowlett93,NibloReeves03}.
Guralnik has undertaken a study of local finiteness of the
cube complex dual to a system of convex walls in a $\CAT(0)$
space obtaining a result similar to one implication of Theorem~\ref{thm:LocalFiniteness}
\cite{GuralnikLocalFiniteness}.

\begin{defn}[Compact--Wall separation]
Let $X$ be a wallspace that is also a metric space.
Then $X$ has \emph{compact--wall separation}
if for each compact set $K \subseteq X$ there exists a constant $f(K)$
such that
whenever $\dist(K,W) \geq f(K)$ there is a wall $W'$
separating $K$ from $W$.
By this we mean that $K$ lies in one open halfspace of $W'$
and a closed halfspace of $W$ lies in the other.

In particular, this condition
holds when $X$ has \emph{ball--wall separation at
$x \in X$}, in the sense that
for each $r\geq0$ there exists $f(r,x)$ such that
if $\dist(x,W)\geq f$ then there is a wall $W'$ separating
$\neb_r(x)$ from $W$.
Note that ball--wall separation with respect to a particular $x\in X$
implies ball--wall separation with respect to every $y\in X$.
\end{defn}

\begin{thm}
\label{thm:LocalFiniteness}
Suppose $X$ is a metric space and also a wallspace
\textup{(}with no vacuous walls\textup{)},
and that each finite radius ball is crossed by finitely many walls.
Then $(X,\W)$ has compact--wall separation if and only if
$C(X,\W)$ is locally finite.
\end{thm}

\begin{proof}
First assume that $(X,\W)$ has compact--wall separation.
Let $c$ be a $0$--cube of $C(X,\W)$.
In order to bound the valence of $c$,
we need to bound the number of oriented walls $c(W)$
that are \emph{reversible}, in the sense that
reversing only the orientation of $W$ yields a new $0$--cube.

Let $x \in X$ be a basepoint.
Consider the finitely many walls
$W_1,\dots,W_n$ that are \emph{oriented away} from $x$
by $c$, in the sense that $x \notin \overleftarrow{c}(W_j)$ for all $j$.
For each $j$ choose $x_j \in \overleftarrow{c}(W_j)$.
(Here we use that there are no vacuous walls.)
Let $K = \{x,x_1,\dots,x_n\}$.

Note that each $W_j$ crosses $K$ since $x \in \overrightarrow{c}(W_j)$ but
$x_j \in \overleftarrow{c}(W_j)$.
Note that by hypothesis
there are finitely many walls crossing $\neb_{f}(K)$,
where $f=f(K)$ is also given by hypothesis.

We claim that each reversible oriented wall
of $c$ is one of these finitely many walls
crossing $\neb_{f}(K)$.
Indeed, if $W$ is reversible and doesn't cross $\neb_{f}(K)$
then by compact--wall separation,
there exists $W'$ separating $K$ and $W$.
Observe that all walls not crossing $K$ are oriented towards $x$.
Since $W'$ does not cross $K$,
it follows that $\overleftarrow{c}(W')$ contains $K$.
For the same reason $\overleftarrow{c}(W)$ contains $K$.
As $W'$ separates $W$ from $K$, it follows that
$\overleftarrow{c}(W')$ is disjoint from $\overrightarrow{c}(W)$,
which contradicts the reversibility of $W$.

We now show that if $C(X,\W)$ is locally finite, then $(X,\W)$ has
compact--wall separation.
Let $K$ be a compact set in $X$.

The local finiteness of walls in $X$ implies that the hemiwallspace
induced by $K$ has finitely many independent walls.
This ensures the compactness of the dual $C(K)$, which is a convex
subcomplex of $C(X,\W)$.
The local finiteness of $C(X,\W)$ implies that there are finitely many
hyperplanes dual to $1$--cubes with a $0$--cube on $C(K)$.

Suppose there is a sequence of walls $W_i$ in $\W$ such that
$\dist(K,W_i) \ge i$, but each $W_i$ is not separated from $K$.
We claim that each $W_i$ corresponds to a hyperplane $H_i$ in $C(X,\W)$
dual to a $1$--cube with a $0$--cube on $C(K)$,
which is impossible.
To verify the claim, consider a minimal (combinatorial)
geodesic $\gamma$ between
$C(K)$ and $C(W_i)$, which is the cubical neighborhood of $H_i$.
One can verify that each edge of $\gamma$ corresponds to a hyperplane
separating $C(K)$ and $C(W_i)$.  This is associated to a wall
separating $K$ and $W_i$.
\end{proof}

We close this section with the following example
of a wallspace $(X,\W)$,
which has compact-wall separation even though
its dual cube complex $C(X,\W)$ is not locally finite.

\begin{exmp}
\label{exmp:Rbad}
Let $X = \R$ and the walls in $\W$ are of the following two types:
\begin{enumerate}
\item $\bigl\{ (-\infty,n],[n,\infty) \bigr\}$
for each $n \in \Z$, and
\item $\bigl\{ \{r\}, \R-\{r\} \bigr\}$ for each $r \in \R$.
\end{enumerate}
The dual cube complex $C(X,\W)$ is the union of a horizontal
line $L$ consisting of infinitely many horizontal edges
with an uncountable set of vertical
edges attached at each vertex
of $L$ and a square freely attached at each horizontal
edge of $L$.

The $n$-th square arises since
$\bigl\{ (-\infty,n],[n,\infty) \bigr\}$
is transverse to $\bigl\{ \{n\}, \R-\{n\} \bigr\}$.
The vertical edges between the $n$-th and $(n+1)$-th
square correspond to
$\bigl\{ \{r\}, \R-\{r\} \bigr\}$
with $n\le r \le n+1$.
\end{exmp}

\section{Review of relative hyperbolicity and quasiconvexity}

In this section we review the definition of a relatively hyperbolic
group and a relatively quasiconvex subgroup of such a group.

A geodesic metric space
$(\widetriangle{X},\dist)$ is \emph{$\delta$--hyperbolic} if every geodesic
triangle with vertices in $\widetriangle{X}$ is $\delta$--thin
in the sense that each side lies in the $\delta$--neighborhood
of the union of the other two sides.
If $\widetriangle{X}$ is a $\delta$--hyperbolic space, the \emph{boundary} (or \emph{boundary at infinity}) of $\widetriangle{X}$,
denoted $\boundary \widetriangle{X}$, is the set of all equivalence classes of
geodesic rays in $\widetriangle{X}$, where two rays $\gamma,\gamma'$ are equivalent if
the distance
$\dist\bigl(\gamma(t),\gamma't)\bigr)$ remains bounded as $t \to \infty$.
The set $\boundary \widetriangle{X}$ has a natural topology, which is compact
and metrizable (see, for instance, Bridson--Haefliger
\cite{BridsonHaefliger} for details).

Increasing the constant $\delta$ if necessary, we will also assume
that every geodesic triangle with vertices in $\widetriangle{X} \cup \boundary \widetriangle{X}$
is $\delta$--thin.

Let $\Delta=\Delta(x,y,z)$ be a triangle
with vertices $x,y,z \in \widetriangle{X} \cup \boundary \widetriangle{X}$.
A \emph{center} of $\Delta$ is a point $w \in \widetriangle{X}$ such that
the ball $\ball{w}{\delta}$ intersects all three sides of the triangle.
It is clear that each side of $\Delta$ contains a center of $\Delta$.

If $\xi \in \boundary \widetriangle{X}$, a function $h \colon \widetriangle{X} \to \R$ is a
\emph{horofunction centered at $\xi$}
if there is a constant $D_0$ such that the following holds.
Let $\Delta(x,y,\xi)$ be a geodesic triangle, and let $w \in \widetriangle{X}$ be a center
of the triangle.
Then
\[
   \left| \bigl( h(x) + \dist(x,w) \bigr) - \bigl( h(y) + \dist(y,w) \bigr)
   \right|  < D_0.
\]
A subset $B \subseteq \widetriangle{X}$ is a \emph{horoball centered at $\xi$}
if there is a horofunction $h$ centered at $\xi$ and a constant $D_1$
such that $h(x) \ge - D_1$ for all $x \in B$ and $h(x) \le D_1$
for all $x \in \widetriangle{X} - B$.

We remark that every horoball or horofunction has a unique center,
and every $\xi \in \boundary \widetriangle{X}$ is the center of a horofunction
(see Gromov \cite[Sec~7.5]{Gromov87} for details).

The following is a well-known result about the interaction between
geodesics and horoballs.
We leave the proof as an exercise for the reader.

\begin{lem}
\label{lem:GeodesicPenetratesHoroball}
Let $B$ and $B'$ be two horoballs of $\widetriangle{X}$ centered at the same point $\xi$
such that $B \subset B'$.
There is a constant $M=M(B,B')$ such that the following holds.
Let $\gamma$ be a geodesic of $\widetriangle{X}$ with endpoints $x,y \in B'-B$.
Then
\[
   \gamma \cap (\widetriangle{X}-B) \subseteq \bignbd{\{x,y\}}{M}.
\]

In particular, if $\dist(x,y) > 2M$, then $\gamma$ must intersect $B$.
\end{lem}

Suppose $G$ has a proper, isometric action on a proper
$\delta$--hyperbolic space $\widetriangle{X}$.
An element $g \in G$ is \emph{loxodromic} if it has infinite order
and fixes exactly two points of $\boundary \widetriangle{X}$.
A subgroup $P \le G$ is \emph{parabolic} if it is infinite and contains
no loxodromic element.  A parabolic subgroup $P$ has a unique fixed point in $\boundary \widetriangle{X}$, called a \emph{parabolic point}.
The stabilizer of a parabolic point is always a maximal parabolic group.
Thus there is a natural one-to-one correspondence between parabolic points
and maximal parabolic subgroups.

The following is Gromov's definition
of relative hyperbolicity from \cite{Gromov87}, as elaborated by Bowditch
in \cite{Bowditch12}.

\begin{defn}[Relatively hyperbolic]
Let $G$ be a countable group with a finite collection $\P$ of infinite
subgroups.
Suppose $G$ acts properly on a proper $\delta$--hyperbolic
space $\widetriangle{X}$, and $\P$ is a set of representatives of the conjugacy classes
of maximal parabolic subgroups.
Suppose also that there is a $G$--invariant collection of disjoint
open horoballs centered at the parabolic points of $G$,
with union ${\mathcal{B}}$, such that the quotient of $\widetriangle{X}-{\mathcal{B}}$ by the action of $G$
is compact.
Then $(G,\P)$ is \emph{relatively hyperbolic}.

In this case, we say that the action of $(G,\P)$ on $\widetriangle{X}$ is
\emph{cusp uniform}, and the space $X = \widetriangle{X}-{\mathcal{B}}$ is a \emph{truncated space}
for the action.  We refer to the horoball components $B$
of ${\mathcal{B}}$ as \emph{horoballs of $X$}.
The subgroups $P \in \P$ are called \emph{peripheral subgroups}.
When $\P=\emptyset$ the truncated space $X$ is equal to $\widetriangle{X}$,
and the group $G$ is word-hyperbolic.

The hypothesis that all parabolic subgroups are infinite is a technical
luxury and not at all necessary for the theory.
\end{defn}

In \cite{Dahmani03}, Dahmani defined a subgroup $H \le G$
to be quasiconvex relative to the peripheral subgroups
if the action of $G$ as a geometrically finite convergence group
restricts to a geometrically finite convergence group action of $H$
on its limit set.
Subsequently Osin \cite{OsinBook06}
gave the following alternate formulation of relative quasiconvexity.
In \cite{HruskaRelQC}, Hruska showed that these definitions
are equivalent.

\begin{defn}[Relatively quasiconvex (A)]
Let $(G,\mathbb{P})$ be relatively hyperbolic with a finite generating set
$\mathcal{S}$.
Let $\mathcal{P} = \bigcup_{P\in \mathbb{P}} P-\{e\}$.
A subgroup $H \le G$ is \emph{relatively quasiconvex} if there is a constant $\kappa$ such that the following
condition holds:
For each geodesic $\gamma$ in $\Cayley(G,\mathcal{S} \cup \mathcal{P})$
connecting two points of $H$, every vertex of $\gamma$
lies within a distance $\kappa$ of $H$ in $\Cayley(G,\mathcal{S})$.
\end{defn}

We occasionally find the following formulation of relative
quasiconvexity useful.
This is also equivalent to the above by \cite{HruskaRelQC}.

\begin{defn}[Relatively quasiconvex (B)]
\label{def:RelQC}
Let $(G,\P)$ be relatively hyperbolic.
A subset $H \subseteq G$ is \emph{$\kappa$--relatively quasiconvex} if the
following holds.
Let $(\widetriangle{X},\dist)$ be some proper $\delta$--hyperbolic space
on which $(G,\P)$ has a cusp uniform action.
Let $\widetriangle{X}-{\mathcal{B}}$ be some truncated space for $G$ acting on $\widetriangle{X}$.
For some basepoint $x \in \widetriangle{X}-{\mathcal{B}}$ there is a constant $\kappa$
such that whenever $\gamma$ is a geodesic in $\widetriangle{X}$ with endpoints in
the orbit $Hx$, we have
\[
   \gamma \subseteq \nbd{Hx}{\kappa} \cup {\mathcal{B}}.
\]
We say that $H$ is \emph{relatively quasiconvex}
if it is $\kappa$--relatively quasiconvex for some $\kappa$.

Relative quasiconvexity of $H$ is independent of
the choice of $\widetriangle{X}$, ${\mathcal{B}}$, and $x$.
However the value of $\kappa$ depends on these choices.
See \cite{HruskaRelQC} for details.
\end{defn}

\begin{defn}[Quasiconvex]
Let $Z$ be a geodesic metric space and $\kappa >0$.
A subset $A \subseteq Z$ is \emph{$\kappa$--quasiconvex}
if $\gamma \subseteq \neb_\kappa (A)$ whenever $\gamma$ is a geodesic
in $Z$ whose endpoints lie on $A$.
We say $A$ is \emph{quasiconvex} if it is $\kappa$--quasiconvex
for some $\kappa>0$.

We often refer to quasiconvex subsets of a word-hyperbolic group $G$,
i.e., those that are quasiconvex as subsets of $\Cayley(G,S)$ for a finite
generating set $S$.
Note that quasiconvexity is independent of the particular choice
of $S$ when $G$ is word-hyperbolic,
but this can fail for general $G$.

When $G$ is relatively hyperbolic, quasiconvexity is a stronger property
than relative quasiconvexity.
\end{defn}

We refer the reader to Section~\ref{sub:Saturation}
for the notion of saturation, denoted $\Sat$.

\begin{prop}[\cite{HruskaWise09}, Prop 6.5]
\label{prop:SaturationRelQC}
Let $(G,\mathbb{P})$ be relatively hyperbolic.
For each $\tau>0$ there is a constant $\lambda=\lambda(\tau)>0$ such that
the following holds.
Let $\gamma$ and $\gamma'$ be geodesics in
$\Cayley(G,\Set{S} \cup \Set{P})$,
such that the endpoints $u_0$ and $u_1$ of $\gamma$ lie within
an $\Set{S}$--distance $\tau$
of\/ $\Sat_\tau\bigl(\textup{Vert}(\gamma')\bigr)$.
Then each vertex $v$ of $\gamma$
lies within an $\Set{S}$--distance $\lambda$
of either $u_0$, $u_1$, or a vertex of $\gamma'$.
\end{prop}

\begin{lem}\label{lem:CoarseIntersection}
Let $H$ be a $\kappa$--relatively quasiconvex subgroup of $(G,\P)$
and let $P \in \P$ be a peripheral
subgroup.
Assume $G$ is finitely generated, and fix a word metric on $G$.
Suppose cosets $gP$ and $aH$
have infinite coarse intersection, meaning that
$\nbd{gP}{r} \cap aH$ has infinite diameter for some $r<\infty$.
Then $\nbd{gP}{s} \cap aH$ has infinite diameter,
where $s$ is a constant depending on $H$ and $P$
but not depending on $r,g$, and $a$.
\end{lem}

\begin{proof}
Choose a cusp uniform action of $(G,\P)$ on a $\delta$--hyperbolic space
$(\widetriangle{X},\dist)$,
with truncated space $\widetriangle{X}-{\mathcal{B}}$, and basepoint $x \in \widetriangle{X}-{\mathcal{B}}$.
Although we claim that $k$ is independent using the
word metric, we shall actually work with a
(pseudo)metric coming from $\widetriangle{X}$.
Note that the result holds for any proper, left-invariant
metric on $G$
if and only if it holds for a particular such metric
(although the exact value of $k$ will typically change).
For convenience, we will identify $G$ with the orbit $Gx$
and use the induced metric from $\widetriangle{X}$.

The maximal parabolic group $gPg^{-1}$
stabilizes $gP(x)$ and has a fixed point $\xi \in \boundary \widetriangle{X}$.
Let $B$ be the horoball component of ${\mathcal{B}}$ centered at $\xi$,
and let $h \colon \widetriangle{X} \to \R$ be an associated horofunction.
Let $D_0$ and $D_1$ be the constants given by the definition of horofunction
and horoball.
Let $S$ denote the ``horosphere'' obtained as
the topological boundary of $B$.
Without loss of generality, we can assume that $B$ and $h$ are
invariant under the action of $gPg^{-1}$.
In other words, the horofunction $h$ is constant on the orbit
$gP(x)$, and this orbit lies in some horoball $B'$ centered at $\xi$ and
containing $B$.

Choose a sequence of distinct points ${ah_i(x)}$
in $\nbd{gP(x)}{r} \cap aH(x)$.
Relative quasiconvexity of $H$, implies that a geodesic between
any two of these points lies in $\nbd{aH(x)}{\kappa} \cup {\mathcal{B}}$.
On the other hand, if $\dist \bigl( ah_i(x),ah_j(x) \bigr)$
is sufficiently large,
then a geodesic $\gamma$ joining them is forced to penetrate the horoball
$B$ by Lemma~\ref{lem:GeodesicPenetratesHoroball}.
This yields a point of the horosphere $S$ that lies within a distance
$\kappa$ of $aH(x)$.
Since $gPg^{-1}$ acts cocompactly on both $gP(x)$ and $S$, the Hausdorff
distance between $gP(x)$ and $S$ is finite, say $\eta$.
Thus $\nbd{gP}{\eta} \cap \nbd{aH(x)}{\kappa}$ is nonempty, since it has a
nontrivial intersection with $S$.
Observe that $\eta$ is a uniform constant which we could have chosen at the beginning relating the geometry of horoballs to their corresponding
peripheral cosets;
in particular $\eta$ is independent
of our choice of coset $gP$ and even independent of $H$.

The first and last points $y$ and $z$
of $\gamma \cap \bar{B}$ lie within a distance $M$
of $ah_i(x)$ and $ah_j(x)$ respectively, where $M$ is the constant
given by Lemma~\ref{lem:GeodesicPenetratesHoroball}.
We thus find that the diameter of $\nbd{gP}{\eta} \cap \nbd{aH}{\kappa}$
is infinite, as it exceeds $\dist\bigl(ah_i(x),ah_j(x)\bigr) - 2M$.
The result therefore follows with $s=\eta+\kappa$.
\end{proof}

\section{Relative cocompactness}
\label{sec:RelCocompact}

\subsection{The hyperbolic case}
In \cite{Sageev97},
Sageev proved the following result.
We note that his actual statement focuses on the case where each $H_i$
is a codimension--$1$ subgroup and he produced his walls
as described in Section~\ref{sub:WallspaceFromCondimension1}.

\begin{thm}
\label{thm:HyperbolicCocompact}
Let $H_1,\dots,H_k$ be quasiconvex
subgroups of a word-hyperbolic group $G$.
Choose an $H_i$--wall for each $i$.
Then $G$ acts cocompactly on the associated dual cube complex.
\end{thm}

As a prelude to the work in this section, let us examine Sageev's theorem in more detail.  The following lemma is an easy consequence of
Corollary~\ref{cor:FiniteDim}.

\begin{lem}\label{lem:maximalcocompactness}
Let $G$ acts on a wallspace $(X,\W)$.
Suppose there are finitely many orbits of collections
of pairwise transverse walls in $X$.
Then $G$ acts cocompactly on $C(X,\W)$.
\end{lem}

Note that the above condition implies in particular that there are only finitely many orbits of walls.
Lemma~\ref{lem:maximalcocompactness}
applies naturally when $X$ is $\delta$--hyperbolic and walls are quasiconvex.
Note that
pairwise transverse walls correspond to $D$--close cosets.
The hypothesis of Lemma~\ref{lem:maximalcocompactness} follows
from the following lemma, which is implicit in Sageev
\cite{Sageev97} where it is deduced from results
of Gitik--Mitra--Rips--Sageev \cite{GMRS98}.

\begin{lem}
\label{lem:Sageev}
Let $\mathcal{H}$ be a finite collection of
$\kappa$--quasiconvex subgroups of the
$\delta$--hyperbolic group $G$
with Cayley graph $\Gamma=\Gamma(G,S)$.
For each $D\geq 0$, there exists $M=M(D,\delta,\kappa)$ such that
the following holds.
Let $\mathcal{A}$ be any collection of left cosets $aH$ with
$a \in G$ and $H \in \mathcal{H}$ such that all pairs of cosets
$aH, a'H' \in \mathcal{A}$ are $D$--close.
Then the following intersection is nonempty:
\[
   \bigcap_{aH \in \mathcal{A}} \neb_M(aH) \ne \emptyset
\]
\end{lem}

As clarified by Niblo--Reeves \cite{NibloReeves03},
a proof of Lemma~\ref{lem:Sageev} involves two steps.
Firstly, for any collection of $n$ distinct $D$--close $\kappa$--quasiconvex
subspaces of a $\delta$--hyperbolic space there is a point $g$
that lies within a distance $M(n)$ of each of these spaces.
Secondly, there is an upper bound $N$ on the maximal number of $D$--close
left cosets $g_j H_i$.

It was this second step that motivated us to introduce and
explore the notion of bounded packing described in
Definition~\ref{defn:BoundedPacking}.

\subsection{Various hemiwallspaces induced by a peripheral subspace}
\label{subsec:VariousHemi}

\begin{defn}[Hemiwallspaces induced by a subspace]
\label{defn:InducedHemiwallspaces}
Let $(X,\V)$ be a wallspace, where $\V$ denotes the collection of halfspaces
with pairing $\iota$ arising from a collection of walls $\W$.
For a nonempty subset $P \subset X$, we list below
several natural ways of inducing a hemiwallspace,
the first of which was already seen in Example~\ref{exmp:SubspaceHemi}.
The other ways presume that $X$ is a metric space,
and $P$ is an infinite diameter subspace.
Indeed, if $\diam(P) <\infty$ then (3), (4), and (5)
do not lead to a hemiwallspace, as there is a violation of
the requirement that at least one halfspace of each
wall be included.
\begin{enumerate}
\item $\U_0=\set{U\in \V}{U\cap P\neq \emptyset}$.
\item $\U_r = \bigset{U\in \V}{U\cap \neb_r(P) \neq \emptyset }$.
\item $\U_\infty=\bigset{U\in \V}{\diam(U\cap P)=\infty }$.
\item $\U_*=\bigset{U\in \V}{\diam\bigl(U\cap \neb_r(P)\bigr)
   =\infty \text{ for some $r>0$} }$.
\item $\U_{r*} = \bigset{U\in \V}{\diam\bigl(U\cap \neb_r(P)\bigr)
   =\infty }$.
\end{enumerate}
Here and below we shall use the symbol $\U_{_{\heartsuit}}$ for a chosen hemiwallspace from the above list.
Several of these hemiwallspaces are illustrated in
Figure~\ref{fig:InducedHemiWallspace}.
In each case it is easily verified that
for each halfspace in $\V$, either it, its complement, or both lie in
$\U_{_\heartsuit}$.
Note that if a group $J$ acts on $(X,\V)$ and stabilizes $P$,
then $J$ acts on $\U_{_\heartsuit}$ as well.

\begin{figure}
\begin{center}
\includegraphics[width=.9\textwidth]{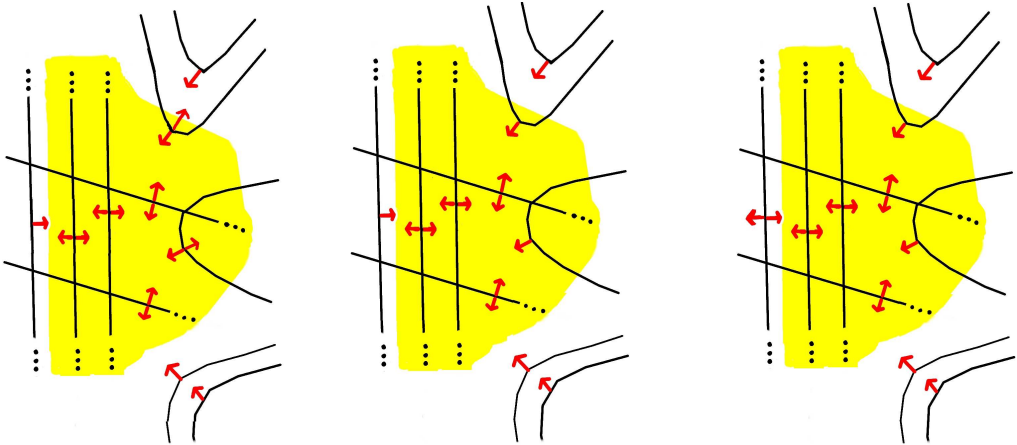}
\end{center}
\caption{A wallspace $(X,\V)$ and a subset $Y$.
The left-hand side shows the halfspaces of $\U_0$.
In the center, we see the halfspaces of $\U_\infty$.
The right-hand side illustrates $\U_*$.
}
\label{fig:InducedHemiWallspace}
\end{figure}

$\U_r$ is similar to $\U_0$ but
will be appropriate for applications, since sometimes
$P$ needs to be ``thickened up'' to capture the properties of
its interaction with walls.
$\U_*$ is a compromise that
is appropriate for an application towards
relatively hyperbolic groups we shall give later.

When $\neb_r(P)$ is a finite set, the axioms imply that each $\U_r$
has finitely many independent walls, and hence $C(\U_r)$ is finite.
As mentioned above, we do not consider $\U_\infty$ in this case,
as it contains \emph{no} halfspaces
and is not a hemiwallspace.
However, when $\diam(P)=\infty$ each of
$\U_\infty$ and $\U_*$ and $\U_{r*}$ is a hemiwallspace.
Indeed, just as either $U\cap P$ or $(X-U)\cap P$ is nonempty,
when $\diam(P)=\infty$ either $U\cap P$ or $(X-U) \cap P$ has infinite
diameter.
\end{defn}

\begin{rem}
\label{rem:HemiVsSubwallspace}
Another natural way to associate a wallspace to a subspace $P \subset X$,
where $X$ is itself a wallspace, is to consider the collection of all
partitions of $P$ that arise from partitions of $X$.
Namely, for each wall $(U,X-U)$ we obtain a wall $(U\cap P, P-U)$.
This is the induced \emph{subwallspace} of
Definition~\ref{def:Subwallspace}.
Subwallspaces and hemiwallspaces play different roles.
We will return to subwallspaces in Section~\ref{sec:FinitenessRH}.

There are several subtleties.
A minor issue is that several different partitions of $X$ might
induce the same partition of $P$.
In particular, many partitions of $X$ could induce the vacuous
partition $\{P,\emptyset\}$ of $P$.
There are ways to disregard this nuisance.
A more serious issue is that we would like to have a natural way of
embedding the cube complex dual to the $P$ wallspace in the
cube complex dual to the $X$ wallspace.
This approach would require some additional work to
suggest such an embedding.
The most important pitfall to this approach is that walls that are transverse
in $X$ might induce walls that are not transverse in $P$.
If one were to use the subwallspace approach and aim for
an embedding of the dual cube complexes, then one would have to
declare walls of $P$ to cross by decree.
It is for this reason that we have adopted the approach of taking a hemiwallspace, consisting of a subcollection
of halfspaces as it appears to naturally avoid all of these difficulties.
\end{rem}

\subsection{Structure of cubulation with peripheral subspaces}

\renewcommand{\P}{\mathcal P}

We will now examine the structure of a cube complex dual to a wallspace
$X$ that has a system $\P$ of peripheral subspaces.
The strategy is to understand the structure of the cube complex
$C(X)$ dual to $(X,\V)$ by recognizing that
important parts of it are convex subcomplexes $C(P)$ where $P \in\P$.
We call elements of $\P$ \emph{peripheral subspaces} or \emph{peripheries}
because of our later relatively hyperbolic applications,
but the initial results do not employ relative hyperbolicity.

Technically speaking we should use the notation $C(X,\V)$ for the cube complex dual to $(X,\V)$ and
$C\bigl(X,\U_{_\heartsuit}(P)\bigr)$ for the cube complex dual to the hemiwallspace
induced by some $P\in\P$.
We have elected to adopt the notation $C(X)$ for the former and
$C_{_\heartsuit}(P)$
for the latter,
as it highlights the method of inducing hemiwallspaces.

Recall that each $n$--cube in $C(X)$ corresponds to a collection of unpaired halfspaces,
together with $n$ walls, which are paired halfspaces,
such that any orientation of these walls yields a $0$--cube,
as explained in Section~\ref{subsec:Cubes}.

Having chosen a way of inducing a hemiwallspace $\U_{_\heartsuit}$
from a subspace $P$,
we say a cube $c$ is \emph{represented} in a peripheral subspace $P$
if all these halfspaces lie in $\U_{_\heartsuit}$.
Technically $c$ is a ``subhemiwallspace''---yikes!

The \emph{depth} of a cube $c \in C(X)$ is its distance to a nearest
canonical cube.

We say the peripheries in $\P$ are \emph{isolated} with respect to $\V$
if there exists $m$ such that each
cube of depth~$\geq m$ is represented in a unique periphery.

It is an interesting problem to find
geometric hypotheses on $\P$ that imply isolation.
Our main results in the relatively hyperbolic setting
depend upon the fundamental point that the coarse geometric isolation
of the peripheral subspaces $\P$
implies their isolation with respect to $\V$.

The following tautology sets a framework for our main results.
In its strongest form, it concludes that $G$ acts
relatively cocompactly on $C(X)$.

\begin{prop}
\label{prop:Tautology}
Let $(X,\V)$ be a wallspace with a collection $\P$
of peripheral subspaces.
Suppose $G$ acts on $X$ stabilizing $\V$ and $\P$.
Let $\mathbb{P}$
be a collection of representatives of the $G$--orbits
of peripheries in $\P$.
Let $\mathcal{U}_{_\heartsuit}$ be one of the five
ways of inducing hemiwallspaces described in
Definition~\ref{defn:InducedHemiwallspaces}.

Suppose there exists $m$ such that each \textup{(}maximal\textup{)}
cube of depth~$\geq m$ is represented in some periphery.
Then $C(X)=GK\cup \bigcup_{P \in \mathbb{P}} GC_{_\heartsuit}(P)$ where
$K$ is a collection of depth~$\leq m$ cubes.

Suppose the family of all peripheries $gP$ is isolated
with respect to $\V$.
Then $g C_{_\heartsuit}(P) \cap g'C_{_\heartsuit}(P') \subset GK$ unless
$gP=g'P'$.
\end{prop}

\begin{proof}
Let $K$ denote a collection of cubes representing all the $G$--orbits
of cubes of depth $\leq m$.
By hypothesis, each cube $c$ either has depth~$\leq m$ so lies in
$gK$ for some $g\in G$, or $c$ has depth~$>m$ and so lies in
$gC_{_\heartsuit}(P)$ for some $g \in G$ and $P \in \mathbb{P}$.

If depth~$>m$ cubes are represented in unique peripheries,
then each intersection
$gC_{_\heartsuit}(P)\cap g'C_{_\heartsuit}(P')$ consists of cubes of depth~$\leq m$
and hence $gC_{_\heartsuit}(P)\cap g'C_{_\heartsuit}(P')\subset GK$,
unless $gP =g'P'$.
\end{proof}

\subsection{Application towards relatively hyperbolic groups}
\label{sec:RelHypApplication}

The following result is obtained in \cite{HruskaWise09}.

\begin{thm}[Thm~8.9, \cite{HruskaWise09}]
\label{thm:RelPacking}
Let $(G,\mathbb{P})$ be relatively hyperbolic,
choose a finite generating set $\Set{S}$ for $G$,
and let $\Set{H}$ be a finite collection of
$\kappa$--relatively quasiconvex subgroups of $G$.
For each constant $D$, there are constants $R$ and $M$ such that
the following holds.
Let $\Set{A}$ be any set of left cosets $aH$
with $a\in G$ and $H \in \Set{H}$ such that
for all $aH,a'H' \in \Set{A}$ we have $\dist_{\Set{S}}(aH,a'H') <D$.
Suppose the following intersection is empty:
\begin{equation*}
\label{eqn:Empty}
   \bigcap_{aH \in \Set{A}} \nbd{aH}{M} = \emptyset
\end{equation*}
Then we have the following:
\begin{enumerate}
\item \label{item:RP:Unique}
There is a unique peripheral coset $bP$ such that for all $aH \in \Set{A}$
the intersection $\nbd{bP}{R} \cap \nbd{aH}{R}$ is nonempty.
\item \label{item:RP:Triple}
$\nbd{bP}{R} \cap \nbd{aH}{R} \cap \nbd{a'H'}{R}$ is nonempty
for all $aH,a'H' \in \Set{A}$.
\item \label{item:RP:UnboundedCosets}
$\nbd{bP}{R} \cap \nbd{aH}{R}$ is unbounded
for all $aH \in \Set{A}$.
\item \label{item:RP:UnboundedConjugates}
$bPb^{-1} \cap aHa^{-1}$
is infinite for all $aH \in \Set{A}$.
\end{enumerate}
\end{thm}

The next result follows from the proof of \cite[Prop~7.8]{HruskaWise09}.
Specifically, for the transition point $z$ that was proven to lie in 
$\bigcap_{A\in\Set{A}} \nbd{A}{M_1}$,
it was also shown in the proof
that $z \in \nbd{g_0P_0}{M_1}$.

\begin{prop}[Prop~7.8, \cite{HruskaWise09}]
\label{prop:NearTwoPeripherals}
Let $(G,\mathbb{P})$ be relatively hyperbolic with a finite
generating set $\Set{S}$.
For each $M>0$ and $\kappa>0$
there exists $M_1=M_1(M,\kappa)$ such that the following holds.
Suppose $\Set{A}$ is a collection of $\kappa$--relatively quasiconvex
subsets and $g_0P_0\ne g_1P_1$ are distinct peripheral cosets.
If $\dist_{\Set{S}}(g_iP_i,A)<M$ for all $i\in \{0,1\}$
and for all $A \in \Set{A}$,
then the following intersection is nonempty:
\[
   \nbd{g_0P_0}{M_1} \cap \bigcap_{A\in\Set{A}} \nbd{A}{M_1}.
\]
\end{prop}

Let $G$ be a finitely generated group with Cayley graph $\Gamma$.
The \emph{frontier} of a subset $S \subseteq G$ is the set of vertices
of $S$ that are adjacent to a vertex of $G - S$ in $\Gamma$.
When $S$ is an $H$--invariant subset of a subgroup $H$,
we say the frontier of $S$ is \emph{$H$--finite}
if it consists of finitely many $H$--orbits.
While the frontier of $S$ obviously depends on $\Gamma$,
its $H$--finiteness does not.
Indeed let $\Gamma,\Gamma'$ be two Cayley graphs of $G$ with respect
to finite generating sets $A,A'$.
An edge in $\Gamma$ between $S$ and $G-S$
yields a finite path in $\Gamma'$ whose length is bounded above by
a constant $D$.
Some vertex on that path must be in the frontier of $S$ with respect
to $\Gamma'$.
Therefore every vertex $v$ in the frontier with respect to $\Gamma$
lies within a distance $D$ in $\Gamma'$ of a vertex $v'$ in the frontier
with respect to $\Gamma'$.

\begin{lem}
\label{lem:HFinite}
Let $(G,\mathbb{P})$ be relatively hyperbolic with a finite generating set
$\mathcal{S}$, and let $H$
be a subgroup that is relatively quasiconvex and finitely generated.
Let $U$ be an $H$--invariant subset
whose frontier $F$ is $H$--finite.
Then $U$ is relatively quasiconvex.
\end{lem}

\begin{proof}
Finite neighborhoods in $G$ of relatively quasiconvex
subspaces are relatively quasiconvex.
Therefore the frontier $F$ of $U$ is relatively quasiconvex.
But any set with a relatively quasiconvex frontier
is itself relatively quasiconvex.
Indeed, it suffices to consider the case when
$\mathcal{S}$ is a finite generating set for $G$
that contains generating sets for each $P \in \mathbb{P}$.
Let $\mathcal{P} = \bigcup_{P\in \mathbb{P}} P-\{e\}$.
Let $\gamma$ be a geodesic in $\Cayley(G,\mathcal{S} \cup \mathcal{P})$
whose endpoints $u_0,u_1$
lie in $U$ but whose internal vertices are not in $U$.
If the first edge of $\gamma$ is labeled by an element of $\mathcal{S}$
then $u_0=f_0$ for some $f_0 \in F$.
If the first edge $e$ of $\gamma$ is labeled by
an element $p_0 \in P_0$ for some $P_0 \in \mathbb{P}$
then $e$ connects $u_0$ with an element $u_0p_0 \in G-U$ .
In this case, the coset $u_0P_0$ intersects $F$ since it intersects
both $U$ and $G-U$; i.e., $u_0\in f_0P_0$ for some $f_0\in F$.
Similarly either $u_1=f_1$ or $u_1\in f_1P_1$ for some $f_1 \in F$ and
$P_1 \in \P$.

Let $\gamma'$ be the geodesic in $\Cayley(G,\mathcal{S} \cup \mathcal{P})$
from $f_0$ to $f_1$.
By Proposition~\ref{prop:SaturationRelQC}
each vertex of $\gamma$ is within an $\mathcal{S}$--distance $\lambda$
of $\{u_0,u_1\} \cup \textup{Vert}(\gamma')$.
Since $F$ is relatively quasiconvex, each vertex of $\gamma'$
is within an $\mathcal{S}$--distance $\kappa$ of $F$, which is a
subset of $U$.
Thus each vertex of $\gamma$ is within an $\mathcal{S}$--distance
$\lambda+\kappa$ of $U$.
\end{proof}

We now state and prove the main result of this paper.

\begin{thm} [Relative cocompactness for relatively hyperbolic groups]
\label{thm:MainResult}
Let $(G,\mathbb{P})$ be relatively hyperbolic.
Let $H_1,\dots,H_r$ be a collection of relatively quasiconvex
finitely generated subgroups.
Suppose an $H_i$--wall $\{U_i,V_i\}$
is chosen for each $i$.
Let $(G,\V)$ be the associated wallspace.

Let $\P:=\bigset{gP}{g\in G, P \in \mathbb{P}}$.
Let $C=C(G)$ be the cube complex dual to $(G,\V)$.  Then we have the following
conclusions:
\begin{enumerate}
\item There exists a compact $K_*$ such that
$C = GK_* \cup \bigcup_{P} G C_*(P)$
and $gC_*(P) \cap g'C_*(P')$ is contained in $GK_*$ unless $gP=g'P'$.
\item There exists $r_0$ such that for any $r\ge r_0$
there exists a compact $K_r$ such that
$C = GK_r \cup \bigcup_{P} G C_r(P)$
and $gC_r(P) \cap g'C_r(P')$ is contained in $GK_r$ unless $gP=g'P'$.
\item There exists $r_0$ such that for any $r\ge r_0$
there exists a compact $K_{r*}$ such that
$C = GK_{r*} \cup \bigcup_{P} G C_{r*}(P)$
and $gC_{r*}(P) \cap g'C_{r*}(P')$ is contained in $GK_{r*}$ unless $gP=g'P'$.
\end{enumerate}
\end{thm}

Note that the three conclusions of the theorem
are nearly identical except for the method of
inducing hemiwallspaces from peripheral subgroups.
The key point enabling these conclusions is that the resulting
peripheral cubulations should be large enough to nearly cover all of $C$.  Shifting the method of inducing hemiwallspaces only really affects
the size of the compact subcomplex $K$ that is necessary for isolation
and covering the rest of $C$.
Recall our convention that $\P$ is a collection of \emph{infinite}
subgroups.

\begin{rem}
\label{rem:Interpretation}
A special case of relative cocompactness is
that $C(G)$ is ``relatively finite dimensional''.
In other words, if each $C_{_\heartsuit}(P)$ is finite dimensional
then so is $C(G)$.
\end{rem}

\begin{proof}[Proof of Theorem~\ref{thm:MainResult}]
Let $\Set{H} = (H_1,\dots,H_r)$,
let $\bar{\Set{H}} = (\bar{H}_1,\dots,\bar{H}_r)$,
and note that we are viewing this as an ordered list,
which could have duplicates.

Consider the map from walls to cosets defined by
$\{gU_i,gV_i\}_i \mapsto g \bar{H}_i$.
This is well-defined since $\{ g \bar{h}_i U_i, g \bar{h}_i V_i\}_i$
maps to $g \bar{h}_i \bar{H}_i = g \bar{H}_i$.
If we view these as ``indexed cosets'' then this map is actually a bijection.
We shall now forget the indexing, and view two indexed cosets as equal
if they are identical subsets of $G$.

The map from walls $gW_i$ to associated cosets $g\bar{H}_i$ is finite-to-one.
Indeed if $g\bar{H}_i = g'\bar{H}_j$, then
$\bar{H}_i=\bar{H}_j$.
The finiteness of $\bar{\Set{H}}$ gives an upper bound $r$ on the number of
walls mapping to a given coset.
We emphasize that commensurable subgroups $H_i$ and $H_j$
could lead to walls with the same underlying halfspace pairs,
depending upon our choices.
This is the source of the noninjectivity of the map from walls
to cosets.

By Lemma~\ref{lem:TransverseToCloseSubgroups}
there is a constant $D$
such that the following holds:
If $W_i=\{U_i,V_i\}$ and $W_j = \{U_j,V_j\}$ and
the walls $gW_i$ and $g'W_j$ are transverse, then
$\dist(g\bar{H}_i,g'\bar{H}_j) < D$.
Henceforth we use the notation $g \bar{H}$ to indicate a coset,
and drop the subscripts which no longer play an indexing role.

A subtle aspect of the argument is that $C$ could be infinite
dimensional.
Let $c$ be a cube of $C$ and let $d$ be a $0$--cube of $c$.
By Zorn's Lemma the collection of hyperplanes dual to $1$--cubes of $c$
extends to a maximal collection of pairwise crossing
hyperplanes dual to $1$--cubes at $d$.
This maximal collection is finite when $c$ lies in a maximal cube
of $C$.

In the first part of the argument we examine the maximal
cubes of $C$.  Some of the maximal cubes will be contained in $GK$;
others will be contained in some $gC_{_\heartsuit}(P)$.
Each cube that is not contained in a maximal cube will be shown
to lie in some $gC_{_\heartsuit}(P)$ as well.

By Proposition~\ref{prop:MaximalCorrespondence},
maximal cubes $c$ in $C$ are in one-to-one correspondence with
finite cardinality
maximal collections of pairwise transverse walls.
We have already indicated a finite-to-one map from collections of
walls to collections of cosets $g\bar{H}$ with $g\in G$ and
$\bar{H} \in \bar{\Set{H}}$.
Moreover a collection of pairwise transverse walls leads to a
collection $\Set{A}$ of pairwise $D$--close cosets.
The dimension of $c$ is equal to the cardinality
of $\Set{A}$.

Each $\bar{H}$ is relatively quasiconvex since its
finite index subgroup $H$ is relatively quasiconvex.
Let $R$ and $M$ be the constants of Theorem~\ref{thm:RelPacking}.
Suppose there exists
\[
   x \in \bigcap_{a\bar{H} \in \Set{A}} \nbd{a\bar{H}}{M}.
\]
Translating $c$ by the element $x^{-1}$ gives a cube $x^{-1}c$
whose corresponding cosets $x^{-1}a\bar{H}$ all intersect the ball of radius
$M$ centered at the identity.
There are only finitely many such cosets counted with multiplicities,
since each coset occurs with multiplicity at most $r$
and the ball is a finite set.

On the other hand, if $\bigcap_{a\bar{H} \in \Set{A}} \nbd{a\bar{H}}{M}$ is empty,
then Theorem~\ref{thm:RelPacking} gives a (unique) peripheral
coset $bP$ such that for all $a\bar{H} \in \Set{A}$ the intersection
$\nbd{bP}{R} \cap \nbd{a\bar{H}}{R}$ is nonempty.
Furthermore, $\nbd{bP}{R} \cap \nbd{a\bar{H}}{R}$ is unbounded for all
$a\bar{H} \in \Set{A}$.

Therefore,
if $aW$ is a wall associated with $a\bar{H}$, then
both halfspaces of $aW$
have infinite intersection with the $S$--neighborhood of $bP$,
for some $S=S(R,\bar{H},W)$.
Indeed $\nbd{a\bar{H}}{R}$ is contained in an
$(S-R)$--neighborhood of each halfspace
of $aW$ for some $S$.
By maximality of the collection of pairwise transverse walls,
each dependent wall $W'$ of $c$ is not transverse with some independent
wall $aW$ of $c$.  The halfspace corresponding to the
$c$--orientation of $W'$ contains one of the halfspaces of $aW$,
which has infinite intersection with $\neb_S(bP)$.
In particular, $c$ is in $C_{r*}(bP)$ for each $r \ge S$.
Let $r_0=\max\{S(R,\bar{H},W)\}$ where the maximum is taken over the finitely
many choices for $\bar{H}$ and $W$.

Suppose now that $c$ does not lie in a maximal cube of $C$.
As mentioned above, the independent walls of $c$ extend to an
infinite cardinality maximal
collection of pairwise transverse walls dual to $1$--cubes at a $0$--cube
$d$ of $c$.
As above, this family of walls leads to a
collection $\Set{A}$ of pairwise $D$--close cosets.
Any ball of radius $M$ can intersect only finitely many such cosets.
Therefore Theorem~\ref{thm:RelPacking} gives a (unique) peripheral
coset $bP$ such that for all $a\bar{H} \in \Set{A}$ the intersection
$\nbd{bP}{R} \cap \nbd{a\bar{H}}{R}$ is nonempty.
Again we conclude that
$\nbd{bP}{R} \cap \nbd{a\bar{H}}{R}$ is unbounded for all
$a\bar{H} \in \Set{A}$.
We conclude that $c$ is in $C_{r*}(bP)$ for each $r \ge S$
as before.

Let us first verify the almost uniqueness of peripheral representations of cubes in the $C_r$ case (as well as the $C_{r*}$ case).
This almost uniqueness holds for any $r>0$.
Indeed suppose $c$ is a cube (not necessarily maximal) represented in
two peripheries.
Thinking of $c$ as a collection of halfspaces, we see that
each of its halfspaces intersects
the $r$--neighborhoods of these two peripheries.
Observe that $\ddot H$ is finitely generated and relatively quasiconvex
since $H$ is.
The halfspaces of $c$ are $\kappa$--relatively quasiconvex for some $\kappa$
by Lemma~\ref{lem:HFinite}
applied to $(U,\ddot H)$ and $(V,\ddot H)$.
Thus by Proposition~\ref{prop:NearTwoPeripherals}
there exists $g \in G$ such that each halfspace
comes within a distance $M_1=M_1(r,\kappa)$ of $g$.
As before, after translating $c$ by $g^{-1}$ we can assume that every halfspace in $c$ intersects the ball of radius $M_1$ centered at the identity.
There are only finitely many possibilities for such a cube $c$,
since the finitely many walls intersecting this ball
can be oriented in only finitely many ways,
and the remaining walls must all be oriented towards the identity.
Thus the original cube $c$ lies in one of finitely many orbits.

To verify almost uniqueness in the $C_*$ case we observe that by
Lemma~\ref{lem:CoarseIntersection}
there exists $s$ such that
whenever $g\bar{H}$, $g'P$ have infinite coarse intersection
then the distance between $g\bar{H}$ and $g'P$ is at most $s$.
It follows that each $C_*(gP)$ is contained in $C_s(gP)$
and therefore almost uniqueness follows from the above argument.
\end{proof}

\subsection{The structure of the dual when $G$ acts
properly and cocompactly on $X$}
\label{subsec:ProperCocompact}

The previous subsection focused on the ``pure case'' where
$G$ is itself a wallspace.
The main result of this subsection is
Theorem~\ref{thm:RelCocompactWallspace}, which
applies Theorem~\ref{thm:MainResult}
to the more general setting of $G$ acting on a general
wallspace $(X,\W)$.
Theorem~\ref{thm:RelCocompactWallspace}
provides additional information about the structure of $C(X)$
by explaining that each $C_{_\heartsuit}(Y)$ is attached to $GK$
along $PK$.

The stabilizer $H$ of a wall $W = \{U,V\}$
\emph{acts cocompactly on $W$} if $H$ acts cocompactly
on the intersection of closed neighborhoods
$\bar\neb_r(U) \cap \bar\neb_r(V)$ for all $r$.

\begin{thm}
\label{thm:RelCocompactWallspace}
Let $(X,\W)$ be a wallspace such that $X$ is also a length space.
Suppose $G$ acts properly and cocompactly on $X$ preserving both its metric and wallspace structures,
and the action on $\W$ has only finitely many $G$--orbits of walls.
Suppose $G$ is hyperbolic relative to $\mathbb{P}$,
and each $H=\Stab(W)$ acts cocompactly on $W$ and $H$ is relatively quasiconvex.
Choose $\heartsuit \in \{r,*,r*\}$.
For each $P \in \mathbb{P}$ let $Y=Y(P) \subset X$ be a
nonempty $P$--invariant $P$--cocompact subspace.

Then there exists a compact $K$ with $GK$ connected
such that $C(X)$ has the following
structure:
\begin{enumerate}
\item $C(X) = GK \cup \bigcup_{P} G C_{_\heartsuit}(Y)$
\item $gC_{_\heartsuit}(Y) \cap g'C_{_\heartsuit}(Y')$
is contained in $GK$ unless $gP=g'P'$.
\item $C_r(Y)\cap GK = PK'$ for some compact $K'$.
In particular, $C_r(Y) \cap GK$ lies in a finite neighborhood
of $PK$ within $GK$ with respect to the path metric on
the $1$--skeleton $GK^1$.
The same conclusion holds for $C_*(Y) \cap GK$.
\end{enumerate}
\end{thm}

\begin{proof}
Our strategy will be to use
$(X,\W)$ to produce a wallspace $(G,\bar\W)$
such that there is a one-to-one correspondence between
$\W$ and $\bar\W$, and transverse walls of $\W$
map to transverse walls of $\bar\W$.

Let $\bigl\{ \{U_i,V_i\} \bigr\}$
be a set of representatives of the finitely many $G$--orbits of walls
in $\W$. Let $H_i$ be the stabilizer of $\{U_i,V_i\}$.

Let $x \in X$, and choose $D$ such that $X=\neb_D(Gx)$.
Let $\phi\colon G \to X$ be the map $g \mapsto gx$.
For each wall $W_i=\{U_i,V_i\}$ in $X$ we obtain an
$H_i$--wall
$\bar{W}_i
=\bigl\{ \bar U_i, \bar V_i \bigr\}
=\bigl\{\phi^{-1} \bigl(\neb_D(U_i)\bigr),
\phi^{-1} \bigl( \neb_D(V_i) \bigr) \bigr\}$ in $G$.
As usual
this extend equivariantly to a collection $\bar\W$ of walls in $G$.

Note that $\bar{W}$ has no walls that are genuine
partitions, with the possible exception of vacuous
walls $\{\emptyset, G\}$ arising from $\{\emptyset,X\}$.
A reader uncomfortable with duplicate walls can avoid them by
varying the ``thickening'' parameter $D$ to $D_i$
dependent on $H_i$, so that distinct $H_i$--walls
will be distinguished by varying amounts of overlaps
(i.e., betwixtness).

Transverse walls in $\W_X$ map
to transverse walls in $\W_G$.
More specifically, any pair of intersecting halfspaces of $X$
maps to a pair of intersecting halfspaces of $G$.
Indeed, if $p\in U \cap U'$ in $X$,
then there exists $g$ such that $\dist(gx,p)<D$,
so $gx \in \neb_D(U) \cap \neb_D(U')$.
Thus $g \in \bar U \cap \bar U'$.

Observe that there is a $G$--equivariant
cubical map from $C(X)$ to $C(G)$.
Indeed since intersecting halfspaces map to intersecting
halfspaces, a $0$--cube in $C(X)$ naturally maps to a
$0$--cube in $C(G)$.
Since transverse walls map to transverse walls,
this embedding extends naturally to the higher dimensional cubes.

The embedding $C(X) \to C(G)$ is an isometric embedding.
Indeed, the distance between two $0$--cubes of $C(X)$
equals the number of walls that they orient differently.
These map to $0$--cubes of $C(G)$ that orient differently
the corresponding walls (and no others).

Since $G$ acts properly, cocompactly on $X$, it is generated by a
finite set $A$.  We let $\Gamma$ denote the Cayley graph of $G$
with respect to $A$ and use the word metric on $G$.
Furthermore, there are constants $\alpha,\beta$ so that
for all $g,g' \in G$:
\[
   \frac{1}{\alpha}\,\dist_G(g,g')-\beta
   \le  \dist_X(gx,g'x) \le \alpha\, \dist_G(g,g') + \beta.
\]

For each $P \in \mathbb{P}$ let $Y=Y(P) \subset X$ be a periphery
stabilized by $P$.
For each choice of $\heartsuit \in \{r,*,r*\}$
we will make a
corresponding choice $\bar\heartsuit \in \{\bar{r},*,\bar{r}*\}$.
Each hemiwallspace $\mathcal{U}_{_{\heartsuit}}$ maps to
a hemiwallspace $\mathcal{U}_{_{\bar\heartsuit}}$.
There are consequently embeddings
$C_{_{\heartsuit}}(Y) \subseteq C_{_{\bar\heartsuit}}(P)$.

As $G$ acts cocompactly on $C(G)$
relative to the $C_{_{\bar\heartsuit}}(P)$
by Theorem~\ref{thm:MainResult},
$G$ also acts cocompactly on the subspace $C(X)$
relative to the $C_{_{\heartsuit}}(Y)$.
Indeed, specifically if
$C(G) = G\bar{K} \cup \bigcup_P G C_{_{\bar\heartsuit}}(P)$
with distinct $g C_{_{\bar\heartsuit}}(P)$'s
intersecting in $G\bar{K}$,
then $C(X) = GK \cup \bigcup_Y G C_{_{\heartsuit}}(Y)$,
where $K = C(X) \cap \bar{K}$.
Notice that the intersection of distinct $g C_{_{\heartsuit}}(Y)$'s
lies in the intersection of
the corresponding distinct $g C_{_{\bar\heartsuit}}(P)$'s,
which therefore lies in $C(X) \cap G\bar{K} = GK$.

To ensure that the peripheral hemiwallspaces of $X$ map to
peripheral hemiwallspaces of $G$,
we choose $D$ sufficiently large
that $Y=Y(P) \subset \neb_D(P x)$ for each $P \in \mathbb{P}$.
For each $r<\infty$ there exists $\bar{r}=\alpha r + \beta$
so that $\mathcal{U}_r(Y)$ maps to $\mathcal{U}_{\bar{r}}(P)$,
since if a halfspace $U$ intersects $\neb_r(Y)$
then $\neb_D(U)$ necessarily intersects $\neb_r(P x)$.
Hence $\phi^{-1}\bigl( \neb_D(U) \bigr)$ intersects
$\neb_{\alpha r+\beta}(P)$,
so $\bar{U} \cap P \ne \emptyset$.
The cases $\heartsuit =*$ and $\heartsuit=r*$ are similar.
We leave the details for the reader.

The third claim is proved separately in Lemma~\ref{lem:Footprint}.
\end{proof}

The following two lemmas are within the framework and hypotheses
of Theorem~\ref{thm:RelCocompactWallspace}.

\begin{lem}[Footprint of $C(Y)$ is $PK$]
\label{lem:Footprint}
Suppose $G$ acts cocompactly on the wallspace $X$.
Then $C_r(Y)\cap GK = PK'$ for some compact $K'$.
In particular, if $GK$ is connected,
then $C_r(Y) \cap GK$ lies in a finite neighborhood
of $PK$ within $GK$ with respect to the path metric on
the $1$--skeleton $GK^1$.
The same conclusion holds for $C_*(Y) \cap GK$.
\end{lem}

\begin{proof}
Recall that a $d$--cube $u$ corresponds to a hemiwallspace
with $d$ independent walls (that are pairwise transverse).
Since $K$ contains finitely many cubes,
there exists $B>0$ such that for each cube $u$ in $GK$
there is a point $z \in X$ such that the $B$--neighborhood
of each halfspace of $u$ contains $z$.
If $u$ also lies in $C_r(Y)$ then all halfspaces of $u$  intersect
$\neb_r(Y)$.

By Lemma~\ref{lem:PointNearY}
there exists a uniform $A$ with the following property:
for any collection of halfspaces that intersect $\neb_r(Y)$ and
also intersect $\neb_B(z)$,
there exists a point $y \in Y$ such that they all intersect
$\neb_A(y)$.
Consequently, each cube $u$
in $C_r(Y) \cap GK$
lies in some cube complex
$C_A(y) = C\bigl( \mathcal{U}_0\bigl( \neb_{A}(y) \bigr)
\bigr)$.

Choose $L$ such that for some $y_0\in Y$
the $P$--translates of the ball $\neb_L(y_0)$
cover $Y$.
Thus for each $y \in Y$ there exists $p \in P$ so that
$\neb_A(y) \subset p\neb_{L+A}(y_0)$.
Moreover, we can assume that $L$ is chosen large enough
so that $\neb_{L+A}(y_0)$ intersects every halfspace of
each cube of $K$.
Observe that $K \subseteq K' = C_{L+A}(y_0)$.

In summary, we have shown that $u$ lies in
the subcomplex $C_{A}(y)$,
which is contained in a $P$--translate of
$K' = C_{L+A}(y_0)$.
Consequently $C_r(Y) \cap GK$ lies in $PK'$,
which contains only finitely many $P$--orbits of cubes
since $K'$ is finite.

The result also holds for $C_*(Y)$ since $C_*(Y)
\subseteq C_r(Y)$ when $r$ equals the constant $s$ of
Lemma~\ref{lem:CoarseIntersection}.
\end{proof}

\begin{lem}[Nearby point]
\label{lem:PointNearY}
Let $\{V_i\}$ denote a collection of halfspaces of $X$
whose stabilizers are uniformly relatively quasiconvex,
and let $Y$ denote a peripheral subspace.
There exists $A=A(\kappa, \lambda, \mu, X)$ such that:
if each $V_i\cap \neb_\mu(Y) \neq \emptyset$ and each $V_i\cap \neb_\lambda(x)\neq \emptyset$ for some $x\in X$,
then there exists $y\in Y$ such that each $V_i\cap \neb_A(y)\neq \emptyset$.
\end{lem}

\begin{proof}
Consider the quasi-isometry $\phi\colon G \to X$
induced by $g \mapsto g x_0$ for some $x_0 \in X$.
Observe that $Y$ is coarsely equal to $\phi(P_0)$ for some
peripheral subgroup $P_0$.
Let $U_i = \phi^{-1}(V_i)$.
By Lemma~\ref{lem:HFinite} each $U_i$ is relatively quasiconvex in $G$.
Choose $g_1$ coarsely equal to $\phi^{-1}(x)$.
For consistency with the statement of
Proposition~\ref{prop:NearTwoPeripherals},
we add the trivial subgroup, denoted by $P_1$,
to the peripheral structure of $G$.
Note that $\{g_1\} = g_1P_1$.
Since nearness in $X$ provides nearness of preimages in $G$,
Proposition~\ref{prop:NearTwoPeripherals} implies that
$\{U_i\}$ are uniformly near a point in $P_0$.
Applying $\phi$, we see that $\{V_i\}$ are uniformly near a
point $y \in Y$.
\end{proof}

The following corollary to Theorem~\ref{thm:RelCocompactWallspace}
is aimed at situations like finite volume hyperbolic
lattices with sufficiently many regular hyperplanes.

\begin{cor}
Let $(G,\mathbb{P})$ be relatively hyperbolic.
Suppose $G$ acts on a proper $\delta$--hyperbolic space $\widetriangle{X}$,
each $P \in \mathbb{P}$ stabilizes a quasiconvex subspace $Y=Y(P)$, and
$\widetriangle{X} = GK \cup \bigcup GY$ where $K$ is compact and assume that
$gY \cap g'Y'$ is contained in $GK$ unless $gP = g'P'$.

Suppose $(\widetriangle{X},\W)$ is a wallspace with a $G$--action,
the halfspaces are quasiconvex, and each wall separates $GK$.
Let $Z$ be a $G$--invariant $G$--finite subset of $X$.
Assume that distinct halfspaces of $X$ have distinct intersection with $Z$.
Assume that $Z \cap U_i \cap V_j \ne \emptyset$
whenever
two walls $\{U_1,V_1\}$ and $\{U_2,V_2\}$ of $X$ are transverse
then all four ``quarterspaces'' intersect $Z$;
e.g., $U_1 \cap V_2 \cap Z \ne \emptyset$.

Then all but finitely many orbits of cubes are represented in a unique
periphery \textup{(}that is, $C_{_\heartsuit}(gY)$\textup{)}.
Consequently, $G$ acts relatively cocompactly on the dual cube complex
$C(\widetriangle{X})$.
\end{cor}

\begin{proof}
Remove the cusps of $\widetriangle{X}$ and apply
Theorem~\ref{thm:RelCocompactWallspace}.
\end{proof}

\section{Local finiteness and properness in the relatively
hyperbolic case}
\label{sec:FinitenessRH}

The goal of this section is to show that,
within the framework of Theorem~\ref{thm:RelCocompactWallspace},
additional assumptions lead to further finiteness properties of
the action of $G$ on $C(X)$.
The main result of this section is 
Theorem~\ref{thm:WallSeparationImpliesExcellent},
which can play an important role in supporting
the proof of virtual specialness of certain relatively hyperbolic groups.

In Section~\ref{sub:AdditionalProper}
we collect those properties that are consequences
of the additional assumption that $G$ acts properly on $C(X)$.
These are listed in Proposition~\ref{prop:ProperImpliesALot}.
Section~\ref{sec:RHLocalFiniteness} focuses on finiteness properties
of $C(X)$ that are consequences of hypotheses on the nature of
the walls in each periphery $Y$ of $X$.
These consequences are stated in
Theorem~\ref{thm:WallSeparationImpliesExcellent}.
Section~\ref{sub:Saturation} provides some background on the geometry
of relatively hyperbolic groups.
Section~\ref{subsec:WallWallRH} explains that the
Wall-Wall Separation property holds for $X$
provided that various analogous properties hold for each periphery $Y$.
The results of Sections
\ref{sub:Saturation}~and~\ref{subsec:WallWallRH} are used to support the
proof of Theorem~\ref{thm:WallSeparationImpliesExcellent}.

\subsection{Additional finiteness properties in the presence of a proper
action on $C(X)$}
\label{sub:AdditionalProper}

We again emphasize that we are working
under the hypothesis of Theorem~\ref{thm:RelCocompactWallspace}.
In this case
when we also know that $G$ acts properly on $C(X)$, i.e., with finite
stabilizers,
then there are numerous additional finiteness properties,
which we collect together in the following statement.
These statements are proven separately in this subsection.

\begin{prop}
\label{prop:ProperImpliesALot}
Let $G$, $X$, $C(X)$, $P$, $Y=Y(P)$, $K$, $\heartsuit$,\dots be as in
Theorem~\ref{thm:RelCocompactWallspace}.
Suppose $G$ acts properly on $C(X)$.

\begin{enumerate}
\item For each $P \in \mathbb{P}$ and corresponding $Y$,
the subcomplex $C_*(Y)$ is superconvex in $C(X)$.
\item Each $0$--cube of $GK$ lies in uniformly finitely many $gC(Y_i)$.
\item If each $P$ acts metrically properly on $C(Y)$ then $G$ acts metrically properly on $C(X)$.
\item
If each subwallspace $Y$ has linear separation
\textup{(}with respect to the subspace metric on $Y$\textup{)}, then $X$ has linear separation.
\item If each $C_{_\heartsuit}(Y)$ is locally finite
then $C(X)$ is locally finite, and $G$ acts metrically properly
on $C(X)$.
\end{enumerate}
\end{prop}

The convex subcomplex $A\subset B$ is \emph{superconvex} if
there is a uniform upper bound on the diameter $d$ of
any \emph{hanging flap} $[0,1]\times [0,d] \subset B$
that is a subcomplex not contained in $A$ but with
$\{0\}\times [0,d] \subset A$ isometrically embedded.

\begin{lem}\label{lem:superconvexity of Cstar(Y)}
For each $P \in \mathbb{P}$ with periphery $Y$
the subcomplex $C_*(Y)$ is superconvex in $C(X)$.
\end{lem}

\begin{proof}
Let $K$ be a compact subcomplex of $C(X)$
satisfying the conclusion of
Theorem~\ref{thm:RelCocompactWallspace}.
By Lemma~\ref{lem:Footprint}, there exists a compact subcomplex $K'$
such that for each of the finitely many $P \in \mathbb{P}$,
we have $C_*(Y)\cap GK \subset PK'$.
By hypothesis $G$ acts properly on $C(X)$,
so $P$ also acts properly on $C_*(Y)\cap GK$.

By almost malnormality of the peripheral subgroups, there exists
an upper bound $M$ on the diameter $P_iK' \cap gP_jK'$ for all
$P_i,P_j \in \mathcal{P}$ with $P_i \ne gP_j$.
Consequently for a hanging flap $F$ whose base lies on $C_*(Y) \cap GK$,
the intersection between $F$ and any $gC_*(Y')$ has diameter at most $M$.
Since $g_i C_*(Y_i) \cap g_j C_*(Y_j) \subset GK$,
we deduce that each sufficiently large diameter flap
$F=[0,1]\times [0,d]$ has some internal $1$--cube $[0,1]\times \{e\}$
in $GK$.

Suppose there were arbitrarily large diameter hanging flaps.
Since $P$ acts properly and cocompactly on $C_*(Y) \cap GK$,
there is a sequence of
hanging flaps whose base lines $\{0\} \times [-d,d]$
in $C_*(Y) \cap GK$ are an increasing sequence with union a line
$\{0\} \times \R$.
For each flap $F_\ell$ there exists $0<e_\ell < M$ such that
the internal $1$--cube $[0,1] \times \{e_\ell\}$ of $F_\ell$
lies in $GK$.
Passing to a subsequence, we can assume $e_\ell = e$ is independent
of $\ell$.
As $GK$ is locally finite, we can pass to a further subsequence
so that the $1$--cubes $[0,1] \times \{e\}$ of each $F_\ell$
coincide.
This $1$--cube is dual to a hyperplane $V$ of $C(X)$.

We now show $\Stab(V) \cap P$ is infinite.
Our infinite sequence of hanging flaps shows that $GK$
contains infinitely many $1$--cubes dual to $V$.
Each of these lies in the $1$--neighborhood of $PK$.
Since $GK$ is locally finite, the
$1$--neighborhood of $PK$ in $GK$ is $P$--cocompact.
Consequently infinitely many elements of $P$
map dual $1$--cubes of $V$ to dual $1$--cubes of $V$,
and thus stabilize $V$.
We conclude that $\Stab(V) \cap P$ is infinite.

Let $W$ denote the wall of $X$
corresponding to $V$.
Since $\Stab(V) = \Stab(W)$, we see that $\Stab(W) \cap P$ is infinite.
It follows
that both halfspaces of $W$ have unbounded
intersection with some $\neb_r(Px)$.
But then $W\in \mathcal U_*$ and hence $V$ intersects $C_*P$
which is a contradiction.
\end{proof}

\begin{cor}[Local finiteness of $\{C(Y_i)\}$]
\label{lem:LocFinC(Y_i)}
Each $0$--cube of $GK$ lies in uniformly finitely many $gC(Y_i)$.
\end{cor}

\begin{proof}
Since $G$ acts properly on $C(X)$ it acts properly on $GK$.
There are finitely many peripheral subgroups $P \in \mathbb{P}$,
so the collection of subcomplexes of the form $gPK'$
arising in Lemma~\ref{lem:Footprint}
is a locally finite collection.
\end{proof}

\begin{cor}[Inheriting metrical properness]
\label{cor:InheritingMetricProp}
If each $P_i$ acts metrically properly on $C(Y_i)$ then $G$ acts metrically properly on $C(X)$.
If each subwallspace $Y_i$ has linear separation
\textup{(}with respect to the subspace metric on $Y_i$\textup{)},
then $X$ has linear separation.
\end{cor}

\begin{proof}
We prove the latter statement.
First observe that if $x,x'$ are points of $X$ with corresponding
canonical cubes $d,d'$, then $\dist_{C(X)}(d,d')$ equals the number of
walls of $X$ separating $x,x'$.
Enlarge $K$ slightly to ensure that $GK$ is connected
and contains all canonical cubes.
Since $G$ acts properly, cocompactly on both $GK$ and on $X$,
we see that $X$ is quasi-isometric to $GK$ with its induced path metric.
In order to verify linear separation, it suffices to show that
$\dist_{GK}(d,d') \le \kappa\, \dist_{C(X)}(d,d')$
for some constant $\kappa$.

By hypothesis, each subwallspace $Y_i$ has linear separation.
As above if $y,y'$ are points of $Y_i$ with corresponding
canonical cubes $c,c'$, then
$\dist_{C(X)} (c,c')$ equals the number of
walls of $X$ separating $y,y'$, which also equals the number of
walls of the subwallspace $Y_i$ separating $y,y'$.
By Lemma~\ref{lem:Footprint} the footprint $GK\cap C_r(Y_i)$
consists of finitely many $P_i$--orbits of cubes.
Each peripheral subgroup $P_i$ is f.g. since $G$ is f.g.
\cite{OsinBook06}.
Enlarge $K$ further to ensure that each
$GK\cap C_r(Y_i)$ is connected.
Since $P_i$ acts properly, cocompactly on both
$GK \cap C_r(Y_i)$ and on $Y_i$,
we see that
$Y_i$ is quasi-isometric to
$GK\cap C_r(Y_i)$ with its induced path metric.
By linear separation, there is a constant $\kappa_i$ such that
the distance
in the footprint $GK\cap C_r(Y_i)$  is bounded above by
 $\kappa_i$ times the distance between $c$ and $c'$ in
 $C_r(Y_i)$ and hence in $C(X)$.

To complete the proof, consider a geodesic $\gamma$
in $C(X)$ from $d$ to $d'$.
For each subpath $\sigma$ of $\gamma$ in some $C_r(Y_i) - GK$ with endpoints in $GK$, we replace it by a path $\sigma'$
in $GK\cap C_r(Y_i)$ with
$\abs{\sigma'} \le \kappa \, \abs{\sigma}$
for some uniform $\kappa$.

To prove the metric properness, we would use
functions $f_i \colon \R_+ \to \R_+$ instead of
the linear factors $\kappa_i$.
 \end{proof}

\begin{cor}
\label{cor:FootprintImpliesNice}
If each $C_{_\heartsuit}(Y)$ is locally finite
then $C(X)$ is locally finite, and $G$ acts metrically properly
on $C(X)$.
\end{cor}

\begin{proof}
Since $G$ acts properly on $C(X)$, it also acts properly on $GK$.
Therefore $GK$ is locally finite, as it admits a proper, cocompact
group action.
If we endow $GK$ with its induced path metric,
then Lemma~\ref{lem:Footprint} implies that each finite ball in
$GK$ intersects only finitely many of the subcomplexes
$gC_{_{\heartsuit}}(Y)$.
Since each $gC_{_{\heartsuit}}(Y)$ is itself locally finite, by
hypothesis, it follows immediately that $C(X)$ is
locally finite.

Finally,
a proper action on a locally finite complex is always metrically proper.
\end{proof}

\subsection{Background on saturations
and relatively thin triangles}
\label{sub:Saturation}

In this subsection we consider a relatively hyperbolic group
$(G,\mathbb{P})$ acting properly, cocompactly on a space $X$.
For each $P_i \in \mathbb{P}$ there is an associated $P_i$--cocompact
nonempty subspace $Y_i$.
The peripheries of $X$ are the $G$--translates of the various $Y_i$.

The \emph{$J$--saturation} of a subspace $A \subseteq X$,
denoted $\Sat_J(A)$,
is the union of $\neb_J(A)$ together with all peripheries
$gY_i$ intersecting $\neb_J(A)$.

\begin{lem}
\label{lem:SaturationProps}
Let $H= \Stab(A)$ act cocompactly on $A$,
and suppose $H$ is relatively quasiconvex in $G$.
For all sufficiently large $J$, the saturation
$\Sat_J(A)$ has the following properties:
\begin{enumerate}
\item $\Sat_J(A)$ is $\kappa$--quasi-isometrically embedded
in $X$.
\item For any $\delta$
there exists $M=M(\delta)$ such that $\Sat_J(A)$ is ``isolated'' in the sense that  for any periphery $Y$, either
$\neb_\delta(Y) \cap  \Sat_J(A)$
has diameter $\leq M$,
or $Y\subset\Sat_J(A)$.
\end{enumerate}
\end{lem}

\begin{proof}
The first claim follows from the quasiconvexity of saturations
due to \Drutu--Sapir \cite{DrutuSapir2005}.
The key property of $\Sat_J(A)$ is that there exists $J'$
such that the intersection of
distinct peripheries lies within $\neb_{J'}(A)$.

The second claim holds when $J$ exceeds
the constant $s$ provided by
Lemma~\ref{lem:CoarseIntersection}
(after applying an appropriate quasi-isometry).
Indeed if $\neb_\delta(Y) \cap \Sat_J(A)$
has large diameter, then since there is a uniform bound
on the diameters of coarse intersection between
distinct peripheries,
we see that
either $Y$ equals one of the $gY'$ added to form the
saturation, or using the key property,
$\neb_\delta(Y) \cap \neb_{J'}(A)$
must have large diameter.
But a sufficiently large diameter implies an infinite
diameter, which
implies that $Y$ comes $s$--close to $A$, and hence
$Y \subset \Sat_J(A)$.
\end{proof}

The following property of quasigeodesic triangles
is a simplification of a result from
\cite[Sec~8.1.3]{DrutuSapir2005}.

\begin{thm}
\label{thm:Triangle}
Let $(G,\mathbb{P})$ be relatively hyperbolic.
For each $\epsilon$ there is a constant $\delta$
such that
if $\Delta$ is an $\epsilon$--quasigeodesic triangle
with sides $c_0$, $c_1$ and $c_2$,
then there is either:
\begin{enumerate}
  \item a point $p$ such that $\nbd{p}{\delta/2}$
  intersects each side of~$\Delta$, or
  \item a peripheral coset $gP$ such that
  $\nbd{gP}{\delta}$ intersects each side of~$\Delta$.
\end{enumerate}
In the second case, illustrated in
Figure~\ref{fig:Triangle},
each side $c_i$ of $\Delta$ has a subpath $c_i'$
that lies in $\neb_\delta(gP)$ such that
the terminal endpoint of $c_i'$
and the initial endpoint of $c'_{i+1}$
are mutually within a distance~$\delta$
\textup{(}indices modulo~$3$\textup{)}.
\end{thm}

\begin{lem}
\label{lem:bounded accumulation}
Let $(X,\W)$ be a wallspace where $X$ is a geodesic metric space.
Let $(G,\mathbb{P})$ be relatively hyperbolic.
Suppose $G$ acts properly and cocompactly on $X$,
and each $P \in \mathbb{P}$ acts cocompactly on a chosen connected
subspace $Y=Y(P)$.
Finally suppose each wall of $\W$ has
a relatively quasiconvex stabilizer.

Let $W'$ be a connected subspace with
relatively quasiconvex cocompact stabilizer.
Choose $J$ sufficiently large as in
Lemma~\ref{lem:SaturationProps}.
Let $\gamma$ be a geodesic from a point $p$
to $\Sat_J(W')$.
Let $q'$ be the first point on $\gamma$ that lies in
$\Sat_J(W')$.
There exists $U>0$ such that each wall $V$ that crosses
both $pq'$ and $W'$
must also cross $\neb_U(q')$.
\end{lem}

\begin{figure}
   \centering
   \begin{minipage}[t]{0.50\linewidth}
      \labellist \small \hair 2pt
         \pinlabel $c'_1$ at 112 64
         \pinlabel $c'_2$ at 184 64
         \pinlabel $c'_0$ at 142 7
         \pinlabel $gP$ at 146 47
      \endlabellist
      \centering
      \includegraphics[width=\linewidth]{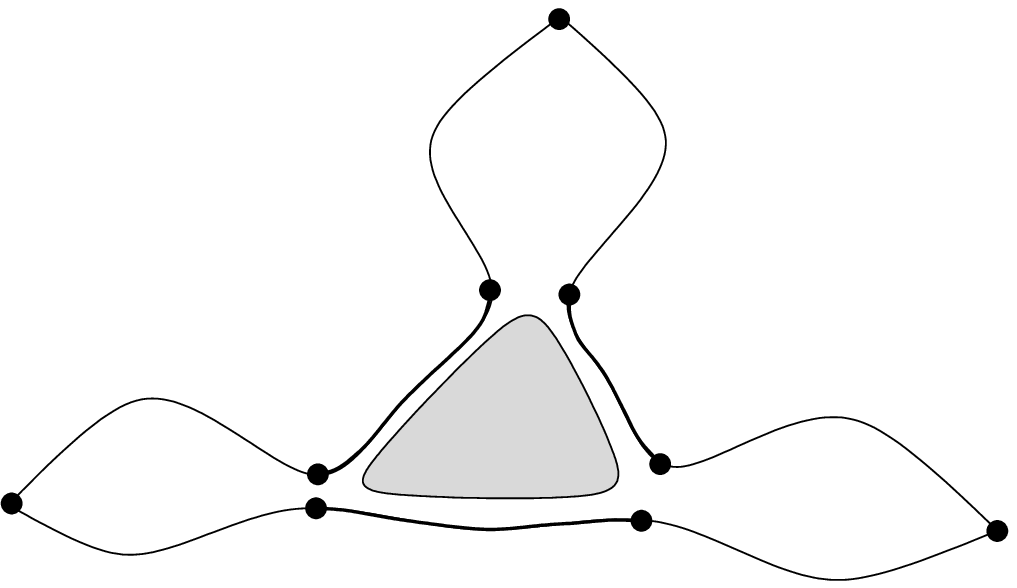}
      \captionsetup{width=.9\textwidth}
      \caption{A quasigeodesic triangle in a relatively hyperbolic group
      with a peripheral coset at the center.
      \label{fig:Triangle}}
   \end{minipage}%
   \hfill%
   \begin{minipage}[t]{0.425\linewidth}
      \labellist \small \hair 3pt
      \pinlabel \color{red}{$a$} [br] at 183 381
      \pinlabel \color{red}{$\bar{a}$} [t] at 224 262
      \pinlabel \color{red}{$b$} [tl] at 462 305
      \pinlabel \color{red}{$\bar{b}$} [t] at 375 265
      \pinlabel \color{red}{$c$} [b] at 409 515
      \pinlabel \color{red}{$\bar{c}$} [l] at 530 395
      \pinlabel $p$ [tr] at 55 255
      \pinlabel $t$ [tl] at 103 255
      \pinlabel $q'$ [t] at 430 245
      \pinlabel $s$ [br] at 535 550
      \pinlabel $V$ [tl] at 40 100
      \pinlabel $W'$ [l] at 525 55
      \endlabellist
      \centering
      \includegraphics[width=.77\linewidth]{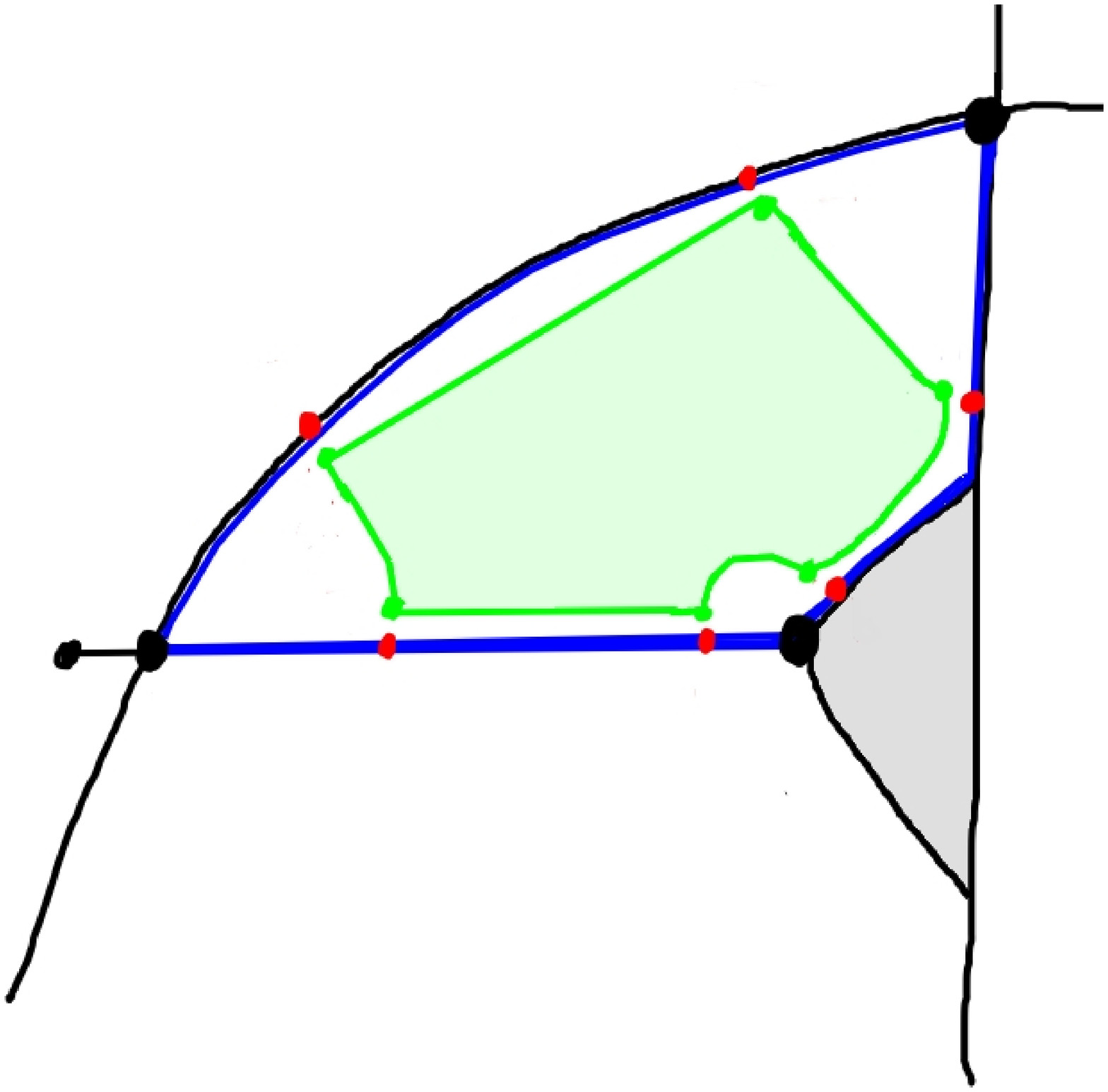}
      \captionsetup{width=\textwidth}
      \caption{The triangle $\Delta(stq')$ contains a central periphery
      $Y$, whose overlap with $\Sat_J(W')$ is approximated by
      $\dist(b,\bar{c})$. \label{fig:Accumulation}}
   \end{minipage}
\end{figure}

We will explain the proof in the special case
when $X$ is a geometric wallspace.
The general case is almost identical.
One applies the same line of reasoning to each closed
halfspace of $V$---these are also relatively quasiconvex
in an appropriate sense generalizing Definition~\ref{def:RelQC}.

We will also explain the proof in the case where
each wall $V$ of $X$ is quasi-isometrically embedded.
This happens precisely when the parabolic subgroups of
$\Stab(V)$ are quasi-isometrically embedded in the peripheral subgroups of $G$ by
\cite[Thm~10.5]{HruskaRelQC}.

To obtain the proof in general, one would have
to consider a quasigeodesic in the saturation of $V$
and observe that the points $a$ and $c$
in Figure~\ref{fig:Accumulation}
are within a bounded distance of $V$.

\begin{proof}[Proof of Lemma~\ref{lem:bounded accumulation}]
Let $s$ be a point in $V \cap W'$.
Let $t$ be a point in $V \cap \gamma$.
Consider a quasigeodesic triangle $\Delta(stq')$ whose three sides are in $V$, $\Sat_J(W')$, and $pq'$.

There exists $M$ such that $\Sat_J(W')$ is ``isolated'' in the sense that  for any periphery $Y$, either
$\neb_\delta(Y) \cap  \Sat_J(W')$ has diameter $\leq M$, or $Y$ lies entirely within  $\Sat_J(W')$.

We will prove the statement for $U=3\delta +M$.
The proof breaks into two cases according to whether the central periphery $Y$ within $\Delta(stq')$ has large or small
overlap with $\Sat_J(W')$
in the sense that $\dist(b,\bar c) \leq M$ or $\dist(b,\bar c)\geq M$.
(The small overlap case includes the situation where the central periphery consists of a singleton.)
We refer the reader to Figure~\ref{fig:Accumulation}.

If $\dist(b,\bar c)\leq M$ then
\[
   \dist(V,q') \leq
   \dist(c,q')\leq \dist(c,\bar c) + \dist(\bar c, b) + \dist(b, \bar b)
   + \dist(\bar b, q')
   \leq \delta + M + \delta + \delta \leq 3\delta +M.
\]

If $\dist(b,\bar c)\geq M$ then $Y\subset \Sat_J(W')$.
Thus $\dist(\bar a, q')\leq \dist(\bar a, Y)\leq \delta$
since $q'$ is a closest point to $\Sat_J(W')$.
We thus have:
\[
   \dist(V,q') \leq
\dist(a,q') \leq \dist(a,\bar a)+ \dist(\bar a,q') \leq \delta+\delta.
\qedhere
\]
\end{proof}

\subsection{Wall-Wall Separation and relative hyperbolicity}
\label{subsec:WallWallRH}

Let $(X,\W)$ be a wallspace where $(X,\dist)$
is a metric space.
We say $(X,\W)$
has the \emph{Wall-Wall Separation Property}
if there exists a constant $D>0$
such that any two walls $W,W'$ in $X$ with $\dist(W,W') > D$
are separated by another wall $W''$
in the sense that they lie in distinct open halfspaces of $W''$.
The Wall-Wall Separation Property was
termed ``Strong Parallel Walls'' (see \cite{NibloReeves03}).
We say $(X,\W)$ has the \emph{Ball-Ball Separation Property}
if for each $r$ there exists $m$ such that if $\dist(x_1,x_2) > m$
then $\neb_r(x_1)$ and $\neb_r(x_2)$ are separated by a wall.

Let $Y$ be a connected subspace of the wallspace $X$.
We say $Y$ has \emph{Ball-WallNbd Separation}
if for each $r>0$
there exists $s$ such that if $q \in Y$ and
$W$ is a wall of $X$
and
$\dist\bigl( \neb_r(q) \cap Y, \neb_r(W) \cap Y\bigr) > s$,
then there exists a wall $W'' \cap Y$ of $Y$
separating the sets
$\neb_r(q) \cap Y$ and $\neb_r(W) \cap Y$.

We say $Y$ has \emph{WallNbd-WallNbd Separation
\textup{[}for osculating walls\textup{]}}
if for each $r>0$
there exists $s$ such that if $W,W'$ are [osculating] walls of $X$
and
$\dist\bigl( \neb_r(W) \cap Y, \neb_r(W') \cap Y\bigr) > s$,
then there exists a wall $W'' \cap Y$ of $Y$
separating the sets
$\neb_r(W) \cap Y$ and $\neb_r(W') \cap Y$.
Note that when one or both of these sets are empty,
then any wall separates them.
Note also that all infinite intersections
arise for some sufficiently large critical value of
$r$---specifically, there exists $r_0$
such that if $\neb_r(W)\cap Y$ is infinite for some $r$ then
$\neb_{r_0}(W)\cap Y$ is infinite.

\begin{thm}
\label{thm:StrongParallelWalls}
Let $G$ act properly and
cocompactly on a wallspace $(X,\W)$.
Suppose $(G,\mathbb{P})$ is relatively hyperbolic.
Suppose for each wall $W$, the stabilizer $H$
is relatively quasiconvex.
For each $P \in \mathbb{P}$ let $Y$ be a $P$--cocompact connected subspace.
Suppose $X$ has
Ball-Ball Separation, and each $Y$ has
WallNbd-WallNbd Separation for osculating walls
and Ball-WallNbd Separation.

Then $(X,\W)$ has Wall-Wall Separation.
\end{thm}

\begin{cor}
Under the hypotheses of Theorem~\ref{thm:StrongParallelWalls},
if $Y$ has WallNbd-WallNbd Separation for arbitrary walls,
it follows that $X$ has WallNbd-WallNbd Separation for arbitrary walls.
\end{cor}

\begin{proof}
Given a wallspace $(X,\W)$, for each $t\ge 0$
there is a new wallspace $(X,\W_t)$
whose walls are $\{ \neb_t(U),\neb_t(V) \}$ for each wall $\{U,V\}$
of $\W$.
In particular, for geometric wallspaces the walls of $\W_t$
have the form $W_t=\neb_t(W)$ for each geometric wall $W \in \W$.

Observe that $X$ has Ball-Ball Separation and
$Y$ has WallNbd-WallNbd Separation and Ball-WallNbd Separation
in $(X,\W_t)$ if the corresponding properties hold in $(X,\W)$.
For instance, if
$\dist\bigl( \neb_{r+2t}(W) \cap Y, \neb_{r+2t}(W') \cap Y\bigr) > s=s(r+2t)$,
then there exists a wall $W'' \cap Y$ of $Y$ separating the sets
$\neb_{r+t}(W_t) \cap Y$ and $\neb_{r+t}(W_t') \cap Y$,
and hence $W''_t$ separates
$\neb_{r}(W_t) \cap Y$ and $\neb_{r}(W_t') \cap Y$.

We apply Theorem~\ref{thm:StrongParallelWalls}
to conclude that $(X,\W_t)$ has Wall-Wall Separation for every $t\ge 0$.
We infer that $(X,\W)$ has WallNbd-WallNbd Separation.
\end{proof}

\begin{proof}[Proof of Theorem~\ref{thm:StrongParallelWalls}]
Choose $J$ sufficiently large that
Lemma~\ref{lem:SaturationProps} holds when $A$ is any of the
finitely many types of walls of $X$.
We will show that if $W,W'$ osculate then
\[
   \dist(W,W') \le \max\bigl\{s + 2r_1,t + 2(u+2r)\bigr\}
\]
where the constants are chosen in the cases below.

\begin{figure}
   \labellist
   \small\hair 2pt
%   \pinlabel \color{blue}{$p$} [r] at 70 217
%   \pinlabel \color{blue}{$p'$} [l] at 642 215
%   \pinlabel \color{blue}{$p$} [r] at 805 215
   \pinlabel \color{blue}{$q$} [tl] at 914 203
   \pinlabel \color{blue}{$q'$} [tr] at 1235 206
%   \pinlabel \color{blue}{$p'$} [l] at 1373 211
%   \pinlabel \color{blue}{$p$} [r] at 1538 217
   \pinlabel \color{blue}{$q$} [tl] at 1652 206
   \pinlabel \color{blue}{$q'$} [tr] at 1973 211
%   \pinlabel \color{blue}{$p'$} [l] at 2108 216
   \pinlabel $W$ [t] at 55 29
   \pinlabel $W'$ [t] at 660 26
   \pinlabel $W$ [t] at 788 28
   \pinlabel $W'$ [t] at 1391 28
   \pinlabel $W$ [t] at 1516 27
   \pinlabel $W'$ [t] at 2129 30
\endlabellist
\begin{center}
\includegraphics[width=.99\textwidth]{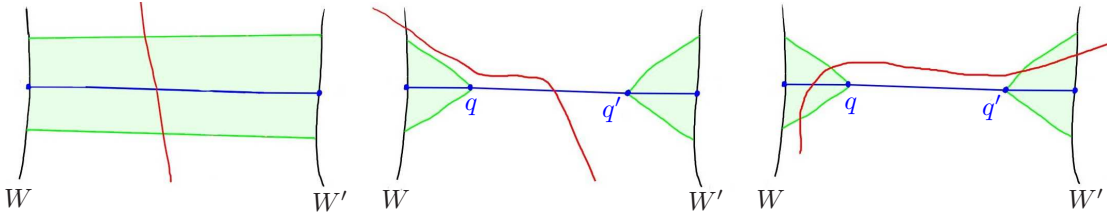}
\bigskip
\end{center}
\caption{In each case, there exists $W''$ separating $W,W'$.}
\label{fig:WallsBetweenWalls}
\end{figure}

We first consider two osculating walls $W,W'$ of $X$
with the property that $\Sat_J(W)$ and $\Sat_J(W')$
contain a common periphery $Y$.
By Lemma~\ref{lem:AAA}
we can choose $r_1$ such that if some wall $W''$ intersects $Y$,
and $W''$ crosses a wall $W$ of $X$,
and $W$ crosses the $J$--neighborhood of $Y$,
then $W''$ crosses $\neb_{r_1}(W) \cap Y$.
Moreover we shall assume that $r_1>J$ to ensure that
$\neb_{r_1}(W) \cap Y$ and $\neb_{r_1}(W') \cap Y$ are nonempty.
Let $s=s(r_1)$ be the WallNbd-WallNbd constant.
If $\dist(W,W') > s+2r_1$, then $\dist\bigl( \neb_{r_1}(W)\cap Y,
\neb_{r_1}(W')\cap Y \bigr) >s$.
By hypothesis there exists a wall $W''$ separating
$\neb_{r_1}(W) \cap Y$ and $\neb_{r_1}(W') \cap Y$,
as shown on the left of Figure~\ref{fig:WallsBetweenWalls}.
Since $W,W'$ osculate, $W''$ must cross either $W$ or $W'$,
say $W$.
But then our choice of $r_1$ guarantees that
$W''$ crosses $\neb_{r_1}(W) \cap Y$,
which is impossible.

We now consider the case where $\Sat_J(W)$ and $\Sat_J(W')$
do not contain a common periphery.
We may also assume that
$\neb_J(W) \cap \neb_J(W') = \emptyset$,
since otherwise $d(W,W')<2J$
and so the conclusion holds.
Consider
a shortest geodesic $qq'$ from $\Sat_J(W)$ to $\Sat_J(W')$.
Choose $r > \max\{J,r_1,r_2\}$, where $r_1,r_2$ are the constants
from Lemmas \ref{lem:AAA}~and~\ref{lem:AAAA}.
Let $u=u(r)$ be the Ball-WallNbd constant.

Suppose $\dist(q,W)>u+2r$.
Since $\dist(q,W) > J$, we have
$q \in Y$ for some periphery $Y \subset \Sat_J(W)$.
It follows that
$\dist \bigl( \neb_{r}(q)\cap Y, \neb_{r}(W) \cap Y \bigr) >u$
and so
Ball-WallNbd Separation in $Y$ yields a wall $W''$
that separates $\neb_{r}(q) \cap Y$ from $\neb_{r}(W) \cap Y$.
Since $W,W'$ osculate, $W''$ must cross either $W$ or $W'$.
But if $W''$ crosses $W$
then Lemma~\ref{lem:AAA} guarantees that
$W''$ crosses $\neb_{r}(W) \cap Y$,
which is impossible.
Similarly if $W''$ crosses $W'$
(as shown on the right side of Figure~\ref{fig:WallsBetweenWalls})
then Lemma~\ref{lem:AAAA} guarantees that
$W''$ crosses $\neb_r(q) \cap Y$,
which is impossible.
An analogous argument works if $\dist(q',W')>u+2r$.

Now suppose $\dist(q,W)$ and $\dist(q',W')$ are both
$\le u+2r$.
By Ball-Ball Separation, there exists $t=t(r)$ such that
whenever $\dist(q,q')>t$, there exists a wall $W''$
that separates
$\neb_{r}(q)$ and $\neb_{r}(q')$.
If
$\dist(W,W') > t + 2(u+2r)$, then $\dist(q,q')>t$.
Consequently there exists a wall $W''$
that separates
$\neb_{r}(q)$ and $\neb_{r}(q')$.
Hence $W''$ separates $W,W'$ by Lemma~\ref{lem:AAAA}.
This contradicts that $W,W'$ osculate.
\end{proof}

\begin{lem}
\label{lem:AAA}
Let $Y$ be a periphery and $W,W''$ be walls in $X$.
Suppose $W''$ crosses $W$, and $W''$ crosses $Y$,
and $\dist(W,Y)<J$.
Then $\neb_{r_1}(W) \cap W'' \cap Y$ is nonempty
for $r_1=r_1(J)$.
\end{lem}

\begin{proof}
For simplicity we assume that $X$ is a geometric wallspace.
By \cite[Prop~7.5(2)]{HruskaWise09} there exists $A=A(J)$
such that the intersection
$\neb_A(W) \cap \neb_A(W'') \cap Y$ contains a point $p$.
Let $p''$ denote a point in $W'' \cap \neb_A(p)$.
For each $q \in W'' \cap Y$
there exists a path $p''q$ from $p''$ to $q$ in
$W'' \cap \neb_B(Y)$ where $B$ is a uniform constant.
Indeed observe that
$\Stab(W'')$ is a finitely generated relatively hyperbolic
group and $\Stab(W'') \cap \Stab(Y)$ is one of its
peripheral subgroups.  There are only finitely many
conjugacy classes of such subgroups and each is finitely
generated by \cite{OsinBook06}.
Let $q$ denote a point of $W'' \cap Y$
where the path $p''q$ is within $1$ of the infimal length.
We claim that the length of $p''q$ is uniformly bounded.
Indeed there exists $s=s(A)$ so that
if $p''q$ is longer than $s$, a standard argument
using the pigeonhole principle and cocompactness
allows us to
chop out a segment and effectively translate $q$
to a point $q'' \in W'' \cap Y$ connected to $p''$
by a shorter path $p''q''$.
The result follows with $r_1=s+2A$.
\end{proof}

\begin{lem}
\label{lem:AAAA}
Assume $\neb_J(W)\cap \neb_J(W')= \emptyset$ and
$\Sat_J(W) \cap \Sat_J(W')$ does not contain a periphery.
There exists $r_2=r_2(J)$ such that the following holds:
Let $qq'$ be a shortest geodesic between
$\Sat_J(W)$ and $\Sat_J(W')$.
Any wall $W''$ crossing both saturations
must intersect both $\neb_{r_2}(q)$ and $\neb_{r_2}(q')$.
Similarly, any wall $W''$ crossing both $\Sat_J(W)$ and $qq'$
must intersect $\neb_{r_2}(q)$.
\end{lem}

\begin{proof}[Sketch]
We describe the proof in the case when walls are
$\mu$--quasi-isometrically embedded.
The union $U=\Sat_J(W) \cup qq' \cup \Sat_J(W')$
is uniformly quasiconvex by the methods of
\cite{DrutuSapir2005}.
Moreover for each $s>0$ there exists $r_2$ such that
$W,W'$ lie in distinct components
of $\neb_s(U)-\neb_{r_2}(q)$ and also of $\neb_s(U)-\neb_{r_2}(q')$.
If there exists a quasigeodesic $\gamma$ in $W''$ from
$p \in \Sat_J(W)$ to $p' \in \Sat_J(W')$, then
$\gamma \subset \neb_s(U)$ so
$\gamma$ must come $r_2$--close to both $q$ and $q'$.
Similarly if there exists a quasigeodesic $\gamma'$
in $W''$ from $p \in \Sat_J(W)$ to $p'' \in qq'$,
then $\gamma \subset \neb_s(U)$ so
$\gamma$ must come $r_2$--close to $q$
(see the middle of Figure~\ref{fig:WallsBetweenWalls}).

In general, when the walls are $\mu$--relatively quasiconvex,
one uses an argument similar to
the proof of \cite[Thm~10.5]{HruskaRelQC}.
\end{proof}

\subsection{Structure and Local finiteness of $C(X)$ in the Relatively Hyperbolic Case with Wall-Wall separation}
\label{sec:RHLocalFiniteness}

The purpose of
Theorem~\ref{thm:WallSeparationImpliesExcellent}
is to assert that $C(X)$ has very strong hyperplane
finiteness properties
under the additional assumption of appropriate peripheral
finiteness properties.
It is most desirable that $C(X)$ be cocompact,
but this often fails to hold outside of a hyperbolic setting.
In many cases having finitely many orbits of hyperplanes
is a suitable generalization of cocompactness,
and this is an immediate consequence of having finitely many
orbits of walls in $X$.
The conclusion of
Theorem~\ref{thm:WallSeparationImpliesExcellent}
ensures that there are finitely many orbits of ``ways''
in which hyperplanes of $C(X)$ can interact,
i.e., osculate or cross.
This finiteness property enables one to prove
virtual specialness in the presence of corresponding
double hyperplane coset separability \cite{HaglundWiseCoxeter}.

\begin{thm}
\label{thm:WallSeparationImpliesExcellent}
Let $G$ act properly, cocompactly on a geometric wallspace $(X,\W)$.
Suppose $(G,\mathbb{P})$ is relatively hyperbolic.
Let $W_1,\dots,W_\ell$ be representatives of the finitely many
$G$--orbits of walls in $X$.
Let $H_j = \Stab_G(W_j)$, and assume each $H_j$ is finitely generated
and relatively quasiconvex.
For each $P \in \mathbb{P}$
let $Y=Y(P)$ be a $P$--cocompact subspace of $X$.
Suppose
\begin{enumerate}
\item
  \label{item:YHasGoodSeparation}
  Each $Y$ has WallNbd-WallNbd and Compact-WallNbd Separation.
\item
  \label{item:GActsProperly}
  $G$ acts properly on $C(X)$.
\item
  \label{item:BoundedPackingInP}
   $H_j \cap gPg^{-1}$ has bounded packing
in $gPg^{-1}$ for each $g\in G$, $P \in \mathbb{P}$, and each $j$.
\end{enumerate}

Then
\begin{enumerate}
\item
\label{item:FinManyGOrbits}
\label{item:CloseEdgesGiveCloseWalls}
There are finitely many $G$--orbits of pairs of
osculating hyperplanes and also of pairs of transverse hyperplanes in
$C(X)$.
\item
\label{item:UnifLocFinCX}
$C(X)$ is uniformly locally finite.
\item
\label{item:MetricallyProperCX}
$G$ acts metrically properly on $C(X)$.
\item
\label{item:FinManyHOrbits}
For each hyperplane $L$ in $C(X)$
there are finitely many $\Stab_G(L)$--orbits
of hyperplanes $L'$
in $C(X)$ that are transverse with or osculate with $L$.
\item
\label{item:FinManyPOrbits}
For each hyperplane $M$ in $C_{_\heartsuit}(Y)$
there are finitely many $\Stab_P(M)$--orbits
of hyperplanes $M'$
in $C_{_\heartsuit}(Y)$ that are transverse with or osculate with $M$.
Here $\heartsuit \in \{r,*,r* \}$.
\end{enumerate}
\end{thm}

There are two natural settings in which
Theorem~\ref{thm:WallSeparationImpliesExcellent}
applies: the first is where $X$ is a truncated finite volume
hyperbolic $3$--manifold; the second is where $X$ is a ``mixed''
$3$--manifold \cite{PrzytyckiWiseMixed}.

In applications, verifying the hypotheses of
Theorem~\ref{thm:WallSeparationImpliesExcellent}
requires some care, but is primarily confined to features of the
peripheral subgroups.
For instance Hypothesis~\eqref{item:BoundedPackingInP}
holds in the finite volume hyperbolic case
since all subgroups of $\Z\times \Z$ have bounded packing.
In the mixed case, the peripheral subgroups are fundamental groups of
graph manifolds, so the walls must be chosen carefully to ensure
Hypothesis~\eqref{item:BoundedPackingInP}.

We have discussed conditions that imply
Hypothesis~\eqref{item:GActsProperly}
in Section~\ref{sec:Properness}.

The consequences of Hypothesis~\eqref{item:YHasGoodSeparation}
are the focus of Section~\ref{subsec:WallWallRH}.
As a hemiwallspace on $Y$ is somewhat extrinsic and depends on the
features of the embedding $Y \to X$,
we decided to couch things in terms of
WallNbd-WallNbd and Compact-WallNbd Separation,
which depend only on interactions between various subsets of $Y$.

\begin{proof}[Proof of Theorem~\ref{thm:WallSeparationImpliesExcellent}]
We first prove Conclusion~\eqref{item:FinManyGOrbits}.
Since $G$ acts properly on $C(X)$, Lemma~\ref{lem:compact-compact}
implies that $X$ has Ball-Ball Separation.
By Theorem~\ref{thm:StrongParallelWalls},
$X$ has Wall-Wall Separation, which gives
a uniform upper bound $A$ on the distance in $X$
between any two osculating walls.
By Lemma~\ref{lem:TransverseWallsMeet}
any two transverse walls are uniformly close
(they intersect).
Therefore there are only finitely many
$G$--orbits of osculating and/or transverse pairs
of walls.

We now prove Conclusion~\eqref{item:UnifLocFinCX}.
Our hypothesis that each $H_j$ is relatively quasiconvex,
together with Hypothesis~\eqref{item:BoundedPackingInP}
allows us to conclude that each $H_j$ has bounded packing in $G$ by Theorem~\ref{thm:RelHypBoundedPacking}.
Combining this with Conclusion~\eqref{item:FinManyGOrbits},
we can apply Theorem~\ref{thm:BoundedPacking}
to conclude that $C(X)$ is uniformly locally finite.

Conclusion~\eqref{item:MetricallyProperCX} follows from
Conclusion~\eqref{item:UnifLocFinCX}
together with our hypothesis that $G$
acts properly on $C(X)$.

To see Conclusion~\eqref{item:FinManyHOrbits},
as in the proof of
Conclusion~\eqref{item:FinManyGOrbits}
there is a uniform bound $A$ on the distance between
pairs of osculating and/or transverse walls of $X$.
Since $G$ acts properly and cocompactly on $X$,
there are only finitely many $H_j$--orbits of walls
intersecting $\neb_A(W_j)$.

Let us now verify Conclusion~\eqref{item:FinManyPOrbits}.
Let $W$ denote the wall of $X$ corresponding to $M$.
Any hyperplane $M'$ of $C_{_\heartsuit}(Y)$
that is transverse with or osculates with $M$
corresponds to a wall $W'$ of $X$
that is transverse with or osculates with $W$.
By the definition of $C_{_\heartsuit}(Y)$
there exists a constant $t=t(\heartsuit)$
such that $W$ and $W'$ intersect $\neb_t(Y)$.
By \cite[Prop~7.5(2)]{HruskaWise09}
there exists $s=s(t)$ such that
$\bigl( \neb_s(W) \cap \neb_t(Y) \bigr)
\cap \neb_s(W')$ is nonempty.
The result follows since
$\Stab_P(M)$ acts cocompactly on
$\neb_s(W) \cap \neb_t(Y)$.
\end{proof}

\section{CAT(0) truncations}
\label{sec:Truncating}

The goal of this section is to explain that the cube complex
obtained within the relatively hyperbolic situation studied in
Section~\ref{sec:RelCocompact}
can often be ``truncated'' to obtain a cocompact $\CAT(0)$ space.

If $G$ acts on a space $B$,
a subset $A \subset B$ is \emph{$G$--cocompact}
provided that $A$ contains a compact set $D$ such that $A=GD$.
Note that outside of a hyperbolic situation it is not always
possible to obtain a $G$--cocompact convex subcomplex.
For instance, let $G$ be any diagonal action of an infinite cyclic group
on the square tiling of $\E^2$.
Then there is no $G$--cocompact convex subcomplex
but there is a $G$--cocompact subspace that is convex with respect to
the $\CAT(0)$ metric.
The possibility of obtaining $G$--cocompact subspaces
is poorly understood.

\begin{defn}
The $G$ action on $X$ has the \emph{convex core property}
if each compact
subset $D\subset X$ lies in a closed, convex,
$G$--cocompact subset $E\subset X$.
\end{defn}

\begin{exmp}
If $G$ is virtually $\Z^n$ and $X$ is a proper CAT(0) space
with a semi-simple $G$--action then $(G,X)$ has the convex core property.

Indeed, let $F\subset X$ be a flat such that $GF=F$ and $F$ is $G$--cocompact.
Then any closed neighborhood $\neb_r(F)$ has the property that:
\begin{enumerate}
\item $\neb_r(F)$ is a convex subspace of $X$.
\item $\neb_r(F)$ is $G$--stable and $G$--cocompact (since $X$ is proper
and $\neb_r(F)$ is cobounded).
\end{enumerate}
\end{exmp}

There are many examples of groups $G$ that act properly
and cocompactly on a $\CAT(0)$ cube complex $X$,
such that $G$ contains a subgroup $H$ that does not act properly and
cocompactly on a $\CAT(0)$ space.
For instance, $H$ might not have a quadratic isoperimetric function,
or $H$ might not have finite homotopy type.
In this case
the $H$--action on $X$ cannot have the convex core property.
Typical free subgroups of $F_2 \times \Z$ do not have the
convex core property for the action on a tree cross a line.

\begin{thm}\label{thm:convex core property}
Let $G$ be a finitely generated group acting
properly on the complete $\CAT(0)$ space $C$.
Let $\set{C_i}{i\in I}$ be a collection of closed
convex subspaces of $C$,
and let $G_i = \stabilizer(C_i)$ for each $i$.
Assume that the family of translates $\{gC_i\}$ is locally finite,
so any finite ball intersects only finitely many of
the translates.

Suppose that $C=GK \cup \bigcup_i GC_i$ for some compact $K\subset C$.
And suppose that $g_iC_i\cap g_jC_j \subset GK$ unless $C_i=C_j$
and $g_j^{-1}g_i\in G_i$.

Then $(G,C)$ has the convex core property provided that
$(G_i,C_i)$ has the convex core property for each $i$.
\end{thm}

\begin{proof}
Since $G$ is finitely generated, we can
assume that the stated compact set $K$ is large enough
that $GK$ is also connected.

Let $D$ be a compact subspace of $C$.
For each $i$, let $D_i=   C_i \cap (GK\cup GD)$,
and note that $D_i$ is $G_i$--cocompact.

Indeed if $A\subset C$ is compact then $C_i \cap GA$ is clearly
$G$--cocompact.  We will show that it is also $G_i$--cocompact.
Let $\set{g_j A}{j \in J}$ be a minimal family of $G$--translates of $A$ that covers
$C_i \cap GA$.
Consider the family of subspaces $\set{g_j^{-1} C_i}{j\in J}$.
This family is finite, since each of its elements is a subspace
intersecting the bounded set $A$.
If $g_j^{-1} C_i = g_k^{-1} C_i$ then $g_j g_k^{-1} \in G_i$.
In other words, $g_j$ and $g_k$ lie in the same right coset
$G_i g_j = G_i g_k$.
So the family $\set{g_j A}{j\in J}$ lies in finitely many $G_i$--orbits.
Let $A' = A_1 \cup \cdots \cup A_\ell$
be a union of representatives of these orbits.
Then $G_i A'$ contains $C_i \cap GA$.

By hypothesis, we can choose $E_i\subset C_i$
to be a $G_i$--invariant, $G_i$--cocompact convex subset
containing $\nbd{D_i}{1}$.
The union
\[
   E = GK\cup \bigcup_{i\in I,g\in G} gE_i
\]
is clearly $G$--cocompact.
Observe that $E$ is connected because $GK$ is connected, each $gE_i$ is connected, and each $gE_i$ intersects $GK$.
Furthermore $E$ is closed and
locally convex because $GK$ is in the interior of $E$
and each point of $E-GK$ lies in a unique $gC_i$
and hence has a convex open neighborhood in
$gC_i - GK$
not meeting any other $g'C_j$.

The convexity holds because it is a closed, connected, locally convex
subspace of a complete $\CAT(0)$ space.
\end{proof}

The following is obtained by combining
Proposition~\ref{prop:Tautology} with Theorem~\ref{thm:convex core property}.

\begin{cor}\label{cor:convex core property}
Suppose the finitely generated group $G$ acts on a wallspace $X$
with an isolated collection $\P$ of peripheries.
Suppose the $\stabilizer(P)$ action on $C(P)$ has the convex
core property for each periphery $P$.
Then the $G$ action on $C(X)$ has the convex core property.
\end{cor}

In particular as a consequence of Theorem~\ref{thm:MainResult}
we have the following:

\begin{cor}
\label{cor:AbelianTruncation}
Suppose $G$ is hyperbolic relative to virtually abelian subgroups.
Let $H_1,\dots,H_r$ be a collection of quasi-isometrically embedded
subgroups with chosen $H_i$--walls
so that $G$ acts properly on the associated wallspace.
Then $G$ acts properly and cocompactly on a $\CAT(0)$ space.
\end{cor}

We also obtain the following result, which
was claimed without proof  in \cite{WiseSmallCanCube04}:

\begin{cor}
Every finitely presented $B(6)$ group
acts properly and cocompactly on a $\CAT(0)$ space.
\end{cor}

\begin{proof}
It can be proven directly that a $B(6)$ group $G$ is
hyperbolic relative to virtually abelian subgroups.
Thus the corollary follows by combining
Corollary~\ref{cor:AbelianTruncation} with the main result of
\cite{WiseSmallCanCube04}.

Alternatively, the main result of \cite{WiseSmallCanCube04} shows that
$G$ acts properly on a $\CAT(0)$ cube complex $C = GK \cup \bigcup_i C(P_i)$
satisfying the hypothesis of Theorem~\ref{thm:convex core property}.
Moreover each $G_i = \stabilizer(P_i)$ is virtually abelian.
Consequently $C$ has a $G$--cocompact convex core $\bar{C}$.
\end{proof}

\noindent
{\bf Acknowledgement:} We are grateful to Mark Hagen,
Fr\'ed\'eric Haglund, Piotr Przytycki, Hung Tran, and Wenyuan Yang for corrections
and especially grateful to the referees for many helpful corrections that improved this paper.

%%%%%%%%%%%%%%%%%%%%%%%%%%%%%%%%%%%%%%%%%%%%%%%%%%%%%%%%%%%%%%%%%%%%%%%%
%%                  BIBLIOGRAPHY
%%%%%%%%%%%%%%%%%%%%%%%%%%%%%%%%%%%%%%%%%%%%%%%%%%%%%%%%%%%%%%%%%%%%%%%%
\bibliographystyle{alpha}
\bibliography{Axioms}

\newcommand{\etalchar}[1]{$^{#1}$}
\begin{thebibliography}{GMRS98}

\bibitem[AB06]{AlibegovicBestvina2006}
E.~Alibegovi\'{c} and M.~Bestvina.
\newblock Limit groups are $\textup{CAT}(0)$.
\newblock {\em J. London Math. Soc. \textup{(}2\textup{)}}, 74(1):259--272,
  2006.

\bibitem[AR90]{AitchisonRubinstein90}
I.R. Aitchison and J.H. Rubinstein.
\newblock An introduction to polyhedral metrics of nonpositive curvature on
  {$3$}--manifolds.
\newblock In S.K. Donaldson and C.B. Thomas, editors, {\em Geometry of
  low-dimensional manifolds: 2. {S}ymplectic manifolds and {J}ones--{W}itten
  theory \textup{(}{D}urham, 1989\textup{)}}, volume 151 of {\em London Math.
  Soc. Lecture Note Ser.}, pages 127--161. Cambridge Univ. Press, 1990.

\bibitem[BH93]{BrinkHowlett93}
B.~Brink and R.B. Howlett.
\newblock A finiteness property and an automatic structure for {C}oxeter
  groups.
\newblock {\em Math. Ann.}, 296(1):179--190, 1993.

\bibitem[BH99]{BridsonHaefliger}
M.R. Bridson and A.~Haefliger.
\newblock {\em Metric spaces of non-positive curvature}.
\newblock Springer-Verlag, Berlin, 1999.

\bibitem[BHW11]{BergeronHaglundWiseSimple}
N.~Bergeron, F.~Haglund, and D.T. Wise.
\newblock Hyperplane sections in arithmetic hyperbolic manifolds.
\newblock {\em J. Lond. Math. Soc. \textup{(}2\textup{)}}, 83(2):431--448,
  2011.

\bibitem[Bow12]{Bowditch12}
B.H. Bowditch.
\newblock Relatively hyperbolic groups.
\newblock {\em Internat. J. Algebra Comput.}, 22(3):1250016, 66, 2012.

\bibitem[BP]{BarrePichot2010}
S.~Barr{\'e} and M.~Pichot.
\newblock Removing chambers in {B}ruhat--{T}its buildings.
\newblock Preprint. arXiv:1003.4614 [math.MG].

\bibitem[BW12]{BergeronWiseBoundary}
N.~Bergeron and D.T. Wise.
\newblock A boundary criterion for cubulation.
\newblock {\em Amer. J. Math.}, 134(3):843--859, 2012.

\bibitem[CCJ{\etalchar{+}}01]{HaagerupBook01}
P.-A. Cherix, M.~Cowling, P.~Jolissaint, P.~Julg, and A.~Valette.
\newblock {\em Groups with the {H}aagerup property: {G}romov's
  a-{T}-menability}, volume 197 of {\em Progress in Mathematics}.
\newblock Birkh\"auser Verlag, Basel, 2001.

\bibitem[CN05]{ChatterjiNiblo04}
I.~Chatterji and G.~Niblo.
\newblock From wall spaces to {$\rm CAT(0)$} cube complexes.
\newblock {\em Internat. J. Algebra Comput.}, 15(5-6):875--885, 2005.

\bibitem[CP11]{CapracePrzytycki2011}
P.-E. Caprace and P.~Przytycki.
\newblock Bipolar {C}oxeter groups.
\newblock {\em J. Algebra}, 338:35--55, 2011.

\bibitem[CS11]{CapraceSageev2011}
P.-E. Caprace and M.~Sageev.
\newblock Rank rigidity for $\textup{CAT}(0)$ cube complexes.
\newblock {\em Geom. Funct. Anal.}, 21(4):851--891, 2011.

\bibitem[Dah03]{Dahmani03}
F.~Dahmani.
\newblock Combination of convergence groups.
\newblock {\em Geom. Topol.}, 7:933--963, 2003.

\bibitem[DS05]{DrutuSapir2005}
C.~Dru{\RomanianComma{t}}u and M.~Sapir.
\newblock Tree-graded spaces and asymptotic cones of groups.
\newblock {\em Topology}, 44(5):959--1058, 2005.
\newblock With an appendix by Denis Osin and Sapir.

\bibitem[Far03]{Farley2003}
D.S. Farley.
\newblock Finiteness and {$\rm CAT(0)$} properties of diagram groups.
\newblock {\em Topology}, 42(5):1065--1082, 2003.

\bibitem[Gau12]{Gautero12}
F.~Gautero.
\newblock A non-trivial example of a free-by-free group with the {H}aagerup
  property.
\newblock {\em Groups Geom. Dyn.}, 6(4):677--699, 2012.

\bibitem[Geo08]{GeogheganBook2008}
R.~Geoghegan.
\newblock {\em Topological methods in group theory}, volume 243 of {\em
  Graduate Texts in Mathematics}.
\newblock Springer, New York, 2008.

\bibitem[Ger97]{Gerasimov97}
V.N. Gerasimov.
\newblock Semi-splittings of groups and actions on cubings.
\newblock In Yu.G. Reshetnyak, L.A. Bokut{$'$}, S.K. Vodop{$'$}yanov, and I.A.
  Ta{\u\i}manov, editors, {\em Algebra, geometry, analysis and mathematical
  physics \textup{(}Russian\textup{)} \textup{(}Novosibirsk, 1996\textup{)}},
  pages 91--109. Izdat. Ross. Akad. Nauk Sib. Otd. Inst. Mat., 1997.
\newblock English translation published as ``Fixed-point-free actions on
  cubings.'' \textit{Siberian Adv. Math.}, 8(3):36--58, 1998.

\bibitem[GMRS98]{GMRS98}
R.~Gitik, M.~Mitra, E.~Rips, and M.~Sageev.
\newblock Widths of subgroups.
\newblock {\em Trans. Amer. Math. Soc.}, 350(1):321--329, 1998.

\bibitem[Gro87]{Gromov87}
M.~Gromov.
\newblock Hyperbolic groups.
\newblock In S.M. Gersten, editor, {\em Essays in group theory}, volume~8 of
  {\em Math. Sci. Res. Inst. Publ.}, pages 75--263. Springer, New York, 1987.

\bibitem[GS06]{GubaSapir2005}
V.S. Guba and M.V. Sapir.
\newblock Diagram groups and directed $2$--complexes: {H}omotopy and homology.
\newblock {\em J. Pure Appl. Algebra}, 205(1):1--47, 2006.

\bibitem[Gura]{GuralnikBoundaries}
D.~Guralnik.
\newblock Coarse decompositions of boundaries for $\textup{CAT}(0)$ groups.
\newblock Preprint. arXiv:math/0611006 [math.GR].

\bibitem[Gurb]{GuralnikLocalFiniteness}
D.~Guralnik.
\newblock Local finiteness for cubulations of $\textup{CAT}(0)$ groups.
\newblock Preprint. arXiv:math/0610950 [math.GR].

\bibitem[Hag]{HagenArboreal}
M.F. Hagen.
\newblock Weak hyperbolicity of cube complexes and quasi-arboreal groups.
\newblock To appear in \emph{J. Topology}. arXiv:1101.5191 [math.GR].

\bibitem[HP98]{HaglundPaulin98}
F.~Haglund and F.~Paulin.
\newblock Simplicit\'e de groupes d'automorphismes d'espaces \`a courbure
  n\'egative.
\newblock In I.~Rivin, C.~Rourke, and C.~Series, editors, {\em The Epstein
  birthday schrift}, volume~1 of {\em Geom. Topol. Monogr.}, pages 181--248.
  Geom. Topol., Coventry, 1998.

\bibitem[Hru10]{HruskaRelQC}
G.C. Hruska.
\newblock Relative hyperbolicity and relative quasiconvexity for countable
  groups.
\newblock {\em Algebr. Geom. Topol.}, 10(3):1807--1856, 2010.

\bibitem[HS92]{HassScott92}
J.~Hass and P.~Scott.
\newblock Homotopy equivalence and homeomorphism of {$3$}--manifolds.
\newblock {\em Topology}, 31(3):493--517, 1992.

\bibitem[HW]{HsuWiseCubulatingMalnormal}
T.~Hsu and D.T. Wise.
\newblock Cubulating malnormal amalgams.
\newblock pages 1--20.
\newblock Submitted.

\bibitem[HW08]{HaglundWiseSpecial}
F.~Haglund and D.T. Wise.
\newblock Special cube complexes.
\newblock {\em Geom. Funct. Anal.}, 17(5):1551--1620, 2008.

\bibitem[HW09]{HruskaWise09}
G.C. Hruska and D.T. Wise.
\newblock Packing subgroups in relatively hyperbolic groups.
\newblock {\em Geom. Topol.}, 13(4):1945--1988, 2009.

\bibitem[HW10a]{HaglundWiseCoxeter}
F.~Haglund and D.T. Wise.
\newblock Coxeter groups are virtually special.
\newblock {\em Adv. Math.}, 224(5):1890--1903, 2010.

\bibitem[HW10b]{HsuWiseCubulating}
T.~Hsu and D.T. Wise.
\newblock Cubulating graphs of free groups with cyclic edge groups.
\newblock {\em Amer. J. Math.}, 132(5):1153--1188, 2010.

\bibitem[JW13]{JanzenWise13}
D.~Janzen and D.T. Wise.
\newblock Cubulating rhombus groups.
\newblock {\em Groups Geom. Dyn.}, 7(2):419--442, 2013.

\bibitem[KK05]{KapovichKleiner2005}
M.~Kapovich and B.~Kleiner.
\newblock Coarse {A}lexander duality and duality groups.
\newblock {\em J. Differential Geom.}, 69(2):279--352, 2005.

\bibitem[KR89]{KrophollerRoller89}
P.H. Kropholler and M.A. Roller.
\newblock Relative ends and duality groups.
\newblock {\em J. Pure Appl. Algebra}, 61(2):197--210, 1989.

\bibitem[Lea13]{Leary_KanThurston}
I.J. Leary.
\newblock A metric {K}an-{T}hurston theorem.
\newblock {\em J. Topol.}, 6(1):251--284, 2013.

\bibitem[LW]{LauerWise07}
J.~Lauer and D.T. Wise.
\newblock Cubulating one-relator groups with torsion.
\newblock Submitted, 2011.

\bibitem[McC09]{McCammond2009}
J.P. McCammond.
\newblock Constructing non-positively curved spaces and groups.
\newblock In M.R. Bridson, P.H. Kropholler, and I.J. Leary, editors, {\em
  Geometric and cohomological methods in group theory \textup{(}Durham,
  2003\textup{)}}, volume 358 of {\em London Math. Soc. Lecture Note Ser.},
  pages 162--224. Cambridge Univ. Press, 2009.

\bibitem[Mou88]{Moussong88}
G.~Moussong.
\newblock {\em Hyperbolic {C}oxeter groups}.
\newblock PhD thesis, Ohio State University, 1988.

\bibitem[Nic04]{NicaCubulating04}
B.~Nica.
\newblock Cubulating spaces with walls.
\newblock {\em Algebr. Geom. Topol.}, 4:297--309, 2004.

\bibitem[NR97]{NibloReeves97}
G.~Niblo and L.~Reeves.
\newblock Groups acting on ${{\rm CAT}(0)}$ cube complexes.
\newblock {\em Geom. Topol.}, 1:1--7, 1997.

\bibitem[NR98a]{NibloReeves98}
G.A. Niblo and L.D. Reeves.
\newblock The geometry of cube complexes and the complexity of their
  fundamental groups.
\newblock {\em Topology}, 37(3):621--633, 1998.

\bibitem[NR98b]{NibloRoller98}
G.A. Niblo and M.A. Roller.
\newblock Groups acting on cubes and {K}azhdan's property ({T}).
\newblock {\em Proc. Amer. Math. Soc.}, 126(3):693--699, 1998.

\bibitem[NR03]{NibloReeves03}
G.A. Niblo and L.D. Reeves.
\newblock Coxeter groups act on {${\rm CAT}(0)$} cube complexes.
\newblock {\em J. Group Theory}, 6(3):399--413, 2003.

\bibitem[NS]{NevoSageevBoundary}
A.~Nevo and M.~Sageev.
\newblock The {P}oisson boundary of \textrm{CAT}$(0)$ cube complex groups.
\newblock Preprint. arXiv:1105.1675 [math.GT].

\bibitem[Osi06]{OsinBook06}
D.V. Osin.
\newblock Relatively hyperbolic groups: {I}ntrinsic geometry, algebraic
  properties, and algorithmic problems.
\newblock {\em Mem. Amer. Math. Soc.}, 179(843):vi+100, 2006.

\bibitem[OW11]{OllivierWiseDensity}
Y.~Ollivier and D.T. Wise.
\newblock Cubulating random groups at density less than {$1/6$}.
\newblock {\em Trans. Amer. Math. Soc.}, 363(9):4701--4733, 2011.

\bibitem[PW]{PrzytyckiWiseMixed}
P.~Przytycki and D.T. Wise.
\newblock Mixed $3$--manifolds are virtually special.
\newblock Submitted. arXiv:1205.6742 [math.GR].

\bibitem[Rol98]{RollerPocSets}
M.A. Roller.
\newblock Poc sets, median algebras and group actions: {A}n extended study of
  {D}unwoody's construction and {S}ageev's theorem.
\newblock Preprint, 1998.

\bibitem[RS99]{RubinsteinSageev99}
H.~Rubinstein and M.~Sageev.
\newblock Intersection patterns of essential surfaces in {$3$}--manifolds.
\newblock {\em Topology}, 38(6):1281--1291, 1999.

\bibitem[RW98]{RubinsteinWang98}
J.H. Rubinstein and S.~Wang.
\newblock $\pi\sb 1$--injective surfaces in graph manifolds.
\newblock {\em Comment. Math. Helv.}, 73(4):499--515, 1998.

\bibitem[Sag95]{Sageev95}
M.~Sageev.
\newblock Ends of group pairs and non-positively curved cube complexes.
\newblock {\em Proc. London Math. Soc. \textup{(}3\textup{)}}, 71(3):585--617,
  1995.

\bibitem[Sag97]{Sageev97}
M.~Sageev.
\newblock Codimension--$1$ subgroups and splittings of groups.
\newblock {\em J. Algebra}, 189(2):377--389, 1997.

\bibitem[Sco77]{Scott77}
P.~Scott.
\newblock Ends of pairs of groups.
\newblock {\em J. Pure Appl. Algebra}, 11(1--3):179--198, 1977.

\bibitem[Sco83]{Scott83}
P.~Scott.
\newblock There are no fake {S}eifert fibre spaces with infinite {$\pi
  \sb{1}$}.
\newblock {\em Ann. of Math. \textup{(}2\textup{)}}, 117(1):35--70, 1983.

\bibitem[Ser77]{Serre77}
J.-P. Serre.
\newblock {\em Arbres, amalgames, {${\rm SL}\sb{2}$}}, volume~46 of {\em
  Ast\'erisque}.
\newblock Soci\'et\'e Math\'ematique de France, Paris, 1977.
\newblock Written in collaboration with H. Bass.

\bibitem[SW05]{SageevWiseTits}
M.~Sageev and D.T. Wise.
\newblock The {T}its alternative for {${\rm CAT}(0)$} cubical complexes.
\newblock {\em Bull. London Math. Soc.}, 37(5):706--710, 2005.

\bibitem[Wis]{WiseIsraelHierarchy}
D.T. Wise.
\newblock The structure of groups with a quasiconvex hierarchy.
\newblock pages 1--189.
\newblock Submitted.

\bibitem[Wis04]{WiseSmallCanCube04}
D.T. Wise.
\newblock Cubulating small cancellation groups.
\newblock {\em Geom. Funct. Anal.}, 14(1):150--214, 2004.

\bibitem[Wis06]{WiseFigure8}
D.T. Wise.
\newblock Subgroup separability of the figure 8 knot group.
\newblock {\em Topology}, 45(3):421--463, 2006.

\bibitem[Wis09]{WiseStructureAnnouncement09}
D.T. Wise.
\newblock Research announcement: {T}he structure of groups with a quasiconvex
  hierarchy.
\newblock {\em Electron. Res. Announc. Amer. Math. Soc.}, 16:44--55, 2009.

\bibitem[Wis12]{WiseFreeCubulation}
D.T. Wise.
\newblock Recubulating free groups.
\newblock {\em Israel J. Math.}, 191:337--345, 2012.

\bibitem[Wri12]{WrightDimension}
N.~Wright.
\newblock Finite asymptotic dimension for $\textrm{CAT}(0)$ cube complexes.
\newblock {\em Geom. Topol.}, 16(1):527--554, 2012.

\end{thebibliography}
%\bibliography{C:/papers/wise}
\end{document}